\documentclass[smallextended,envcountsect,envcountsame]{svjour3}
\journalname{Probability Theory and Related Fields}
\usepackage{mathptmx}
\smartqed

\usepackage{amsmath,amssymb,amsfonts,mathrsfs,graphicx,color,bbm}
\usepackage{stmaryrd}
\usepackage{enumerate}
\usepackage[english]{babel}
\usepackage[bookmarksopen,colorlinks,linkcolor=black,citecolor=black]{hyperref}

\newcommand{\E}{\mathbb{E}}
\newcommand{\R}{\mathbb{R}}
\newcommand{\Q}{\mathbb{Q}}
\newcommand{\N}{\mathbb{N}}
\newcommand{\Z}{\mathbb{Z}}
\renewcommand{\P}{\mathbb{P}}

\newcommand{\cA}{\mathcal{A}}
\newcommand{\cC}{\mathcal{C}}
\newcommand{\cE}{\mathcal{E}}

\newcommand{\cI}{\mathcal{I}}

\newcommand{\condparentheses}[2]{\left(\left.#1\,\right\vert#2\right)}
\newcommand{\condparenthesesreversed}[2]{\left(#1\left\vert\,#2\right.\right)}
\newcommand{\condP}[2]{\mathbb{P}\condparentheses{#1}{#2}}

\newcommand{\interior}[1]{{\rm int}\!\left(#1\right)}
\newcommand{\closure}[1]{{\rm clo}\!\left(#1\right)}

\DeclareMathOperator{\Capa}{Cap}
\DeclareMathOperator{\Vol}{Vol}
\DeclareMathOperator{\diam}{diam}

\newcommand{\boundary}{\partial}

\newcommand{\set}[1]{\left\{#1\right\}}
\newcommand{\abs}[1]{\left\vert#1\right\vert}

\newcommand{\floor}[1]{\left\lfloor#1\right\rfloor}
\newcommand{\ceiling}[1]{\left\lceil#1\right\rceil}

\newcommand{\shortset}[1]{\{#1\}}
\newcommand{\shortabs}[1]{\vert#1\vert}

\newcommand{\ocinterval}[1]{\left(#1\right]} 
\newcommand{\cointerval}[1]{\left[#1\right)} 

\newcommand{\indicatorofset}[1]{\mathbbm{1}_{#1}}
\newcommand{\indicator}[1]{\indicatorofset{\set{#1}}}

\newcommand{\union}{\cup}
\newcommand{\bigunion}{\bigcup}
\newcommand{\intersect}{\cap}

\newcommand{\decreasesto}{\downarrow}
\newcommand{\increasesto}{\uparrow}

\renewcommand{\th}{{\rm th}}

\renewcommand{\phi}{\varphi}
\renewcommand{\emptyset}{\varnothing}

\newcommand{\blank}[1]{}

\newcommand{\T}{\mathbb{T}}

\spnewtheorem{varquestion}{Question}{\bf}{\it} 

\newcommand{\textandreference}[2]{\texorpdfstring{\hyperref[#2]{#1\ref*{#2}}}{#1\ref*{#2}}}

\newcommand{\lbsect}[1]{\label{s:#1}}
\newcommand{\refsect}[1]{\textandreference{Section~}{s:#1}}

\newcommand{\lbsubsect}[1]{\label{ss:#1}}
\newcommand{\refsubsect}[1]{\textandreference{Section~}{ss:#1}}
\newcommand{\refSubSect}[1]{\textandreference{Section~}{ss:#1}}

\newcommand{\lbsubsubsect}[1]{\label{sss:#1}}
\newcommand{\refsubsubsect}[1]{\textandreference{Section~}{sss:#1}}

\newcommand{\lbappendix}[1]{\label{app:#1}}
\newcommand{\refappendix}[1]{\textandreference{Appendix~}{app:#1}}

\newcommand{\lbthm}[1]{\label{t:#1}}
\newcommand{\refthm}[1]{\textandreference{Theorem~}{t:#1}}

\newcommand{\lbprop}[1]{\label{p:#1}}
\newcommand{\refprop}[1]{\textandreference{Proposition~}{p:#1}}

\newcommand{\lblemma}[1]{\label{l:#1}}
\newcommand{\reflemma}[1]{\textandreference{Lemma~}{l:#1}}

\newcommand{\lbcoro}[1]{\label{c:#1}}
\newcommand{\refcoro}[1]{\textandreference{Corollary~}{c:#1}}

\newcommand{\lbdefn}[1]{\label{d:#1}}
\newcommand{\refdefn}[1]{\textandreference{Definition~}{d:#1}}

\newcommand{\lbitem}[1]{\label{item:#1}}
\newcommand{\refitem}[1]{\ref{item:#1}}

\newcommand{\lbfig}[1]{\label{fig:#1}}
\newcommand{\reffig}[1]{\textandreference{Fig.~}{fig:#1}}

\newcommand{\lbquestion}[1]{\label{q:#1}}
\newcommand{\refquestion}[1]{\textandreference{Question~}{q:#1}}

\numberwithin{equation}{section}

\begin{document}

\title{Extremal geometry of a Brownian porous medium}

\author{
Jesse Goodman \and Frank den Hollander
}

\institute{J. Goodman \and F. den Hollander 
\at Mathematical Institute, Leiden University, P.O.\ Box 9512, 2300 RA Leiden, The Netherlands.\\
\email{goodmanja@math.leidenuniv.nl, denholla@math.leidenuniv.nl}}

\date{Received: date / Accepted: date}

\maketitle

\begin{abstract}
The path $W[0,t]$ of a Brownian motion on a $d$-dimensional torus $\T^d$ run for time $t$ is 
a random compact subset of $\T^d$. We study the geometric properties of the complement 
$\T^d\setminus W[0,t]$ as $t\to\infty$ for $d\geq 3$. In particular, we show that the largest 
regions in $\T^d\setminus W[0,t]$ have a linear scale $\phi_d(t)=[(d\log t)/(d-2)\kappa_d t]^{1/(d-2)}$, 
where $\kappa_d$ is the capacity of the unit ball. More specifically, we identify the sets $E$ for 
which $\T^d\setminus W[0,t]$ contains a translate of $\phi_d(t)E$, and we count the number of 
disjoint such translates. Furthermore, we derive large deviation principles for the largest inradius of
$\T^d\setminus W[0,t]$ as $t\to\infty$ and the $\epsilon$-cover time of $\T^d$ as $\epsilon 
\downarrow 0$. Our results, which generalise laws of large numbers proved by Dembo, Peres and Rosen \cite{DPR2003}, 
are based on a large deviation estimate for the shape of the component with largest capacity 
in $\T^d \setminus W_{\rho(t)}[0,t]$, where $W_{\rho(t)}[0,t]$ is the Wiener sausage of radius 
$\rho(t)$, with $\rho(t)$ chosen much smaller than $\phi_d(t)$ but not too small. The idea 
behind this choice is that $\T^d \setminus W[0,t]$ consists of ``lakes'', whose linear size is of 
order $\phi_d(t)$, connected by narrow ``channels''. We also derive large deviation principles 
for the principal Dirichlet eigenvalue and for the maximal volume of the components of $\T^d
\setminus W_{\rho(t)}[0,t]$ as $t\to\infty$. Our results give a complete picture of the extremal
geometry of $\T^d\setminus W[0,t]$ and of the optimal strategy for $W[0,t]$ to realise the extremes.   

\end{abstract}


\section{Introduction}
\lbsect{Intro}


\subsection{Five key questions}
\lbsubsect{Mot}

$\bullet$ 
Our basic object of study is the complement of a random path:

\begin{varquestion}
\lbquestion{Geometry}
Run a Brownian motion $W=(W(t))_{t \geq 0}$ on a $d$-dimensional torus $\T^d$, $d\geq 3$.  
What is the geometry of the random set $\T^d \setminus W[0,t]$ for large $t$?
\end{varquestion}

\noindent
\reffig{d=2sim} shows a simulation in $d=2$. 

\begin{figure}[htbp]
{\hfill
\includegraphics[width=0.4\textwidth]{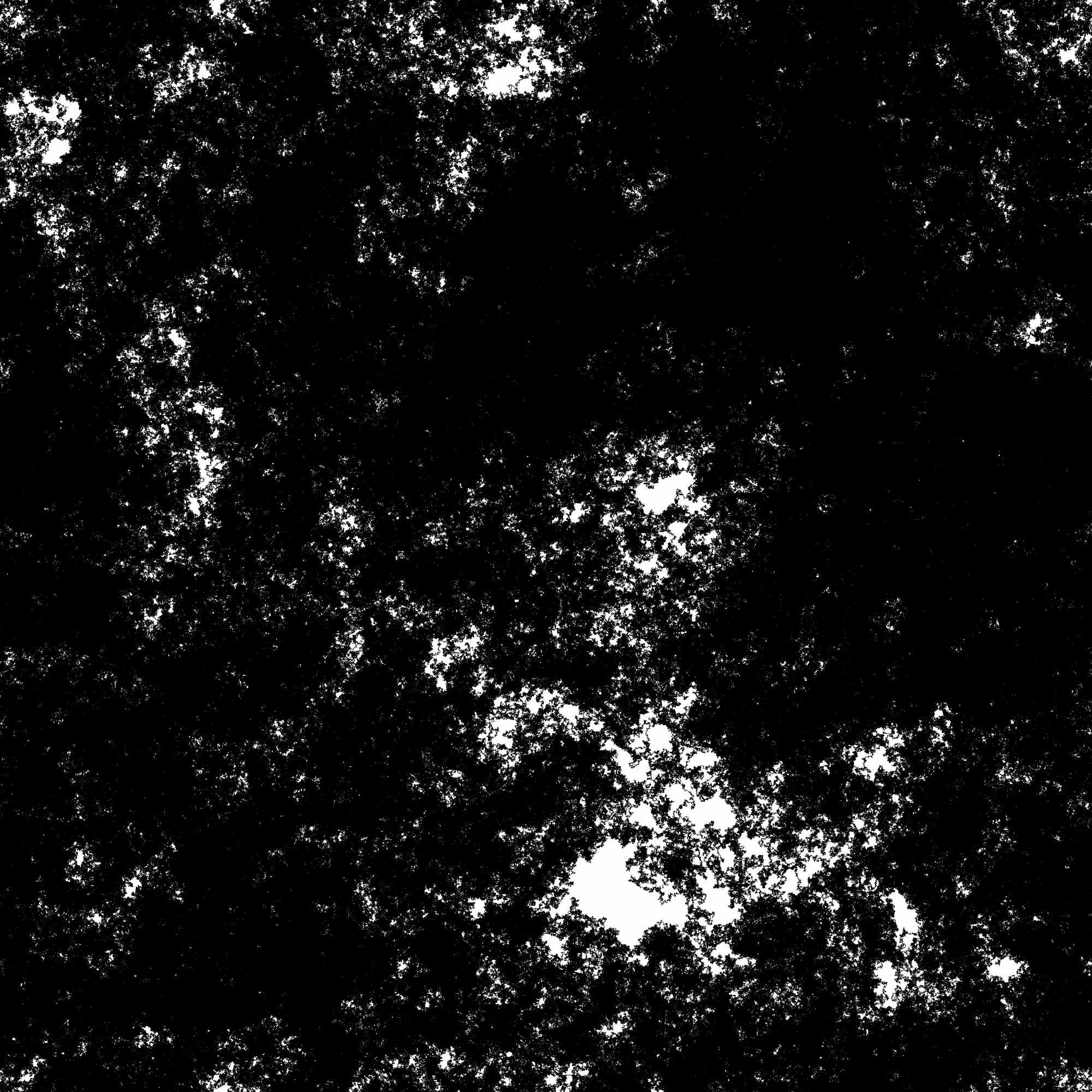}
\hfill
}
\caption{Simulation of $W[0,t]$ (shown in black) for $t=15$ in $d=2$. The holes in $\T^2\setminus W[0,t]$
(shown in white) have an irregular shape. The goal is to understand the geometry of the largest holes. The 
present paper only deals with $d \geq 3$. In \refsubsubsect{2d} below we will reflect on what 
happens in $d=2$.
}
\lbfig{d=2sim}
\end{figure}

Regions with a random boundary have been studied intensively in the literature, and 
questions such as \refquestion{Geometry} have been approached from a variety of perspectives. 
Sznitman~\cite{S1998} studies the principal Dirichlet eigenvalue when a Poisson cloud of 
obstacles is removed from Euclidean space $\R^d$, $d\geq 1$. Van den Berg, Bolthausen and 
den Hollander~\cite{vdBBdH2001} consider the large deviation properties of the volume of a 
Wiener sausage on $\R^d$, $d \geq 2$, and identify the geometric strategies for achieving 
these large deviations. Probabilistic techniques also play a role in the analysis of 
deterministic shapes, such as strong circularity in rotor-router and sandpile models 
shown by Levine and Peres~\cite{LP2009}, and heat flow in the von Koch snowflake and its 
relatives analysed by van den Berg and den Hollander~\cite{vdBdH1999}, van den 
Berg~\cite{vdB2000}, and van den Berg and Bolthausen~\cite{vdBB2004}. The discrete 
analogue to \refquestion{Geometry}, random walk on a large discrete torus, is connected to 
the random interlacements model of Sznitman~\cite{S2010} (to which we will return in 
\refsubsubsect{RandomInterlacements} below).

\refquestion{Geometry} is studied by Dembo, Peres and Rosen~\cite{DPR2003} for $d\geq 3$ and 
Dembo, Peres, Rosen and Zeitouni~\cite{DPRZ2004} for $d=2$.  In both cases, a law of large 
numbers is established for the \emph{$\epsilon$-cover time} (the time for the Brownian motion to come 
within distance $\epsilon$ of every point) as $\epsilon\decreasesto 0$. For $d\geq 3$, Dembo, Peres and Rosen also obtain the multifractal spectrum of \emph{late points} (those points that are approached within distance 
$\epsilon$ on a time scale that is a positive fraction of the $\epsilon$-cover time). In the present 
paper we will consider a large but fixed time $t$, and we will use a key lemma from \cite{DPR2003} 
to obtain global information about $\T^d \setminus W[0,t]$. Throughout the paper we fix a 
dimension $d \geq 3$. The behaviour in $d=2$ is expected to be quite different (see the discussion in \refsubsubsect{2d} below).

A random set is an infinite-dimensional object, hence issues of measurability require care. In 
general, events are defined in terms of whether a random closed set intersects a given closed 
set, or whether a random open set contains a given closed set: see Matheron~\cite{M1975} or
Molchanov~\cite{M2005} for a general theory of random sets and questions related to their 
geometry.  On the torus we will parametrize these basic events as 
\begin{equation}
\label{BasicEvent}
\set{(x+\phi E)\intersect W[0,t]=\emptyset} = \shortset{x+\phi E\subset \T^d\setminus W[0,t]}, 
\qquad x\in\T^d,\,E\subset\R^d \text{ compact}
\end{equation}
(see \eqref{AdditionOfSet} below), where $\phi>0$ acts as a scaling factor. The set $E$ in 
\eqref{BasicEvent} plays a role similar to that of a test function, and we will restrict our 
attention to suitably regular sets $E$, for instance, compact sets with non-empty interior.

\medskip\noindent
$\bullet$ 
In giving an answer to \refquestion{Geometry}, we must distinguish between global properties, 
such as the size of the largest inradius or the principal Dirichlet eigenvalue of the random 
set, and local properties, such as whether or not the random set is locally connected. In the 
present paper we focus on the \emph{global properties} of $\T^d \setminus W[0,t]$. We will 
therefore be interested in the existence of subsets of $\T^d\setminus W[0,t]$ of a given form:

\begin{varquestion}
\lbquestion{ExistsTranslate}
For a given compact set $E\subset\R^d$, what is the probability of the event
\begin{equation}
\label{ExistsTranslateEvent}
\set{\exists x\in\T^d\colon\, x+\phi E\subset \T^d\setminus W[0,t]} 
= \bigunion_{x\in\T^d}\set{x+\phi E\subset \T^d\setminus W[0,t]}
\end{equation}
formed as the uncountable union of events from \eqref{BasicEvent}?
\end{varquestion}

\noindent
For instance, questions about the inradius can be formulated in terms of 
\refquestion{ExistsTranslate} by setting $E$ to be a ball.

The answer to \refquestion{ExistsTranslate} depends on the scaling factor $\phi$.
To obtain a non-trivial result we are led to choose $\phi=\phi_d(t)$ depending on 
time, where
\begin{equation}
\label{phidtDefinition}
\phi_d(t) = \left(\frac{d}{(d-2)\kappa_d}\,\frac{\log t}{t}\right)^{1/(d-2)},
\qquad t>1,
\end{equation}
and $\kappa_d$ is the constant
\begin{equation}
\label{kappadDefinition}
\kappa_d=\frac{2\pi^{d/2}}{\Gamma(d/2-1)}.
\end{equation}
We will see that $\phi_d(t)$ represents the \emph{linear size of the largest subsets of} 
$\T^d\setminus W[0,t]$, in the sense that the limiting probability of the event in 
\eqref{ExistsTranslateEvent} decreases from $1$ to $0$ as the set $E$ increases from 
small to large, in the sense of small or large capacity (see \refsubsubsect{CapacityDefinition} 
below).

In what follows we will see that $T^d\setminus W[0,t]$ is controlled by two spatial scales:
\begin{equation}
\label{twoscales}
\phi_\mathrm{local}(t) = \left(\frac{1}{t}\right)^{1/(d-2)}, \qquad
\phi_\mathrm{global}(t) = \left(\frac{\log t}{t}\right)^{1/(d-2)}.
\end{equation}
The linear size of the \emph{typical} holes in $T^d\setminus W[0,t]$ is of order $\phi_\mathrm{local}(t)$, 
the linear size of the \emph{largest} holes of order $\phi_\mathrm{global}(t)$. The choice \eqref{phidtDefinition} of $\phi_d(t)$ is a fine tuning of the latter.

\medskip\noindent
$\bullet$
For a typical point $x\in\T^d$, the event $\set{x+\phi_d(t)E\subset\T^d\setminus W[0,t]}$ 
in \eqref{BasicEvent} is unlikely to occur even when $E$ is small. However, given a compact set
$E\subset\R^d$, the points $x\in\T^d$ for which $x+\phi_d(t)E\subset\T^d\setminus W[0,t]$
(i.e., the points that realize the event in \eqref{ExistsTranslateEvent}) are atypical, and 
we can ask whether the subset $x+\phi_d(t)E$ is likely to form part of a considerably 
larger subset:

\begin{varquestion}
\lbquestion{LargerSubset}
Are the points $x\in\T^d$ for which $x+\phi_d(t)E\subset\T^d\setminus W[0,t]$ likely to satisfy 
$x+\phi_d(t)E'\subset\T^d\setminus W[0,t]$ for some substantially larger set $E'\supset E$?
\end{varquestion}

\noindent
\refquestion{LargerSubset} aims to distinguish between the two qualitative pictures shown in 
\reffig{SparseVsDense}, which we call \emph{sparse} and \emph{dense}, respectively. We will 
show that in $d \geq 3$ the answer to \refquestion{LargerSubset} is no, i.e., the picture is dense 
as in part (b) of \reffig{SparseVsDense}. In \refsubsubsect{2d} below we will argue that in $d=2$
the answer to \refquestion{LargerSubset} is yes, i.e., the picture is sparse as in part (a) of 
\reffig{SparseVsDense}. This can already be seen from \reffig{d=2sim}.
 
\begin{figure}[htbp]
{
\hfill
(a)
\includegraphics[width=0.4\textwidth]{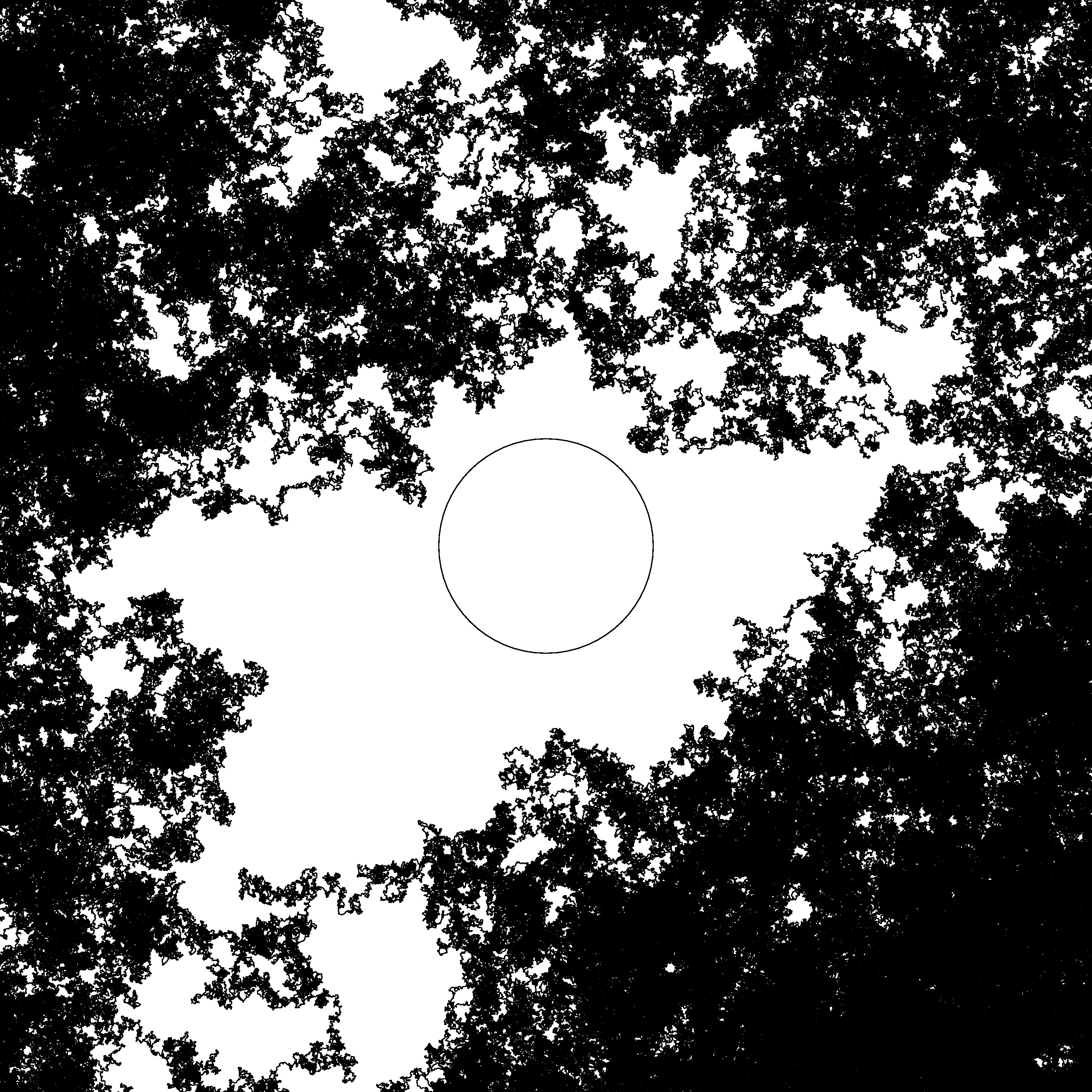}
\hfill
(b)
\includegraphics[width=0.4\textwidth]{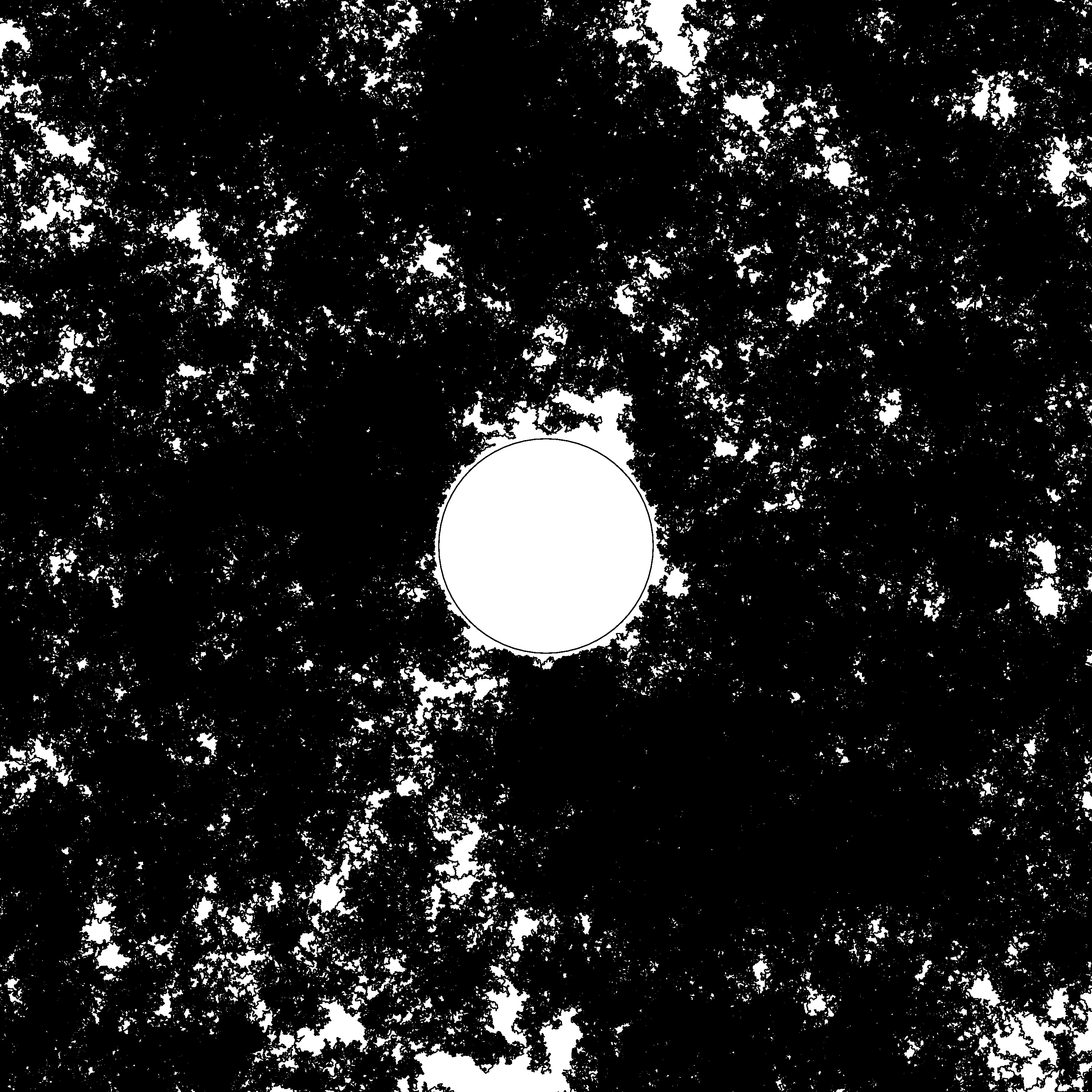}
\hfill
}
\caption{The vicinity of $x+\phi_d(t)E\subset\T^d\setminus W[0,t]$. The circular region in part (a) 
can be enlarged substantially while remaining a subset of $\T^d\setminus W[0,t]$, the circular
region in part (b) cannot.}
\lbfig{SparseVsDense}
\end{figure}

In a similar spirit, we can ask about \emph{temporal} versus \emph{spatial} avoidance 
strategies:

\begin{varquestion}
\lbquestion{Avoidance}
For a given $x\in\T^d$, does the unlikely event $\set{x+\phi_d(t)E\subset\T^d\setminus 
W[0,t]}$ arise primarily because the Brownian motion spends an unusually small amount 
of time near $x$, or because the Brownian motion spends a typical amount of time near 
$x$ and simply happens to avoid the set $x+\phi_d(t) E$?
\end{varquestion}

Questions \ref{q:LargerSubset} and \ref{q:Avoidance}, though not equivalent, are interrelated: 
if the Brownian motion spends an unusually small amount of time near $x$, then it may be 
plausibly expected to fill the vicinity of $x$ less densely, and vice versa. We will show 
that in $d \geq 3$ the Brownian motion follows a spatial avoidance strategy (the second 
alternative in \refquestion{Avoidance}) and that, indeed, the Brownian motion is very likely 
to spend approximately the same amount of time around all points of $\T^d$. In \refsubsubsect{2d} 
below we will argue that in $d=2$ the first alternative in \refquestion{Avoidance} applies.

\medskip\noindent
$\bullet$
The negative answer to \refquestion{LargerSubset} and the heuristic picture in 
\reffig{SparseVsDense}(b) suggest that regions of $\T^d$ where $W[0,t]$ is relatively 
dense nearly separate the large subsets $x+\phi_d(t)E\subset\T^d\setminus W[0,t]$ into 
disjoint components. Making sense of this heuristic is complicated by the fact that 
$\T^d\setminus W[0,t]$ is connected almost surely (see \refprop{Connected} below), 
so that all large subsets belong to the same connected component in $\T^d\setminus W[0,t]$.

\begin{varquestion}
\lbquestion{Components}
Can the approximate component structure of the large subsets of $\T^d\setminus W[0,t]$ 
be captured in a well-defined way?
\end{varquestion}

\noindent
We will provide a positive answer to \refquestion{Components} by enlarging the Brownian path 
$W[0,t]$ to a Wiener sausage $W_{\rho(t)}[0,t]$ of radius $\rho(t)=o(\phi_d(t))$. Under 
suitable hypotheses on the \emph{enlargement radius} $\rho(t)$ (see \eqref{RadiusBound} below) 
we are able to control certain properties of all the connected components of $\T^d\setminus 
W_{\rho(t)}[0,t]$ simultaneously: for instance, we compute the asymptotics of their maximum 
possible volume and capacity and minimal possible Dirichlet eigenvalue. The well-definedness 
of the approximate component structure lies in the fact that (subject to the hypothesis in 
\eqref{RadiusBound} below) these properties do not depend on the precise choice of $\rho(t)$.

The existence of a connected component of $\T^d\setminus W_{\rho(t)}[0,t]$ having a given 
property, for instance, having at least a specified volume, involves an uncountable union 
of the events in \eqref{ExistsTranslateEvent} as $E$ runs over a suitable class of connected 
sets. Central to our arguments is a \emph{discretization procedure} that reduces such an uncountable 
union to a suitably controlled finite union (see \refsect{LatticeAnimals} below).
 

\subsection{Outline}
\lbsubsect{Outl}

Our main results concern the \emph{extremal geometry} of the set $\T^d\setminus W[0,t]$ 
as $t\to\infty$. Our key theorem is a large deviation estimate for the \emph{shape of the 
component with largest capacity} in $\T^d \setminus W_{\rho(t)}[0,t]$ as $t\to\infty$, 
where $W_{\rho(t)}[0,t]$ is the Wiener sausage of radius $\rho(t)$. From this we derive 
large deviation principles for the maximal volume and the principal Dirichlet eigenvalue 
of the components of $\T^d \setminus W_{\rho(t)}[0,t]$ as $t\to\infty$, and identify the 
number of disjoint translates of $\phi_d(t)E$ in $\T^d \setminus W[0,t]$ as $t\to\infty$ for 
suitable sets $E$. We further derive large deviation principles for the largest inradius as 
$t\to\infty$ and the $\epsilon$-cover time as $\epsilon \downarrow 0$, extending laws of 
large numbers that were derived in Dembo, Peres and Rosen~\cite{DPR2003}. Along the
way we settle the five questions raised in \refsubsect{Mot}.

It turns out that the costs of the various large deviations are \emph{asymmetric}: polynomial 
in one direction and stretched exponential in the other direction. Our main results are linked 
by the heuristic that sets of the form $x+\phi_d(t)E$ appear according to a \emph{Poisson 
point process} with total intensity $t^{J_d(\Capa E)+o(1)}$, where $J_d$ is given by 
\eqref{JdkappaDefinition} below (see \reffig{JandI} below). 

The remainder of the paper is organised as follows. In \refsubsect{DefNot} we give definitions 
and introduce notations. In Sections~\ref{ss:CompStruct} and \ref{ss:GeomStruct} we state 
our main results: four theorems, five corollaries and two propositions. In \refsubsect{Discussion} 
we discuss these results, state some conjectures, make the link with random interlacements, 
and reflect on what happens in $d=2$. \refsect{Preparations} contains various estimates on 
hitting times, hitting numbers and hitting probabilities for Brownian excursions between the 
boundaries of concentric balls, which serve as key ingredients in the proofs of the main results. \refsect{LatticeAnimals} looks at non-intersection probabilities for lattice animals, which serve as discrete 
approximations to continuum sets. The proofs of the main results are given in 
Sections~\ref{s:ProofTheorems}--\ref{s:CoroProofs}. \refappendix{ExcursionProofs} contains 
the proof of two lemmas that are used along the way.


\subsection{Definitions and notations}
\lbsubsect{DefNot}


\subsubsection{Torus}

The $d$-dimensional \emph{unit torus} $\T^d$ is the quotient space $\R^d/\Z^d$, with the 
canonical projection map $\pi_0\colon\,\R^d\to\T^d$. We consider $\T^d$ as a Riemannian 
manifold in such a way that $\pi_0$ is a local isometry. The space $\R^d$ acts on $\T^d$ 
by translation: given $x=\pi_0(y_0)\in\T^d, y_0,y\in\R^d$, we define $x+y=\pi_0(y_0+y)\in
\T^d$. (Having made this definition, we will no longer need to refer to the projection 
map $\pi_0$, nor to the particular representation of the torus $\T^d$.) Given a set 
$E\subset\R^d$, a scale factor $\phi>0$, and a point $x\in\T^d$ or $x\in\R^d$, we can 
now define 
\begin{equation}
\label{AdditionOfSet}
x+\phi E=\set{x+\phi y\colon\,y\in E}.
\end{equation}

Euclidean distance in $\R^d$ and the induced distance in $\T^d$ are both denoted by 
$d(\cdot,\cdot)$. The distance from a point $x$ to a set $E$ is $d(x,E)=\inf\set{d(x,y)
\colon\,y\in E}$. The closed ball of radius $r$ around a point $x$ is denoted by $B(x,r)$, 
for $x\in\T^d$ or $x\in\R^d$. We will only be concerned with the case $0<r<\tfrac{1}{2}$, 
so that $B(x,r)=x+B(0,r)$ for $x\in\T^d$ and the local isometry $B(0,r)\to B(x,r)$, 
$y\mapsto x+y$, is one-to-one.


\subsubsection{Brownian motion and Wiener sausage}

We write $\P_{x_0}$ for the law of the Brownian motion $W=(W(t))_{t\geq 0}$ on $\T^d$ 
started at $x_0\in\T^d$, i.e., the Markov process with generator $-\tfrac{1}{2}\Delta_{\T^d}$, 
where $\Delta_{\T^d}$ is the Laplace operator for $\T^d$. We can always take $W(t)=x_0
+\tilde{W}(t)$, where $\tilde{W}=(\tilde{W}(t))_{t \geq 0}$ is the standard Brownian motion 
on $\R^d$ started at $0$, so that $W$ is the projection onto $\T^d$ (via $\pi_0$) of a Brownian motion in $\R^d$. When $x_0\in\R^d$ we will also use $\P_{x_0}$ for the law of the Brownian motion on $\R^d$. 
When the initial point $x_0$ is irrelevant we will write $\P$ instead of $\P_{x_0}$. The image 
of the Brownian motion over the time interval $[a,b]$ is denoted by $W[a,b]=\set{W(s)\colon\,
a\leq s\leq b}$.

For $r>0$ and $E\subset\R^d$ or $E\subset\T^d$, we write $E_r=\union_{x\in E} B(x,r)$ and 
$E_{-r}=[\union_{x\in E^c} B(x,r)]^c$. The \emph{Wiener sausage} of radius $r$ run for time 
$t$ is the $r$-enlargement of $W[0,t]$, i.e., $W_r[0,t]=\union_{s \in [0,t]} B(W(s),r)$.


\subsubsection{Capacity}
\lbsubsubsect{CapacityDefinition}

The (Newtonian) {\em capacity} of a Borel set $E\subset\R^d$, denoted by $\Capa E$, can be 
defined as
\begin{equation}
\label{CapacityDoubleInt}
\Capa E = \left( \inf_{\mu\in\mathcal{P}(E)} \iint_{E\times E} 
G(x,y) \, d\mu(x) \, d\mu(y)\right)^{-1},
\end{equation}
where the infimum runs over the set of probability measures $\mu$ on $E$, and
\begin{equation}
\label{Green}
G(x,y)=\frac{\Gamma(d/2 -1)}{2 \pi^{d/2} d(x,y)^{d-2}}
\end{equation}
is the Green function associated with Brownian motion on $\R^d$ (throughout the paper we 
restrict to $d \geq 3$). In terms of the constant $\kappa_d$ from \eqref{kappadDefinition}, 
we can write $G(x,y)=1/\kappa_d\,d(x,y)^{d-2}$, and it emerges that $\kappa_d=\Capa B(0,1)$ 
is the capacity of the unit ball.\footnote{See Port and Stone~\cite[Section 3.1]{PS1978}. 
The alternative normalization $\Capa B(0,1)=1$ is used also, for instance, in 
Doob~\cite[Chapter 1.XIII]{D1984}. This corresponds to replacing $G(x,y)$ by 
$1/d(x,y)^{d-2}$ in (\ref{CapacityDoubleInt}--\ref{Green}).} 

The function $E\mapsto\Capa E$ is non-decreasing in $E$ and satisfies the scaling relation
\begin{equation}
\label{CapacityScaling}
\Capa (\phi E) = \phi^{d-2} \Capa E, \qquad \phi>0,
\end{equation}
and the union bound
\begin{equation}
\label{CapacityOfUnion}
\Capa (E\union E') + \Capa (E\intersect E') \leq \Capa E+\Capa E'.
\end{equation}

Capacity has an interpretation in terms of Brownian hitting probabilities:
\begin{equation}
\label{CapacityAndHittingSimple}
\lim_{d(x,0)\to\infty} d(x,0)^{d-2}\,\P_x\big(W\cointerval{0,\infty}
\intersect E\neq\emptyset\big) 
= \frac{\Capa E}{\kappa_d}, \qquad E\subset\R^d\text{ bounded Borel.}
\end{equation}
Thus, capacity measures how likely it is for a set to be hit by a Brownian motion that starts 
far away. We will make extensive use of asymptotic properties similar to 
\eqref{CapacityAndHittingSimple}.

If a set $E$ is \emph{polar} -- i.e., with probability 1, $E$ is not hit by a Brownian motion 
started away from $E$ -- then $\Capa E=0$.  For instance, any finite or countable union 
of $(d-2)$-dimensional subspaces has capacity zero.


\subsubsection{Sets}

The boundary of a set $E$ is denoted by $\boundary E$, the interior by $\interior{E}$, and 
the closure by $\closure{E}$.  We define
\begin{equation}
\cE_c=\set{\text{$E\subset\R^d$ compact: $E$ and $\R^d \setminus E$ connected}}
\end{equation} 
We will use these sets to describe the possible components of $\T^d\setminus W_{\rho(t)}[0,t]$.  
We further define
\begin{equation}
\cE^* = \set{\text{$E\subset\R^d$ compact: $\Capa E=\Capa(\interior{E})$}} 
\union \set{\text{$E\subset\R^d$ bounded open: $\Capa E=\Capa(\closure{E})$}}.
\end{equation}
The condition $\Capa(\interior{E})=\Capa(\closure{E})$ in the definition of $\cE^*$ is satisfied 
when every point of $\boundary E$ is a regular point for $\interior{E}$, which in turn is satisfied 
when $E$ satisfies a cone condition at every point (see Port and Stone~\cite[Chapter~2, 
Proposition~3.3]{PS1978}).  In particular, any finite union of cubes, or any $r$-enlargement 
$E_r$ with $r>0$ of a compact set $E$, belongs to $\cE^*$.


\subsubsection{Maximal capacity of a component}

A central role will be played by the largest capacity $\kappa^*(t,\rho)$ for a component of 
$\T^d\setminus W_\rho[0,t]$, defined by
\begin{equation}
\kappa^*(t,\rho)=\sup\set{\Capa E\colon\, \text{$E\subset\R^d$ connected,
$x+E\subset\T^d\setminus W_\rho[0,t]$ for some $x\in\T^d$}}.
\end{equation}
Note that by rescaling we have
\begin{equation}
\frac{\kappa^*(t,\rho)}{\phi_d(t)^{d-2}}
= \sup\set{\Capa E\colon\,\text{$E\subset\R^d$ connected, $x+\phi_d(t)E\subset\T^d\setminus 
W_\rho[0,t]$ for some $x\in\T^d$}}.
\end{equation}


\subsection{Component structure}
\lbsubsect{CompStruct}

We begin by describing the component structure of $\T^d\setminus W_\rho[0,t]$. In 
formulating the results below we will use the abbreviation (see \reffig{JandI}(a)) 
\begin{equation}
\label{JdkappaDefinition}
J_d(\kappa)=\frac{d}{d-2}\left(1-\frac{\kappa}{\kappa_d}\right), \qquad \kappa \geq 0.
\end{equation}

\begin{figure}[htbp]
{
\hfill
(a)
\includegraphics{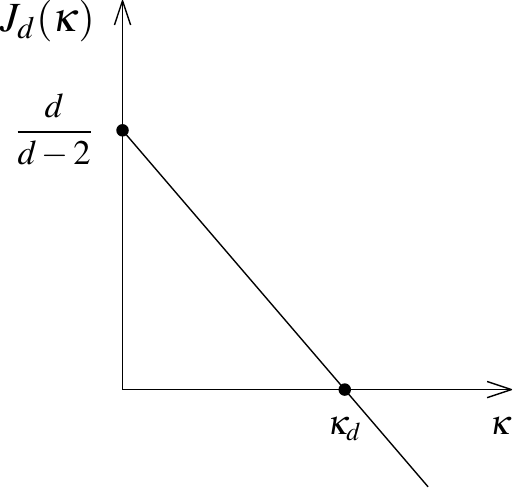}
\hfill
(b)
\includegraphics{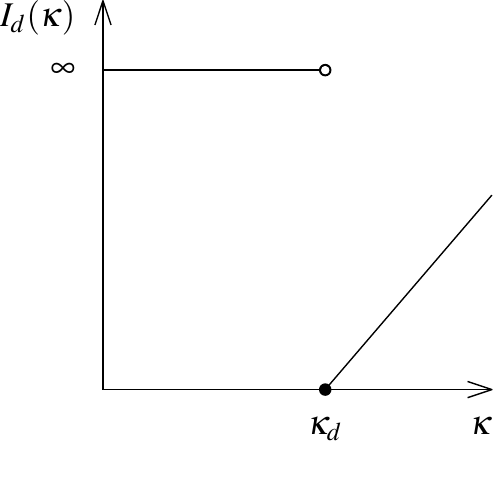}
\hfill
}
\caption{(a) The function $\kappa \mapsto J_d(\kappa)$ in \eqref{JdkappaDefinition}. 
(b) The rate function $\kappa \mapsto I_d(\kappa)$ in \eqref{Idrf}.}
\lbfig{JandI}
\end{figure}

Our first theorem quantifies the likelihood of finding sets of large capacity that do not intersect $W_\rho[0,t]$, for $\rho$ in a certain window between the local and the global spatial scales 
defined in \eqref{twoscales}.

\begin{theorem}
\lbthm{CapacitiesInWc}
Fix a positive function $t \mapsto \rho(t)$ satisfying
\begin{equation}
\label{RadiusBound}
\lim_{t\to\infty} \frac{\rho(t)}{\phi_d(t)} = 0,
\qquad
\lim_{t\to\infty} \frac{(\log t)^{1/d}\rho(t)}{\phi_d(t)} =\infty.
\end{equation}
Then the family $\P(\kappa^*(t,\rho)/\phi_d(t)^{d-2}\in\cdot)$, $t>1$, satisfies the 
LDP on $[0,\infty]$ with rate $\log t$ and rate function (see \reffig{JandI}(b))
\begin{equation}
\label{Idrf}
I_d(\kappa) =
\begin{cases}
-J_d(\kappa), & \kappa\geq \kappa_d,\\
\infty, & \kappa < \kappa_d,
\end{cases}
\end{equation}
with the convention that $I_d(\infty)=\infty$.
\end{theorem}

The counterpart of \refthm{CapacitiesInWc} for small capacities is contained in the following 
two theorems, which show that components of small capacity are likely to exist and to be numerous. 
Let $\chi_\rho(t,\kappa)$ denote the number of components $C$ of $\T^d\setminus W_\rho[0,t]$ 
such that $C$ contains some ball of radius $\rho$ and has the form $C=x+\phi_d(t)E$ 
for a connected open set $E$ with $\Capa E\geq\kappa$.

\begin{theorem}
\lbthm{ComponentCounts}
Fix a positive function $t\mapsto\rho(t)$ satisfying \eqref{RadiusBound}, and let $\kappa<\kappa_d$. 
Then
\begin{equation}
\lim_{t\to\infty}\frac{\log\chi_{\rho(t)}(t,\kappa)}{\log t}=J_d(\kappa) 
\qquad\text{in $\P$-probability.}
\end{equation}
\end{theorem}

\begin{theorem}
\lbthm{NoTranslates}
Fix a non-negative function $t\mapsto\rho(t)$ satisfying $\rho(t)=o(\phi_d(t))$, and let $E\subset\R^d$ be compact with $\Capa E <\kappa_d$. Then
\begin{equation}
\log\P\left( \nexists\, x\in\T^d\colon\, x+\phi_d(t)E\subset\T^d\setminus W_{\rho(t)}[0,t] \right)
\leq -t^{J_d(\Capa E)+o(1)}, \qquad t\to\infty.
\end{equation}
\end{theorem}

The next theorem identifies the shape of the components of $\T^d \setminus W_{\rho(t)}[0,t]$. 
For $E\subset E'$ a pair of nested compact connected subsets of $\R^d$, we say that a component 
$C$ of $\T^d\setminus W_\rho[0,t]$ satisfies condition \eqref{CtrhoEEprime} when
\begin{equation*}
\tag{$\cC(t,\rho,E,E')$}
\label{CtrhoEEprime}
C=x+\phi_d(t)U, \qquad x\in\T^d,\,E\subset U\subset E'.
\end{equation*}
Define $\chi_\rho(t,E,E')$ to be the number of components of $\T^d\setminus W_\rho[0,t]$ 
satisfying condition \eqref{CtrhoEEprime}, and define $F_\rho(t,E,E')$ to be the event 
\begin{equation}
F_\rho(t,E,E')=
\set{\parbox{275pt}{There exists a component $C=x+\phi_d(t) U$ of $\T^d\setminus W_\rho[0,t]$ 
satisfying condition \eqref{CtrhoEEprime}, and any other component $C'=x'+\phi_d(t) U'$ 
has $\Capa U'<\Capa U$.}}
\end{equation}
In words, $F_\rho(t,E,E')$ is the event that $\T^d\setminus W_\rho[0,t]$ contains a component 
sandwiched between $x+\phi_d(t) E$ and $x+\phi_d(t)E'$, and any other component has smaller 
capacity (when viewed as a subset of $\R^d$).

\begin{theorem}
\lbthm{ShapeOfComponents}
Fix a positive function $t\mapsto\rho(t)$ satisfying \eqref{RadiusBound}, let $E\in\cE_c$, and 
let $\delta>0$. If $\Capa E\geq\kappa_d$, then
\begin{equation}
\lim_{t\to\infty}\frac{\log\P\bigl(F_{\rho(t)}(t,E,E_\delta)\bigr)}{\log t}
= J_d(\Capa E) \quad (=-I_d(\Capa E)),
\end{equation}
while if $\Capa E<\kappa_d$, then
\begin{equation}
\lim_{t\to\infty}\frac{\log \chi_{\rho(t)}(t,E,E_\delta)}{\log t} 
= J_d(\Capa E) \qquad\text{in $\P$-probability.}
\end{equation}
\end{theorem}

Theorems~\ref{t:CapacitiesInWc}--\ref{t:ShapeOfComponents} yield the following corollary. For 
$E\subset\R^d$, let $\chi(t,E)$ denote the maximal number of disjoint translates $x+\phi_d(t)E$ 
in $\T^d\setminus W[0,t]$.

\begin{corollary}
\lbcoro{UnhitSet}
Suppose that $E\in\cE^*$. Then
\begin{equation}
\label{HittingSetProbSimple}
\lim_{t\to\infty} \P\left(\exists\, x\in\T^d\colon\, 
x+\phi_d(t) E\subset\T^d\setminus W[0,t]\right) 
=
\begin{cases}
1,&\Capa E<\kappa_d,\\
0,&\Capa E>\kappa_d.
\end{cases}
\end{equation}
Furthermore,
\begin{equation}
\label{HittingLargeSetProb}
\lim_{t\to\infty} \frac{\log\P(\exists\, x\in\T^d\colon\, 
x+\phi_d(t) E\subset\T^d\setminus W[0,t])}{\log t}= J_d(\Capa E) \vee 0,
\end{equation}
and if $\Capa E<\kappa_d$, then
\begin{equation}
\label{DisjointTranslates}
\lim_{t\to\infty} \frac{\log\chi(t,E)}{\log t} = J_d(\Capa E)
\quad \text{in $\P$-probability.}
\end{equation}
\end{corollary}


\subsection{Geometric structure}
\lbsubsect{GeomStruct}

Having described the components in terms of their capacities in \refsubsect{CompStruct}, 
we are ready to look at the geometric structure of our random set. Our first corollary 
concerns the maximal volume of a component of $\T^d\setminus W_\rho[0,t]$, which we denote 
by $V(t,\rho)$. Volume is taken w.r.t.\ the $d$-dimensional Lebesgue measure, and we write 
$V_d=\Vol B(0,1)$ for the volume of the $d$-dimensional unit ball.

\begin{corollary}
\lbcoro{VolumeLDP}
Subject to \eqref{RadiusBound}, the family $\P(V(t,\rho(t))/\phi_d(t)^d\in\cdot)$, $t>1$, 
satisfies the LDP on $(0,\infty)$ with rate $\log t$ and rate function
\begin{equation}
I^{\rm volume}_d(v) = I_d\left( \kappa_d (v/V_d)^{(d-2)/d} \right) \! .
\end{equation}
Moreover, for $v<V_d$,
\begin{equation}
\label{SmallVolume}
\log \P\big(V(t,\rho(t))/\phi_d(t)^d < v\big) 
\leq -t^{J_d(\kappa_d(v/V_d)^{(d-2)/d})+o(1)}, 
\qquad t\to\infty.
\end{equation}
\end{corollary}

Our second corollary concerns $\lambda(t,\rho)=\lambda(\T^d\setminus W_\rho[0,t])$, the 
principal Dirichlet eigenvalue of $\T^d\setminus W_\rho[0,t]$, where by $\lambda(E)$ (for 
$E\subset\T^d$ or $E\subset\R^d$) we mean the principal eigenvalue of the operator 
$-\tfrac{1}{2}\Delta_E$ with Dirichlet boundary conditions on $\boundary E$. We write 
$\lambda_d=\lambda(B(0,1))$ for the principal Dirichlet eigenvalue of the $d$-dimensional 
unit ball.

\begin{corollary}
\lbcoro{EvalLDP}
Subject to \eqref{RadiusBound}, the family $\P(\phi_d(t)^2 \lambda(t,\rho(t)) \in\cdot)$, $t>1$, 
satisfies the LDP on $(0,\infty)$ with rate $\log t$ and rate function
\begin{equation}
I^{\rm Dirichlet}_d(\lambda)
= I_d\left( \kappa_d (\lambda_d/\lambda)^{(d-2)/2} \right) \! .
\end{equation}
Moreover, for $\lambda> \lambda_d$, 
\begin{equation}
\label{LargeEval}
\log\P\big(\phi_d(t)^2 \lambda(t,\rho(t)) \geq \lambda\big) 
\leq -t^{J_d(\kappa_d(\lambda_d/\lambda)^{(d-2)/2})+o(1)}, \qquad t\to\infty.
\end{equation}
\end{corollary}

Our last two corollaries concern the largest inradius of $\T^d\setminus W[0,t]$, 
\begin{equation}
\rho_{\rm in}(t) = \sup_{x\in\T^d} d(x,W[0,t]) 
= \sup\set{\rho\geq 0\colon\, \T^d\setminus W_\rho[0,t] \neq \emptyset},
\qquad t>0,
\end{equation}  
and the $\epsilon$-cover time,
\begin{equation}
\cC_\epsilon =\sup_{x\in\T^d} \inf\set{\big. t\geq 0\colon\, d(x,W[0,t])\leq\epsilon} 
= \inf\set{\big. t\geq 0\colon\,\rho_{\rm in}(t)\leq\epsilon},
\qquad 0<\epsilon<1.
\end{equation}
For the latter we need the scaling function
\begin{equation}
\label{psidepsilonDefinition}
\psi_d(\epsilon) 
= \frac{\epsilon^{-(d-2)} \log (1/\epsilon)}{\kappa_d}, \qquad 0<\epsilon<1.
\end{equation}

\begin{corollary}
\lbcoro{InradiusLDP}
The family $\P(\rho_{\rm in}(t)/\phi_d(t)\in\cdot\,)$, $t>1$, satisfies the LDP on 
$(0,\infty)$ with rate $\log t$ and rate function 
\begin{equation}
I^{\rm inradius}_d(r) = I_d(\kappa_d\,r^{d-2}).
\end{equation} 
Moreover, for $0<r<1$,
\begin{equation}
\label{SmallInradius}
\log \P\big(\rho_{\rm in}(t)/\phi_d(t) < r\big) 
\leq - t^{J_d(\kappa_d\,r^{d-2})+o(1)}, \qquad t\to\infty.
\end{equation}
\end{corollary}

\begin{corollary}
\lbcoro{CoverTimeLDP}
The family $\P(\cC_\epsilon/\psi_d(\epsilon)\in\cdot\,)$, $0<\epsilon<1$, satisfies the 
LDP on $(0,\infty)$ with rate $\log(1/\epsilon)$ and rate function 
\begin{equation}
I_d^{\rm cover}(u) 
=
\begin{cases}
u-d, &u\geq d,\\
\infty, &0<u< d.
\end{cases}
\end{equation}
Moreover, for $0<u<d$,
\begin{equation}
\label{SmallCoverTime}
\log \P\big(\cC_\epsilon/\psi_d(\epsilon)<u\big) 
\leq -\epsilon^{-(d-u)+o(1)}, \qquad \epsilon \downarrow 0.
\end{equation}
\end{corollary}

\refcoro{CoverTimeLDP} is equivalent to \refcoro{InradiusLDP} because of the relation 
$\set{\rho_{\rm in}(t) > \epsilon} = \set{\cC_\epsilon > t}$ and the asymptotics
\begin{equation}
\label{phipsiAsymptotics}
\phi_d(u\psi_d(\epsilon)) \sim \left( \frac{u}{d} \right)^{1/(d-2)} \epsilon,
\quad \epsilon \decreasesto 0,\,u>0,
\qquad \psi_d(r\phi_d(t)) \sim \frac{t}{d r^{d-2}},
\quad t\to\infty,\,r>0.
\end{equation}


\subsection{Discussion}
\lbsubsect{Discussion} 


\subsubsection{Upward versus downward deviations and the role of 
\texorpdfstring{$J_d(\kappa)$}{J\_d(kappa)}}
\lbsubsubsect{UpDownDiscussion}

\refthm{CapacitiesInWc} says that the region with largest capacity not intersecting the Wiener sausage 
of radius $\rho(t)$ lives on scale $\phi_d(t)$, and that upward large deviations on this scale 
have a cost that decays {\em polynomially} in $t$.  \refthm{ComponentCounts} identifies how 
many components there are with small capacity. This number grows {\em polynomially} in $t$.  
\refthm{NoTranslates} says that this number is extremely unlikely to be zero: the cost is 
{\em stretched exponential} in $t$. Theorem~\ref{t:ShapeOfComponents} completes the picture 
obtained from Theorems~\ref{t:CapacitiesInWc}--\ref{t:NoTranslates} by showing that components 
can approximate any shape in $\cE_c$.

Theorems~\ref{t:CapacitiesInWc}--\ref{t:ShapeOfComponents} and are linked by the 
heuristic that components of the form $x+\phi_d(t)E$ appear according to a Poisson point 
process with total intensity $t^{J_d(\Capa E)+o(1)}$. When $\Capa E > \kappa_d$ we have 
$J_d(\Capa E)<0$, and the likelihood of even a single such component is $t^{-\abs{J_d(\Capa E)}
+o(1)}$, as in \refcoro{UnhitSet}. When $\Capa E<\kappa_d$ we have $J_d(\Capa E)>0$, and a 
Poisson random variable $X$ of mean $t^{J_d(\Capa E)+o(1)}$ satisfies $X=t^{J_d(\Capa E)+o(1)}$ 
with high probability and $\P(X=0)= \exp[ -t^{J_d(\Capa E)+o(1)}]$. Based on this heuristic, we 
conjecture that the inequalities in \eqref{SmallVolume}, \eqref{LargeEval}, \eqref{SmallInradius} 
and \eqref{SmallCoverTime} are all equalities asymptotically.


\subsubsection{Components and the role of \texorpdfstring{$\rho(t)$}{rho(t)}}
\lbsubsubsect{ComponentsDiscussion}

Theorems~\ref{t:CapacitiesInWc}--\ref{t:ShapeOfComponents} concern components of the form 
$x+\phi_d(t)E$.  We begin by remarking that, with high probability, all components have 
this form:

\begin{proposition}
\lbprop{NoWrapping}
Assume \eqref{RadiusBound}. Let ${\rm Wrap}(t,\rho)$ be the event that $\T^d\setminus 
W_\rho[0,t]$ has a component $C$ that, when considered as a Riemannian manifold with 
its intrinsic metric, is not the isometric image $x+E$ of some bounded subset $E$ of $\R^d$. 
Then
\begin{equation}
\lim_{t\to\infty}\frac{\log\P\left( {\rm Wrap}(t,\rho(t)) \right)}{\log t}
 = -\infty.
\end{equation}
\end{proposition}
\noindent
Informally, such a component must ``wrap around'' the torus, so that the local isometry 
from $\R^d$ to $\T^d$ is not a global isometry. \refprop{NoWrapping} means that, apart from 
a negligible event, we may sensibly consider the components as subsets of $\R^d$ and discuss 
their capacities as defined in \eqref{CapacityDoubleInt}.

Collectively, Theorems~\ref{t:CapacitiesInWc}--\ref{t:ShapeOfComponents}, 
Corollaries~\ref{c:VolumeLDP}--\ref{c:EvalLDP} and \refprop{NoWrapping} show that $\T^d\setminus 
W_{\rho(t)}[0,t]$ has a \emph{component structure}, with well-defined bounds on the capacities, 
volumes and principal Dirichlet eigenvalues of these components. By contrast, the choice $\rho(t)=0$ 
does not give a component structure at all:
\begin{proposition}
\lbprop{Connected}
With probability $1$, the set $\T^d\setminus W[0,t]$ is path-connected, open and dense for 
every $t$, and the set $\T^d\setminus W\cointerval{0,\infty}$ is path-connected, locally 
path-connected and dense.
\end{proposition}

The picture behind Propositions~\ref{p:NoWrapping}--\ref{p:Connected} is that the set $\T^d 
\setminus W[0,t]$ consists of ``lakes'' whose linear size is of order $\phi_d(t)$, connected 
by narrow ``channels'' whose linear size is at most $\phi_d(t)/(\log t)^{1/d}$. By inflating 
the Brownian motion to a Wiener sausage of radius $\rho(t)$ with (recall \eqref{twoscales} and \eqref{RadiusBound}) 
\begin{equation}
\label{regime}
\phi_\mathrm{local}(t) \ll \phi_d(t)/(\log t)^{1/d} \ll \rho(t) \ll \phi_d(t) 
\asymp \phi_\mathrm{global}(t), 
\end{equation}
we effectively block off these channels, so that $\T^d\setminus W_{\rho(t)}[0,t]$ consists of 
disjoint lakes.

\refprop{Connected} shows that some lower bound on $\rho(t)$ is necessary for the results 
of Theorems~\ref{t:CapacitiesInWc}--\ref{t:ShapeOfComponents}, 
Corollaries~\ref{c:VolumeLDP}--\ref{c:EvalLDP} and \refprop{NoWrapping} to 
hold.\footnote{The choice $\rho(t)=0$ makes the eigenvalue result in \refcoro{EvalLDP} false 
for $d\geq 4$, since the path of the Brownian motion itself is a polar set for $d\geq 4$. However, 
for $d = 3$ the eigenvalue $\lambda(t,\rho(t))$ is non-trivial even when $\rho(t)=0$, and we conjecture 
that \refcoro{EvalLDP} remains valid, i.e., the eigenvalue is determined primarily by the large lakes in $\T^d\setminus W[0,t]$, and not by the narrow channels connecting them. See 
the rough estimates in van den Berg, Bolthausen and den Hollander~\cite{vdBBdHpr}.} 
It would be of interest to know whether the condition $\rho(t)\gg \phi_d(t)/(\log t)^{1/d}$ 
can be relaxed, i.e., whether the true size of the channels is of smaller order than 
$\phi_d(t)/(\log t)^{1/d}$. By analogy with the random interlacements model (see 
\refsubsubsect{RandomInterlacements} below), the relevant regime to study would be 
$\phi_\mathrm{local}(t) \asymp \phi_d(t)/(\log t)^{1/(d-2)} \ll \rho(t) \ll \phi_d(t)/(\log t)^{1/d}$, 
i.e., the missing part of \eqref{regime}.


\subsubsection{A comparison with random interlacements}
\lbsubsubsect{RandomInterlacements}

The discrete analogue of $\T^d\setminus W[0,t]$ is the complement $\T_N^d\setminus S[0,n]$ 
of the path of a random walk $S=(S(n))_{n\in\N_0}$ on a large discrete torus $\T_N^d=(\Z/N\Z)^d$.  
The spatial scale being fixed by discretization, it is necessary to take $N\to\infty$ and 
$n\to\infty$ simultaneously, and the choice $n=uN^d$ for $u\in(0,\infty)$ has been extensively 
studied: see for instance Benjamini and Sznitman~\cite{BenjSznit2008}, Sznitman~\cite{S2010} 
and Sidoravicius and Sznitman~\cite{SS2009}. Teixeira and Windisch~\cite{TeixWind2011} prove 
that $S[0,uN^d]$, seen \emph{locally} from a \emph{typical} point, converges in law as $N\to\infty$,
namely,
\begin{equation}
\label{DiscreteTorusLocal}
\lim_{N\to\infty} \P\left( (X_N+E) \cap  S[0,uN^d] = \emptyset \right) 
= e^{-u\Capa_{\Z^d} E}, \qquad E\subset\Z^d\text{ finite},
\end{equation}
where $X_N$ is drawn uniformly from $\T_N^d$, and $\Capa_{\Z^d} E$ is the discrete capacity. The 
right-hand side of \eqref{DiscreteTorusLocal} is the non-intersection probability 
\begin{equation}
\P(E \cap \cI^u = \emptyset) = e^{-u\Capa_{\Z^d} E}
\end{equation}
for the \emph{random interlacements} model with parameter $u$ introduced by Sznitman~\cite{S2010}. 
The set $\cI^u\subset\Z^d$ can be constructed as the union of a certain Poisson point process of random walk 
paths, with an intensity measure proportional to the parameter $u$. The random interlacements model 
has a critical value $u_*\in(0,\infty)$ such that $\Z^d\setminus\cI^u$ has an unbounded component 
a.s.\ when $u<u_*$ and has only bounded components a.s.\ when $u>u_*$.

The continuous analogue of \eqref{DiscreteTorusLocal} is the probability of the event in 
\eqref{BasicEvent} with the scaling factor $\phi=\phi_\mathrm{local}(t) = t^{-1/(d-2)}$ instead of 
$\phi=\phi_d(t) \asymp \phi_\mathrm{global}(t)$. Our methods (see Propositions~\ref{p:ExcursionNumbers} 
and \ref{p:NSuccessfulProb} below) yield
\begin{equation}
\label{CtsTorusLocal}
\lim_{t\to\infty} \left( (X+t^{-1/(d-2)}E) \cap W[0,t] = \emptyset \right) 
= e^{-\Capa E},\qquad E\subset\R^d\text{ compact},
\end{equation}
for $X$ drawn uniformly from $\T^d$, which implies that the random set $\T^d\setminus W[0,t]$, 
seen locally from a typical point, converges in law (see Molchanov~\cite[Theorem 6.5]{M2005} 
for a discussion of convergence in law for random sets) to a random closed set $\cI$ uniquely 
characterized by its non-intersection probabilities
\begin{equation}
\label{CtsInterlacements}
\P(E \cap \cI = \emptyset) = e^{-\Capa E}, \qquad E\subset\R^d\text{ compact}.
\end{equation}
As with the discrete random interlacements $\cI^u$, the limiting random set $\cI$ can be constructed 
from a Poisson point process of Brownian motion paths (see Sznitman~\cite[Section 2]{SznitmanBrInt}).

Because of scale invariance, no parameter is needed in \eqref{CtsTorusLocal}--\eqref{CtsInterlacements}. 
Indeed, the continuous model corresponds to a rescaled limit of the discrete model when $(N,u)$ 
is replaced by $(kN,u/k^{d-2})$ and $k\to\infty$. In this rescaling the parameter $u$ tends to zero,
and $\Z^d\setminus\cI^u$ loses its finite component structure, which is in accordance with the 
connectedness result \refprop{Connected}.

Inflating the Brownian motion to a Wiener sausage can be interpreted as reintroducing a kind of 
discretization. However, because of \eqref{RadiusBound}, the spatial scale $\rho(t)$ of this 
discretization is much larger than the spatial scale $\phi_\mathrm{local}(t) = t^{-1/(d-2)}$ corresponding 
to \eqref{CtsTorusLocal} (cf.\ \refsubsubsect{ComponentsDiscussion}).

In the random interlacements model no sharp bound is currently known for the tail behaviour of the 
capacity of the component containing the origin. Recently, Popov and Teixeira~\cite{PopovTeixeiraSoftLocal} 
showed that for $d \geq 3$ the \emph{diameter} of the component containing $0$ in $\Z^d\setminus\cI^u$ 
has an exponential tail for $u$ sufficiently large (with a logarithmic correction in $d=3$).  In particular, the 
largest diameter of a component in a box of volume $N^d$, $d\geq 4$, can grow at most as $\log N$, 
and therefore the largest capacity of a component can grow at most as $(\log N)^{d-2}$.

When this last bound is translated heuristically to our context, the corresponding assertion is that the 
maximal capacity of a component is at most of order $(\log t)^{d-2}/t$.  By \refthm{CapacitiesInWc}, 
this bound is very far from sharp for $d\geq 4$.  It is tempting to conjecture that the \emph{capacity} 
of the component containing $0$ in $\Z^d\setminus\cI^u$ also has an exponential tail for $u$ sufficiently 
large. The reasonableness of this conjecture is related to whether or not the condition on $\rho(t)$
in \eqref{RadiusBound} can be weakened to $\rho(t) \geq u\phi_\mathrm{local}(t)$ for $u$ sufficiently large.
Possibly the scaling behaviour of $\T^d\setminus W_{\rho(t)}[0,t]$ with $\rho(t) = u\phi_\mathrm{local}(t)$ undergoes 
some sort of percolation transition at a critical value $\bar u_* \in (0,\infty)$.


\subsubsection{Corollaries of the capacity bounds}

\refcoro{UnhitSet} summarizes for which set $E$ a subset $x+\phi_d(t)E \subset \T^d\setminus W[0,t]$ 
can be expected to exist: according to Theorems~\ref{t:CapacitiesInWc}--\ref{t:ShapeOfComponents},
subsets of large capacity are unlikely to exist, whereas subsets of small capacity are numerous. 

Corollaries~\ref{c:VolumeLDP}--\ref{c:EvalLDP} follow from 
Theorems~\ref{t:CapacitiesInWc}--\ref{t:ShapeOfComponents} with the help of the isoperimetric 
inequalities
\begin{equation}
\label{CapaVolEval}
\frac{\Capa E}{\kappa_d} \geq \left( \frac{\Vol E}{V_d} \right)^{(d-2)/d} 
\geq \left( \frac{\lambda_d}{\lambda(E)} \right)^{(d-2)/2},
\qquad E\subset\R^d\text{ bounded open,}
\end{equation}
where  we recall that $\kappa_d,V_d,\lambda_d$ are the capacity, volume and principal Dirichlet 
eigenvalue of $B(0,1)$. The first inequality is the Poincar\'e-Faber-Szeg\"o theorem, which 
says that among all sets with a given volume the ball has the smallest capacity. The second 
inequality is the Faber-Krahn theorem, which says that among all sets of a given volume the 
ball has the smallest Dirichlet eigenvalue.\footnote{See e.g.\ 
Bandle~\cite[Theorems II.2.3 and III.3.8]{B1980} or P\'olya and 
Szeg\"o~\cite[Section I.1.12]{PS1951}. These references consider the capacity only when $d=3$, 
but their methods apply for all $d\geq 3$.} Comparing with \refthm{CapacitiesInWc}, we see
that the most efficient way to produce a component of a given large volume (or small principal 
Dirichlet eigenvalue) is for that component to be a ball.

Equality holds throughout \eqref{CapaVolEval} when $E$ is a ball, and the lower bounds in 
Corollaries~\ref{c:VolumeLDP}--\ref{c:EvalLDP}, together with 
Corollaries~\ref{c:InradiusLDP}--\ref{c:CoverTimeLDP}, follow by specializing 
Theorems~\ref{t:CapacitiesInWc}--\ref{t:ShapeOfComponents} to that case.

The large deviation principles in \refthm{CapacitiesInWc} and 
Corollaries~\ref{c:VolumeLDP}--\ref{c:CoverTimeLDP} each imply a weak law of large numbers, 
e.g.\ $\lim_{t\to\infty} \kappa^*(t,\rho(t))/\phi_d(t)^{d-2}=1$ in $\P$-probability. 
The weak laws of large numbers implied by Corollaries~\ref{c:InradiusLDP}--\ref{c:CoverTimeLDP} 
were proved in Dembo, Peres and Rosen~\cite{DPR2003} in the stronger form $\lim_{t\to\infty}
\rho_{\rm in}(t)/\phi_d(t)=1$ and $\lim_{t\to\infty} \cC_\epsilon/\psi_d(\epsilon)=d$
$\P$-a.s. The $L^1$-version of this convergence is proved in van den Berg, 
Bolthausen and den Hollander~\cite{vdBBdHpr}. Note that none of these forms are equivalent: 
for instance, a.s.\ convergence does not follow from 
Corollaries~\ref{c:InradiusLDP}--\ref{c:CoverTimeLDP}, since the sum $\sum_{t\in\N} 
\exp[-I_d(\kappa)\log t]$ fails to converge when $I_d(\kappa)$ is small.


\subsubsection{The maximal diameter of a component}
\lbsubsubsect{DiameterDiscussion}

There is no analogue of \refcoro{VolumeLDP} for the maximal diameter instead of the 
maximal volume. The capacity and the diameter are related by $\Capa E \leq \kappa_d 
(\diam E)^{d-2}$. However, there is no inequality in the reverse direction: a set of 
fixed capacity can have an arbitrarily large diameter. It turns out that the maximal 
diameter of the components of $\T^d \setminus W_{\rho(t)}[0,t]$ is of larger order 
than $\phi_d(t)$. More precisely, suppose that $\rho(t)=o(\phi_d(t))$, and let $D(t,\rho(t))$ 
denote the largest diameter of a component of $\T^d\setminus W_{\rho(t)}[0,t]$. Then 
$\lim_{t\to\infty} D(t,\rho(t))/\phi_d(t)=\infty$ in $\P$-probability. Indeed, choose 
a compact connected set $E$ of zero capacity and large diameter, say $E=[0,L]\times
\set{0}^{d-1}$ with $L$ large. Then, by \refthm{NoTranslates}, $\T^d \setminus W_{\rho(t)}[0,t]$ 
has a component containing $x+\phi_d(t) E$ for some $x$ with a high probability.  See also the discussion at the end of \refsubsubsect{RandomInterlacements} above.


\subsubsection{The second-largest component}

The component of second-largest capacity (or second-largest volume, principal Dirichlet 
eigenvalue, or inradius) has a different large deviation behaviour, due to the fact 
that $E\mapsto\Capa E$ is not additive. Indeed, typically $\Capa(E^{(1)}\union E^{(2)})
<\Capa(E^{(1)})+\Capa(E^{(2)})$, even for disjoint sets $E^{(1)}, E^{(2)}$. In the case 
of concentric spheres, $\Capa(\boundary B(0,r_1)\union\boundary B(0,r_2))=\max\set{\Capa
(\boundary B(0,r_1)), \Capa(\boundary B(0,r_2))}$. It follows that the most efficient way 
to produce two large but disjoint components is to have them almost touching.


\subsubsection{Answers to Questions \ref{q:Geometry}--\ref{q:Components}}

The results in this paper give a partial answer to \refquestion{Geometry}. \refquestion{ExistsTranslate} 
is answered by \refcoro{UnhitSet} subject to $E\in\cE^*$, $\Capa E\neq\kappa_d$ (see also 
\refsect{LatticeAnimals} for results that are simultaneous over a certain class of sets $E$). The resolution to \refquestion{LargerSubset}, namely, the fact that the dense picture in 
\reffig{SparseVsDense}(b) applies, is provided by \refcoro{UnhitSet}.  If $E\subset E'$ with 
$\Capa E'\geq\Capa E+\delta$, $\delta>0$, and $E,E'\in\cE^*$, then, compared to subsets of 
the form $x+\phi_d(t)E$, subsets of the form $x+\phi_d(t)E'$ are much less numerous (when 
$\Capa E<\kappa_d$) or much less probable (when $\Capa E\geq\kappa_d$). Moreover, if 
\eqref{RadiusBound} holds, then Theorems~\ref{t:CapacitiesInWc}--\ref{t:ComponentCounts} 
answer \refquestion{LargerSubset} \emph{simultaneously} over all possible sets $E'$.
The answer to \refquestion{Avoidance}, namely, that the Brownian motion follows a spatial 
avoidance strategy, will follow from \refprop{ExcursionNumbers} below. Finally, 
Theorems~\ref{t:CapacitiesInWc}--\ref{t:ShapeOfComponents}, 
Corollaries~\ref{c:VolumeLDP}--\ref{c:EvalLDP} and \refprop{NoWrapping} provide the answer 
to \refquestion{Components}.


\subsubsection{Two dimensions}
\lbsubsubsect{2d}

It remains a challenge to extend the results in the present paper to $d=2$ (see \reffig{d=2sim}). 
A law of large numbers for the $\epsilon$-cover time is derived in Dembo, Peres, Rosen and 
Zeitouni~\cite{DPRZ2004}:
\begin{equation}\label{CoverTimeLLN2d}
\lim_{\epsilon\decreasesto 0} \frac{\cC_\epsilon}{\psi_2(\epsilon)}=2 \quad \text{ a.s.},
\qquad 
\psi_2(\epsilon)=\frac{[\log(1/\epsilon)]^2}{\pi}.
\end{equation}
However, the relation $\psi_2(\epsilon(t))\sim\psi_2(\tilde{\epsilon}(t))$, where $\epsilon(t),\tilde{\epsilon}(t)\decreasesto 0$, no longer implies $\epsilon(t)\sim\tilde{\epsilon}(t)$: cf.\ \eqref{phipsiAsymptotics}. Hence the identity $\set{\rho_{\rm in}(t) > \epsilon} = \set{\cC_\epsilon > t}$ 
does not lead to a law of large numbers for the largest inradius $\rho_{\rm in}(t)$ itself, but only 
for its logarithm $\log \rho_{\rm in}(t)$:
\begin{equation}\label{Inradius2d}
\rho_{\rm in}(t) = e^{-\sqrt{\pi t/2} + o(\sqrt{t})}, \quad \frac{\log \rho_{\rm in}(t)}{\sqrt{t}} \to -\sqrt{\pi/2}, \qquad  t\to\infty. 
\end{equation}
In order to give a detailed geometric description, the error term $o(\sqrt{t})$ in \eqref{Inradius2d} would need to 
be controlled up to order $O(1)$.  Rough asymptotics for the logarithm of the average principal 
Dirichlet eigenvalue are conjectured in van den Berg, Bolthausen and den Hollander~\cite{vdBBdHpr}. 

In contrast to $d\geq3$, the large subsets of $\T^2 \setminus W[0,t]$ are expected to arise because 
of a \emph{temporal} avoidance strategy and to resemble the \emph{sparse} picture of 
\reffig{SparseVsDense}(a) (see Questions~\ref{q:LargerSubset}--\ref{q:Avoidance}). Furthermore, 
the Poisson point process heuristic, valid for $d \geq 3$ as explained in \refsubsubsect{UpDownDiscussion}, 
fails in $d=2$. The components of $\T^2 \setminus W[0,t]$ are expected to have a hierarchical 
structure, with long-range spatial correlations.


\section{Brownian excursions}
\lbsect{Preparations}

In this section we list a few properties of Brownian excursions that will be needed
as we go along. \refSubSect{Excursions} looks at the times and the numbers of excursions 
between the boundaries of two concentric balls, \refsubsect{BMCapa} estimates the hitting 
probabilities of these excursions in terms of capacity, while \refsubsect{ContCapa} 
collects a few elementary properties of capacity.


\subsection{Counting excursions between balls}
\lbsubsect{Excursions}

$\bullet$
Excursion times.
Let $x\in\T^d$ and $0<r<R<\tfrac{1}{2}$. Regard these values as fixed for the moment. 
Set $T_0=\inf\set{t\geq 0\colon\,W(t)\in\boundary B(x,R)}$ and, for $i\in\N$, define 
recursively the hitting times (see \reffig{ExcTim})
\begin{equation}
\label{hittimedef}
\begin{aligned}
T'_i &= \inf\set{t\geq T_{i-1}\colon\,W(t)\in\boundary B(x,r)},\\
T_i &= \inf\set{t\geq T'_i\colon\,W(t)\in\boundary B(x,R)}.
\end{aligned}
\end{equation}
We call $W[T'_i,T_i]$ the \emph{$i^\th$ excursion from $\boundary B(x,r)$ to $\boundary 
B(x,R)$}, and write $\xi'_i(x)=W(T'_i)$, $\xi_i(x)=W(T_i)$ for its starting and ending 
points.\footnote{If the starting point $x_0$ lies inside $B(x,R)$, then the 
Brownian motion may travel from $\boundary B(x,r)$ to $\boundary B(x,R)$ before time 
$T_0$. To simplify the application of Dembo, Peres and Rosen~\cite[Lemma 2.4]{DPR2003}, 
we do not call this an excursion from $\boundary B(x,r)$ to $\boundary B(x,R)$.} 

Set
\begin{equation}
\begin{aligned}
&\tau_0(x,r,R)=\tau'_0(x,r,R)=T_0(x),\\
&\tau_i(x,r,R)=T_i-T_{i-1},\,\tau'_i(x,r,R)=T'_i-T_{i-1}, \quad i\in\N.
\end{aligned}
\end{equation}
Thus, $\tau_i(x,r,R)$ is the duration of the $i^\th$ excursion from $\boundary B(x,R)$ 
to itself via $\boundary B(x,r)$, while $\tau'_i(x,r,R)<\tau_i(x,r,R)$ is the duration 
of the $i^\th$ excursion from $\boundary B(x,R)$ to $\boundary B(x,r)$.

\begin{figure}[htbp]
\begin{center}
\includegraphics{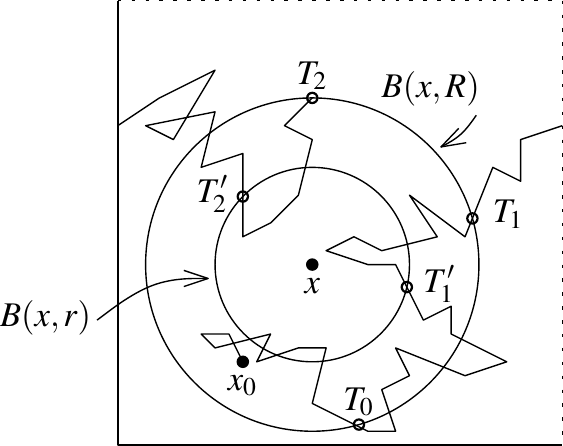}
\caption{Hittings that define the times $T_i$, $i\in\N_0$, and $T'_i$, $i\in\N$. The 
open circles indicate the locations of the starting and ending points $\xi'_i(x)=W(T'_i)$, 
$\xi_i(x)=W(T_i)$ of the excursions.}
\lbfig{ExcTim}
\end{center}
\end{figure}

\noindent
(All the variables $T_i,T'_i,\xi_i,\xi'_i,\tau_i,\tau'_i$ depend on all the parameters 
$x,r,R$. Nevertheless, in our notation we only indicate some of these dependencies.)

\medskip\noindent
$\bullet$
Excursion numbers.
Define
\begin{align}
N(x,t,r,R)
&= \max\set{i\in\N_0\colon\,T_i\leq t} 
= \max\set{j\in\N_0\colon\,\sum_{i=0}^j \tau_i(x,r,R) \leq t},\\
N'(x,t,r,R)
&= \max\set{j\in\N_0\colon\,\sum_{i=0}^j \tau'_i(x,r,R) \leq t}.
\end{align}
Thus, $N(x,t,r,R)$ is the number of completed excursions from $\boundary B(x,r)$ to 
$\boundary B(x,R)$ by time $t$, while $N'(x,t,r,R)$ is the number of (necessarily 
completed) excursions when the total time spent \emph{not} making an excursion 
reaches $t$.

As we will see in \refprop{ExcursionNumbers} below, $N(x,t,r,R)$ and $N'(x,t,r,R)$ have 
very similar scaling behaviour for $t\to\infty$ and $r\ll R\ll 1$. Indeed, the times 
$\tau_i(x,r,R)$ and $\tau'_i(x,r,R)$ are typically large (since the Brownian motion 
typically visits the bulk of $\T^d$ many times before travelling from $\boundary B(x,R)$ 
to $\boundary B(x,r)$), whereas $\tau_i(x,r,R)-\tau'_i(x,r,R)=T_i(x)-T'_i(x)$ scales as 
$R^2$. The advantage of $N'(x,t,r,R)$ is that it is independent of non-intersection events within 
$B(x,r)$ given the starting and ending points $\xi'_i(x)$,$ \xi_i(x)$ of the excursions.

Define
\begin{equation}
\label{NdDefinition}
N_d(t,r,R)= \frac{\kappa_d t}{r^{-(d-2)}-R^{-(d-2)}}.
\end{equation}
The following proposition shows that $N_d(t,r,R)$ represents the typical size for the 
random variables $N(x,t,r,R)$ and $N'(x,t,r,R)$.

\begin{proposition}
\lbprop{ExcursionNumbers}
For any $\delta \in (0,1)$ there is a $c=c(\delta)>0$ such that, uniformly in 
$x,x_0\in\T^d$, $t>1$ and $0<r^{1-\delta}\leq R\leq c$,
\begin{align}
\P_{x_0} \! \left( \big. N(x,t,r,R) \geq (1+\delta) N_d(t,r,R) \right) 
&\leq e^{-c N_d(t,r,R)},
\label{NUpwardBound}
\\
\P_{x_0} \! \left( \big. N'(x,t,r,R) \geq (1+\delta) N_d(t,r,R) \right) 
&\leq e^{-c N_d(t,r,R)},
\label{NprimeUpwardBound}
\\
\P_{x_0} \! \left( \big. N(x,t,r,R) \leq (1-\delta) N_d(t,r,R) \right) 
&\leq e^{-c N_d(t,r,R)}.
\label{NDownwardBound}
\end{align}
\end{proposition}

\begin{proof}
The result follows from a lemma in Dembo, Peres and Rosen~\cite{DPR2003}, which we 
reformulate in our notation.  (Note that the constant $\kappa_d$ defined by \eqref{kappadDefinition} 
corresponds to the quantity $1/\kappa_{\T^d}$ from~\cite[page 2]{DPR2003} rather than $\kappa_{\T^d}$.)

\begin{lemma}[{\rm \cite[Lemma 2.4]{DPR2003}}]
\lblemma{ExcursionTimes}
There is a constant $\eta>0$ such that if $N\geq \eta^{-1}$, $0<\delta<\delta_0<\eta$ and 
$0<2r\leq R<R_0(\delta)$, then for some $c=c(r,R)>0$ and uniformly in $x,x_0\in\T^d$,
\begin{equation}
\P_{x_0} \! \left( 1-\delta \leq \frac{\kappa_d}{N(r^{-(d-2)}-R^{-(d-2)})}
\sum_{i=0}^N \tau_i(x,r,R) \leq 1+\delta \right)
\geq 1-e^{-c\delta^2 N}. 
\end{equation}
Moreover, $c$ can be chosen to depend only on $\delta_0$ as soon as $R>r^{1-\delta_0}$.
The same result holds when $\tau'_i(x,r,R)$ is replaced by $\tau_i(x,r,R)$.
\end{lemma}
(The same result for $\tau'_i$ is not included in \cite{DPR2003}, but follows from the 
estimates in that paper. Indeed, $\tau_i-\tau'_i$ is shown to be an error term.)

To prove \refprop{ExcursionNumbers}, we begin with \eqref{NDownwardBound}. Fix 
$\delta>0$. We may assume without loss of generality that $\delta<\tfrac{1}{2}$ and 
$1/(1-\tfrac{1}{2}\delta)<1+\tfrac{2}{3}\delta<1+\eta$. Set $N=\floor{(1-\delta)N_d(t,r,R)}
+1$. Since $N/N_d(t,r,R) \to 1-\delta$ as $N_d(t,r,R)\to\infty$, we can choose $r$ small 
enough so that $\tfrac{1}{2}N_d(t,r,R)\leq N\leq (1-\tfrac{1}{2}\delta) N_d(t,r,R)$ and 
$N\geq\eta^{-1}$, uniformly in $R$ and $t>1$. We have 
\begin{equation}
\big\{N(x,t,r,R) \leq (1-\delta) N_d(t,r,R)\big\} 
= \set{N(x,t,r,R)<N} = \set{T_N\geq t}.
\end{equation} 
Since $T_N=\sum_{i=0}^N \tau_i(x,r,R)$, it follows that
\begin{align}
\P_{x_0} \! \big( N(x,t,r,R)\leq (1-\delta) N_d(t,r,R) \big)
&= \P_{x_0} \! \left( \sum_{i=0}^N \tau_i(x,r,R) \geq t \right)
= \P_{x_0} \! \left( \frac{\kappa_d\sum_{i=0}^N \tau_i(x,r,R)}{N(r^{-(d-2)}-R^{-(d-2)})} 
\geq \frac{N_d(t,r,R)}{N} \right)
\notag\\
&\leq \P_{x_0} \! \left( \frac{\kappa_d\sum_{i=0}^N \tau_i(x,r,R)}{N(r^{-(d-2)}-R^{-(d-2)})} 
\geq \frac{1}{1-\tfrac{1}{2}\delta} \right) \! .
\end{align}
Hence \eqref{NDownwardBound} follows from \reflemma{ExcursionTimes} with $\delta$ and $\delta_0$ 
replaced by $\tfrac{1}{2}\delta/(1-\tfrac{1}{2}\delta)$ and $\tfrac{2}{3}\delta$, respectively, 
with the constant $c$ in \refprop{ExcursionNumbers} chosen small enough so that $2r\leq R<R_0
[\tfrac{1}{2}\delta/(1-\tfrac{1}{2}\delta)]$.

The proof of \eqref{NprimeUpwardBound} is similar. Let $\delta>0$ be such that 
$\tfrac{1}{2}\delta/(1+\tfrac{1}{2}\delta)<\eta$ and set $N'=\ceiling{(1+\delta)N_d(t,r,R)}$. 
As before, we have 
\begin{equation}
\P_{x_0}\bigl(N'(t,x,r,R)\geq (1+\delta) N_d(t,r,R)\bigr) 
\leq \P_{x_0}\Bigl(\frac{\kappa_d\sum_{i=0}^{N'}\tau'_i(x,r,R)}{N'(r^{-(d-2)}-R^{-(d-2)})}
\leq \frac{1}{1+\tfrac12\delta} \Bigr)
\end{equation} 
and we can apply the version of \reflemma{ExcursionTimes} with $\tau'_i(x,r,R)$ instead of 
$\tau_i(x,r,R)$ and $\delta$ replaced by $\tfrac{1}{2}\delta/(1+\tfrac{1}{2}\delta)$.

Finally, because $N'(t,x,r,R) \leq N(t,x,r,R)$, \eqref{NUpwardBound} follows from 
\eqref{NprimeUpwardBound}.
\qed\end{proof}

\refprop{ExcursionNumbers} forms the link between the global structure of $\T^d$, notably 
the fact that a Brownian motion on $\T^d$ has a finite mean return time to a small ball, 
and the excursions of $W$ within small balls, during which $W$ cannot be distinguished from 
a Brownian motion on all of $\R^d$.


\subsection{Hitting sets by excursions}
\lbsubsect{BMCapa}

The concentration inequalities in \refprop{ExcursionNumbers} will allow us to treat the 
number of excursions as deterministic. This observation motivates the following definition.

\begin{definition}
\lbdefn{NSuccessful}
Let $0<r<R<\tfrac{1}{2}$, $\phi>0$ and $N\in\N$. A pair $(x,E)$ with $x\in\T^d$, $E\subset
\R^d$ Borel, will be called \emph{$(N,\phi,r,R)$-successful} if none of the first $N$ excursions 
of $W$ from $\boundary B(x,r)$ to $\boundary B(x,R)$ hit $x+\phi E$.
\end{definition}

\begin{proposition}
\lbprop{NSuccessfulProb}
Let $0<\epsilon<r<R<\tfrac{1}{2}$. Then, uniformly in $\phi>0$, $x_0,x\in\T^d$ and 
$E\subset\R^d$ a Borel set with $\phi E\subset B(0,\epsilon)$, and uniformly in 
$(\xi'_i(x),\xi_i(x))_{i=1}^N$,
\begin{equation}
\label{NSuccessfulProbFormula}
\begin{aligned}
&\P_{x_0} \! \condparenthesesreversed{ (x,E)
\text{\rm{ is $(N,\phi,r,R)$-successful}} }{(\xi'_i(x),\xi_i(x))_{i=1}^N}
\\
&\quad
= \exp\left[ -N \left( \frac{\phi}{r} \right)^{d-2} 
\frac{\Capa E }{\kappa_d} [1+o(1)] \right], \qquad r/\epsilon, R/r \to\infty.
\end{aligned}
\end{equation}
\end{proposition}

\noindent
Since the error term is uniform in $(\xi'_i(x),\xi_i(x))_{i=1}^N$, \refprop{NSuccessfulProb} 
also applies to the unconditional probability $\P_{x_0}( (x,E)\text{\rm{ is 
$(N,\phi,r,R)$-successful}} )$.  

To prove \refprop{NSuccessfulProb} we need the following lemma for the hitting probability 
of a single excursion given its starting and ending points. For $\xi'\in\boundary B(x,r)$, 
$\xi\in\boundary B(x,R)$, write $\P_{\xi',\xi}$ for the law of an excursion $W[0,\zeta_R]$, 
$\zeta_R=\inf\set{t\geq 0\colon\, d(x,W(t))\geq R}$, from $\boundary B(x,r)$ to $\boundary 
B(x,R)$, started at $\xi'$ and conditioned to end at $\xi$.

\begin{lemma}
\lblemma{CapacityAndHittingDistant}
Let $0<\epsilon<r<R<\tfrac{1}{2}$. Then, uniformly in $x\in\T^d$, $\xi'\in\boundary B(x,r),
\xi\in\boundary B(x,R)$ and $E$ a Borel set with $E\subset B(0,\epsilon)$,
\begin{equation}
\P_{\xi',\xi}((x+E)\intersect W[0,\zeta_R]\neq\emptyset) 
= \frac{\Capa E}{\kappa_d \, r^{d-2}}\,[1+o(1)],
\qquad r/\epsilon, R/r\to\infty.
\end{equation}
\end{lemma}

\reflemma{CapacityAndHittingDistant} is a more elaborate version of 
\eqref{CapacityAndHittingSimple}: it states that the asymptotics of 
\eqref{CapacityAndHittingSimple} remain valid when we stop the Brownian motion upon 
exiting a sufficiently distant ball, and hold conditionally and uniformly, provided the balls and the 
set are well separated. In the proof we use the relation
\begin{equation}
\label{CapacityAndHittingUniform}
\int_{\boundary B(0,r)} \P_x(E\intersect W\cointerval{0,\infty}\neq\emptyset) 
\,d\sigma_r(x)
= \frac{\Capa E}{\kappa_d \, r^{d-2}}, \qquad \text{$E$ a Borel subset of  $B(0,r)$},
\end{equation}
where $\sigma_r$ denotes the uniform measure on $\boundary B(0,r)$. Equation 
\eqref{CapacityAndHittingUniform} becomes an identity as soon as $B(0,r)$ contains 
$E$, and as such it is a more precise version of \eqref{CapacityAndHittingSimple}: 
see Port and Stone~\cite[Chapter 3, Theorem 1.10]{PS1978} and surrounding material. 

We defer the proof of \reflemma{CapacityAndHittingDistant} to 
\refsubsect{CapaHittingDistantProof}. We can now prove \refprop{NSuccessfulProb}.

\begin{proof}
Conditional on their starting and ending points $(\xi'_i(x),\xi_i(x))_{i=1}^N$, the successive 
excursions from $\boundary B(x,r)$ to $\boundary B(x,R)$ are independent with laws 
$\P_{\xi'_i(x),\xi_i(x)}$.  Applying \reflemma{CapacityAndHittingDistant}, we have
\begin{align}
&\P_{x_0} \! \condparenthesesreversed{(x,E)\text{ is $(N,\phi,r,R)$-successful}}
{(\xi'_i(x),\xi_i(x))_{i=1}^N}\notag\\
&\quad = \prod_{i=1}^N \P_{\xi'_i(x),\xi_i(x)}((x+\phi E)\intersect W[0,\zeta_R] = \emptyset)
= \left( 1-\frac{\Capa(\phi E)}{\kappa_d \, r^{d-2}}\,[1+o(1)] \right)^N \! .
\label{NSuccessfulProduct}
\end{align}
Since $\Capa(\phi E)\leq\kappa_d\,\epsilon^{d-2} = o(r^{d-2})$ as $r/\epsilon\to\infty$, 
we can rewrite the right-hand side of \eqref{NSuccessfulProduct} as
\begin{equation}
\exp\left[-N\frac{\Capa(\phi E)}{\kappa_d\,r^{d-2}}\,[1+o(1)] \right],
\end{equation}
so that the scaling relation in \eqref{CapacityScaling} implies the claim.
\qed\end{proof}


\subsection{Properties of capacity}
\lbsubsect{ContCapa}

In this section we collect a few elementary properties of capacity.


\subsubsection{Continuity}

\begin{proposition}
\lbprop{CapacityContinuity}
Let $E$ denote a Borel subset of $\R^d$. 
\begin{enumerate}
\item
\lbitem{CapacityCompact}
If $E$ is compact, then $\Capa E_r\decreasesto\Capa E$ as $r\decreasesto 0$.
\item
\lbitem{CapacityOpen}
If $E$ is open, then $\Capa E_{-r}\increasesto\Capa E$ as $r\decreasesto 0$.
\item
\lbitem{CapacityCtyPoint}
If $E$ is bounded with $\Capa(\closure{E})=\Capa(\interior{E})$, then 
$\Capa E_r\decreasesto\Capa E$ and $\Capa E_{-r}\increasesto\Capa E$ as 
$r\decreasesto 0$.
\end{enumerate}
\end{proposition}

\begin{proof}
For $r\decreasesto 0$ we have $E_r\decreasesto\closure{E}$ and $E_{-r}\increasesto\interior{E}$ 
for any set $E$. By Port and Stone~\cite[Chapter 3, Proposition 1.13]{PS1978}, it follows that 
$\Capa E_{-r} \increasesto \Capa(\interior{E})$ and, if $E$ is bounded, $\Capa E_r\decreasesto
\Capa(\closure{E})$. The statements about $E$ follow depending on which inequalities in 
$\Capa(\interior{E})\leq\Capa E\leq\Capa(\closure{E})$ are equalities.
\qed\end{proof}

\refprop{CapacityContinuity} is a statement about the continuity of $E\mapsto\Capa E$ with 
respect to enlargement and shrinking. The assumptions on $E$ are necessary, since there are 
sets $E$ with $\Capa(\closure{E})>\Capa(\interior{E})$. Note that $E\mapsto\Capa E$ is 
\emph{not} continuous with respect to the Hausdorff metric, even when restricted to 
reasonable classes of sets. For instance, the finite sets $B(0,1)\intersect\tfrac{1}{n}\Z^d$ 
converge to $B(0,1)$ in the Hausdorff metric, but have zero capacity for all $n$.


\subsubsection{Asymptotic additivity}

\begin{lemma}
\lblemma{SeparatedCapacity}
Let $0<\epsilon<r$. Then, uniformly in $x_1,x_2\in\R^d$ with $d(x_1,x_2)\geq r$ 
and $E^{(1)},E^{(2)}$ Borel subsets of $\R^d$ with $E^{(1)},E^{(2)}\subset B(0,\epsilon)$,
\begin{equation}
\Capa\big((x_1+E^{(1)})\union (x_2+E^{(2)})\big) 
= \big( \! \Capa E^{(1)} + \Capa E^{(2)}\big)\,[1-o(1)], 
\qquad r/\epsilon\to\infty.
\end{equation}
\end{lemma}

\begin{proof}
Fix $\tilde{r}$ large enough so that $(x_1+E^{(1)})\union (x_2+E^{(2)})\subset B(0,
\tilde{r})$. On the event $\set{W\text{ hits }x_j+E^{(j)}}$, write $Y_j$ for the 
first point of $x_j+E^{(j)}$ hit by $W$. Applying \eqref{CapacityOfUnion}, 
\eqref{CapacityAndHittingUniform}, and the Markov property, we get 
\begin{align}
0
&\leq \Capa(x_1+E^{(1)})+\Capa(x_2+E^{(2)})
-\Capa\big((x_1+E^{(1)})\union (x_2+E^{(2)})\big) \notag\\
&= \kappa_d \, \tilde{r}^{d-2} \int_{\boundary B(0,\tilde{r})} 
\P_x \! \left( W\text{ hits $x_1+E^{(1)}$ and }x_2+E^{(2)} \right) 
d\sigma_{\tilde{r}}(x) \notag\\
&\leq \sum_{\set{j,j'}=\set{1,2}} \kappa_d \, \tilde{r}^{d-2} \int_{\boundary B(0,\tilde{r})} 
\E_x \! \left( \indicator{W \text{ hits $x_j+E^{(j)}$}} 
\P_{Y_j}\big(W\text{ hits } B(x_{j'},\epsilon)\big) \right) d\sigma_{\tilde{r}}(x) \notag\\
&\leq \sum_{\set{j,j'}=\set{1,2}} \kappa_d \, \tilde{r}^{d-2} \int_{\boundary B(0,\tilde{r})} 
\P_x \! \left( W\text{ hits } x_j+E^{(j)} \right) 
\frac{\epsilon^{d-2}}{(r-\epsilon)^{d-2}}\, d\sigma_{\tilde{r}}(x) \notag\\
&= \frac{\epsilon^{d-2}}{(r-\epsilon)^{d-2}} \left( \Capa E^{(1)} +\Capa E^{(2)} \right),
\end{align}
where the second inequality uses that every $Y_j\in x_j+E^{(j)}$ is at least a distance 
$r-\epsilon$ from $x_{j'}$. But $(\epsilon/(r-\epsilon))^{d-2}=o(1)$ for $r/\epsilon
\downarrow 0$, and so the claim follows.
\qed\end{proof}


\section{Non-intersection probabilities for lattice animals}
\lbsect{LatticeAnimals}

An event such as 
\begin{equation}
\set{\exists x\in\T^d\colon\, (x+\phi_d(t)E)\intersect W[0,t]=\emptyset}
\end{equation} 
is a simultaneous statement about an infinite collection $(x+\phi_d(t) E)_{x\in\T^d}$ of 
sets. In this section, we apply the results of \refsect{Preparations} to prove simultaneous
statements for a finite collection of discretized sets, the lattice animals defined below. 
\refSubSect{HitLaLaAn} proves a bound for sets of large capacity that forms the basis 
for \refthm{CapacitiesInWc}, while \refsubsect{HitSmLaAn} proves bounds for sets of 
small capacity that form the basis for Theorems~\ref{t:ComponentCounts}--\ref{t:NoTranslates}.

\begin{definition}
A \emph{lattice animal} is a connected set $A\subset\R^d$ that is the union of a finite 
number of closed unit cubes with centres in $\Z^d$. We write $\cA^\boxempty$ for the 
collection of all lattice animals, and $\cA^\boxempty_Q$ for the collection of lattice 
animals $A\in\cA^\boxempty$ that contain $0$ and consist of at most $Q$ unit cubes.
\end{definition}

It is readily verified that, for any $d\geq 2$, there is a constant $C<\infty$ such that 
\begin{equation}
\label{LatticeAnimalGrowth}
\shortabs{\cA^\boxempty_Q} \leq e^{CQ}, \qquad Q\in\N.
\end{equation}
In fact, subadditivity arguments show that $|\cA^\boxempty_Q|$ grows exponentially, 
in the sense that $\lim_{Q\to\infty}|\cA^\boxempty_Q|^{1/Q}$ exists in $(1,\infty)$ 
for any $d\geq 2$. See, for instance, Klarner~\cite{K1967} for the case $d=2$, or Mejia 
Miranda and Slade~\cite[Lemma 2]{MejMirSlade2011} for a general upper bound that implies 
\eqref{LatticeAnimalGrowth}.

Lattice animals are commonly considered as discrete combinatorial objects. In our context, 
we can identify $A\in\cA^\boxempty$ with the collection $A\intersect\Z^d$ of lattice points 
in $A$. Requiring $A$ to be a connected subset of $\R^d$ is then equivalent to requiring 
the vertices $A\intersect\Z^d$ to form a connected subgraph of the lattice $\Z^d$. (Because 
of the details of our definition, the relevant choice of lattice structure is that vertices 
$x,y\in\Z^d$ are adjacent when their $\ell_\infty$-distance is $1$.)

For $n\in\N$, set $G_n=x+\tfrac{1}{n}\Z^d$ to be a \emph{grid} of $n^d$ points in $\T^d$, 
for some $x\in\T^d$. The choice of $x$ (i.e., the alignment of the grid) will generally 
not be relevant to our purposes.


\subsection{Large lattice animals}
\lbsubsect{HitLaLaAn}

\begin{proposition}
\lbprop{HitLatticeAnimal}
Fix an integer-valued function $t \mapsto n(t)$ such that 
\begin{equation}
\label{BoundOnCubeNumber}
\lim_{t\to\infty} \frac{n(t)\phi_d(t)}{(\log t)^{1/d}} = 0.
\end{equation}
Given $A\in\cA^\boxempty$, write $E(A)=n(t)^{-1}\phi_d(t)^{-1}A$.  Then, for each $\kappa$,
\begin{equation}
\label{HitLatticeAnimalProb}
\limsup_{t\to\infty}
\frac{\log\P_{x_0} \! \left( \exists x\in G_{n(t)}, 
A\in\cA^\boxempty\colon\, \Capa E(A)
\geq \kappa,(x+\phi_d(t)E(A))\intersect W[0,t]=\emptyset \right)}{\log t}
\leq J_d(\kappa).
\end{equation}
\end{proposition}

\refprop{HitLatticeAnimal} gives an upper bound on the probability of finding unhit sets 
of large capacity, simultaneously over all sets of the form $E(A)$, $A\in\cA^\boxempty$. 
Note that $x+\phi_d(t) E(A)$ is a finite union of cubes of side length $1/n(t)$ centred at points 
of $G_{n(t)}$.  In \refsect{ProofTheorems} we will use $x+\phi_d(t) E(A)$ as a lattice approximation 
to a generic set $x+\phi_d(t) E$. The fineness of this lattice approximation is determined by 
the relation between the lengths $1/n(t)$ and $\phi_d(t)$. The hypothesis in \eqref{BoundOnCubeNumber} 
means that the lattice scale $1/n(t)$ is a factor of order $o((\log t)^{1/d})$ smaller compared 
to the scale $\phi_d(t)$. This order is chosen so that the number of lattice animals does not 
grow too quickly.

Before proving \refprop{HitLatticeAnimal}, we give some definitions and make some remarks 
that we will use throughout \refsect{LatticeAnimals}. We abbreviate
\begin{equation}
\label{phint}
\phi=\phi_d(t), \qquad n=n(t), \qquad E(A)=n^{-1}\phi^{-1} A.
\end{equation}
For $x\in\T^d$, we introduce the nested balls $B(x,r)$ and $B(x,R)$, where
\begin{equation}
\label{rRDefinition}
r=\phi^{1-\delta}, \qquad R=\phi^{1-2\delta},
\end{equation}
and $\delta\in(0,\frac{1}{2})$ is fixed. We have $\phi\ll r\ll R\to 0$ as $t\to\infty$, and 
we will always take $t$ large enough so that $\phi<1$ and $R<\frac{1}{2}$.

Suppose $\kappa\in(0,\infty)$ is given and consider the collection of lattice animals $A\in\cA^\boxempty$ such that $\Capa E(A)\leq \kappa$. By \eqref{CapaVolEval}, 
it follows that $\Vol E(A)$ is uniformly bounded. Consequently, we may assume that such a lattice animal $A$ consists of 
at most $Q=Q(t)$ unit cubes, where $Q$ is suitably chosen with
\begin{equation}
\label{QBound}
Q=O(n^d\phi^d).
\end{equation}
Suppose, instead, that $A\in\cA^\boxempty$ is minimal subject to the condition $\Capa E(A)
\geq\kappa$, and suppose that $n\phi\to\infty$. By \eqref{CapacityOfUnion}, upon removing 
a single unit cube from $A$ the capacity $\Capa E(A)$ decreases by at most $O(1/n^{d-2}\phi^{d-2})$, 
and so it follows that $\kappa\leq\Capa E(A)\leq \kappa + O(1/n^{d-2}\phi^{d-2})$. In particular, 
$\Capa E(A)$ is uniformly bounded for $t$ sufficiently large, and we may again assume \eqref{QBound}.  

In what follows, we will always work in a context where one of these two assumptions applies.  We will therefore always assume that $A$ consists of at most $Q$ cubes, where $Q$ satisfies \eqref{QBound}.

Given $x\in G_n$ and $A\in\cA^\boxempty$, the translate $x+\phi E(A)$ can be written 
as $x'+\phi E(A')$, where $x'\in G_n$ and $0\in A'$. By the above, we have $A'\in
\cA^\boxempty_Q$. Since $A'$ is connected and $0\in A'$, it follows that $\phi E(A')\subset 
B(0,\phi Q\sqrt{d})$. If $Q=t^{o(1)}$ (in particular, if \eqref{BoundOnCubeNumber} is assumed, 
or the weaker hypothesis in \eqref{WeakBoundOnCubeNumber}), then $r/\phi Q\to\infty$ as 
$t\to\infty$. We may therefore always take $t$ large enough so that $B(0,\phi Q\sqrt{d})
\subset B(0,r)$, and we may apply \refprop{NSuccessfulProb} to $\phi E(A)$, uniformly over 
$A\in \cA_Q^\boxempty$.

\begin{proof}
Note that if we replace $n$ by a suitable 
multiple $kn=k(t)n(t)$ for $k(t)\in\N$, we can only increase the probability in 
\eqref{HitLatticeAnimalProb}.  Thus it is no loss of generality to assume that $n\phi\to\infty$. 

The event that $W$ hits $x+\phi E(A)$ is decreasing in $A$. Therefore we may restrict our 
attention to lattice animals $A$ that are minimal subject to $\Capa E(A)\geq\kappa$. By 
the remarks above, we may assume that $A\in\cA^\boxempty_Q$. Combining \eqref{BoundOnCubeNumber} 
and \eqref{QBound}, we have $Q=o(\log t)$.  

Set $N=(1-\delta)N_d(t,r,R)$.  Recalling \eqref{NdDefinition} and \eqref{rRDefinition}, we 
have $N_d(t,r,R)= t^{\delta+o(1)}$ as $t\to\infty$. If the event in \eqref{HitLatticeAnimalProb} 
occurs, then there must exist a point $x\in G_n$ with $N(x,t,r,R) < N$ or a pair $(x,A)\in 
G_n\times \cA^\boxempty_Q$ such that $\Capa E(A)\geq \kappa$ and $(x,E(A))$ is $(\floor{N},
\phi,r,R)$-successful. Write $\tilde{\chi}^\boxempty$ for the number of such pairs. Then
\begin{align}
&\P_{x_0} \! \left( \exists x\in G_n, A\in\cA^\boxempty\colon\,\Capa E(A)
\geq \kappa, (x+\phi E(A))\intersect W[0,t]=\emptyset \right)\notag\\
&\quad\leq \abs{G_n} \max_{x\in G_n} \P_{x_0}(N(x,t,r,R)<N) 
+ \P_{x_0}(\tilde{\chi}^\boxempty\geq 1)\notag\\
&\quad\leq t^{d/(d-2)+o(1)} e^{-c t^{\delta+o(1)}} + \P_{x_0}(\tilde{\chi}^\boxempty\geq 1)
\end{align}
by \refprop{ExcursionNumbers}.  The first term in the right-hand side is negligible. For the 
second term, $Q=o(\log t)$ implies that $|\cA_Q^\boxempty|\leq e^{O(Q)}=t^{o(1)}$ by 
\eqref{LatticeAnimalGrowth}, and so \refprop{NSuccessfulProb} gives
\begin{align}
\E(\tilde{\chi}^\boxempty)
&\leq \abs{G_n}\abs{\cA_Q^\boxempty} \max_{x\in G_n, A\in\cA_Q^\boxempty}
\P_{x_0}((x,E(A))\text{ is $(\floor{N} \! ,\phi,r,R)$-successful})\notag\\
&\leq (t^{d/(d-2) +o(1)}) (t^{o(1)}) (t^{-d\kappa/[(d-2)\kappa_d]+O(\delta)})
\leq t^{-d(\kappa/\kappa_d-1)/(d-2)+O(\delta)},
\label{ExpectedSuccessfulAnimals}
\end{align}
and the Markov inequality completes the proof.
\qed\end{proof}

\refprop{HitLatticeAnimal} bounds the probability that a single rescaled lattice animal 
$x+\phi_d(t)E(A)$ is not hit. We will also need the following bounds, for finite unions 
of lattice animals that are relatively close, and for pairs of lattice animals that are 
relatively distant.

\begin{lemma}
\lblemma{HitFiniteUnionOfLAs}
Assume \eqref{BoundOnCubeNumber}. Fix a capacity $\kappa\geq\kappa_d$, a positive integer 
$k\in\N$ and a positive function $t\mapsto h(t)>0$ satisfying 
\begin{equation}
\label{BoundOnAnimalSeparation}
\lim_{t\to\infty}\frac{\log(h(t)/\phi_d(t))}{\log t}=0.
\end{equation}
Then the probability that there exist a point $x\in G_{n(t)}$ and lattice animals $A^{(1)},
\dotsc,A^{(k)}\in\cA^\boxempty$, such that the union $E=\union_{j=1}^k E(A^{(j)})$ satisfies 
$\Capa E\geq\kappa$, $\phi_d(t)E\subset B(0,h(t))$, and $(x+\phi_d(t)E) \intersect W[0,t]
=\emptyset$, is at most $t^{-I_d(\kappa)+o(1)}$.
\end{lemma}

\begin{proof}
The proof is the same as for \refprop{HitLatticeAnimal}. Abbreviate $h=h(t)$. Since $h=t^{o(1)}
\phi$, it follows that $r/h\to\infty$ as $t\to\infty$, so that \refprop{NSuccessfulProb} 
applies to $\phi E$. Similarly, writing $A^{(j)}=y_j+\tilde{A}^{(j)}$ with $\tilde{A}^{(j)}
\in\cA_Q^\boxempty$ and $y_j\in B(0,nh)\intersect\Z^d$, we have that there are at most 
$O((nh)^{dk})|\cA_Q^\boxempty|^k$ possible choices for $A^{(1)},\dotsc,A^{(k)}$. This 
number is $t^{o(1)}$ by \eqref{BoundOnCubeNumber} and \eqref{BoundOnAnimalSeparation}, so 
that a counting argument applies as before.
\qed\end{proof}

\begin{lemma}
\lblemma{LargeDistantComponents}
Assume \eqref{BoundOnCubeNumber}. Fix a positive function $t \mapsto h(t)>0$ satisfying 
\begin{equation}
\label{BoundOnAnimalPairSeparation}
\liminf_{t\to\infty} \frac{h(t)}{\phi_d(t) \log t} > 0,
\end{equation}
and let $\kappa^{(1)},\kappa^{(2)}>\kappa_d$, $x_1\in\T^d$. Then the probability that there 
exist a point $x_2\in G_{n(t)}$ with $d(x_1,x_2)\geq h(t)$ and lattice animals $A^{(1)},A^{(2)}
\in\cA^\boxempty$ with $\Capa E(A^{(j)})\geq \kappa^{(j)}$ such that $(x_j+\phi_d(t)E(A^{(j)}))
\intersect W[0,t]=\emptyset$, $j=1,2$, is at most $t^{-[d \kappa^{(1)}/(d-2)\kappa_d]
-I_d(\kappa^{(2)})+o(1)}$.
\end{lemma}

\begin{proof}
We resume the notation and assumptions from the proof of \refprop{HitLatticeAnimal}, this 
time taking $\delta<\tfrac{1}{4}$. Abbreviate $h=h(t)$.

For $x_2\in G_n$ such that $d(x_1,x_2)\geq 2R$, the events of $(x_j,E(A_j))$ being 
$(\floor{N},\phi,r,R)$-successful, $j=1,2$, are conditionally independent given 
$(\xi'_i(x_j),\xi_i(x_j))_{i,j}$. The required bound for the case $d(x_1,x_2)\geq 2R$ 
therefore follows by the same argument as in the proof of 
\refprop{HitLatticeAnimal}.

For $x_2\in G_n$ such that $d(x_1,x_2)\leq 2R$, set $\tilde{r}=\phi^{1-3\delta}$, $\tilde{R}
=\phi^{1-4\delta}$ and $\tilde{N}=(1-\delta)N_d(t,\tilde{r},\tilde{R})$. We have $\phi E(A_j)
\subset B(0,\phi Q\sqrt{d})$ for $j=1,2$, with $Q=o(\log t)$ (without loss of generality, as 
in the proof of \refprop{HitLatticeAnimal}). Write $x_2=x_1+\phi y$, where $y\in\R^d$ with 
$h/\phi \leq d(0,y) \leq 2R/\phi$. The hypothesis \eqref{BoundOnAnimalPairSeparation} implies 
that $h/\phi Q\to\infty$. Hence we can apply \reflemma{SeparatedCapacity} (with $\epsilon
=\phi Q\sqrt{d}$ and $h$ playing the role of $r$), to conclude that 
\begin{equation}
\label{CapacityTwoAnimals}
\Capa\big( E(A_1)\union (y+E(A_2)) \big) = \big(\Capa E(A_1)+\Capa E(A_2)\big)[1-o(1)].
\end{equation}
We also have $E(A_1) \union (y+E(A_2))\subset B(0,2R+\phi Q\sqrt{d})$ with $\tilde{r}/R,\tilde{r}
/\phi Q\to\infty$. In particular, $x_1+\phi(E(A_1)\union (y+E(A_2)))\subset B(x_1,\tilde{r})$ 
for $t$ large enough. As in the proof of \refprop{HitLatticeAnimal}, $(x_j+\phi E(A_j))
\intersect W[0,t]=\emptyset$ implies that $N(x_1,t,\tilde{r},\tilde{R})<N$ or $(x_1,E(A_1)
\union(y+E(A_2)))$ is $(\floor{\smash{\tilde{N}}\big.},\phi,\tilde{r},\tilde{R})$-successful. 
By \eqref{CapacityTwoAnimals} and \refprop{NSuccessfulProb},
\begin{align}
&
\P_{x_0} \! \left( \big(x_1,E(A_1)\union(y+E(A_2))\big)\text{ is 
$(\floor{\smash{\tilde{N}}\big.},\phi,\tilde{r},\tilde{R})$-successful} \right)
\\&\quad
\leq \exp\left[ -\tilde{N}(\phi/\tilde{r})^{d-2}
(\kappa^{(1)}+\kappa^{(2)}-o(1))/\kappa_d \right],
\end{align}
and the rest of the proof is the same as for \refprop{HitLatticeAnimal}.
\qed\end{proof}


\subsection{Small lattice animals}
\lbsubsect{HitSmLaAn}

The bound in \refprop{HitLatticeAnimal} is only meaningful when $\kappa>\kappa_d$. For 
$\kappa<\kappa_d$, there are likely to be many unhit sets of capacity $\kappa$, and the two propositions that follow will quantify this statement.  

For $E\subset\R^d$, write $\chi(t,n(t),E)$ 
for the number of points $x\in G_{n(t)}$ such that $(x+\phi_d(t) E)\intersect W[0,t]=\emptyset$, 
and write $\chi^{\rm disjoint}(t,n(t),E)$ for the maximal number of disjoint translates 
$x+\phi_d(t) E$ such that $x\in G_{n(t)}$ and $(x+\phi_d(t) E)\intersect W[0,t]=\emptyset$. 
For $\kappa>0$, define
\begin{equation}
\begin{aligned}
\chi_+^\boxempty(t,n(t),\kappa)
&= \sum_{\substack{A\in\cA^\boxempty\colon\, 0\in A, \\ 
\Capa E(A)\geq\kappa}} \chi(t,n(t),E(A)),\\
\chi_-^\boxempty(t,n(t),\kappa)
&= \min_{\substack{A\in\cA^\boxempty\colon\, \\ 
\Capa E(A)\leq\kappa}} \chi^{\rm disjoint}(t,n(t),E(A)).
\end{aligned}
\end{equation}

\begin{proposition}
\lbprop{AnimalCounts}
Fix an integer-valued function $t \mapsto n(t)$ satisfying condition \eqref{BoundOnCubeNumber} 
such that $\lim_{t\to\infty} n(t)\phi_d(t)=\infty$. Then, for $0<\kappa<\kappa_d$,
\begin{equation}
\lim_{t\to\infty} \frac{\log\chi_+^\boxempty(t,n(t),\kappa)}{\log t} = J_d(\kappa), 
\quad
\lim_{t\to\infty} \frac{\log\chi_-^\boxempty(t,n(t),\kappa)}{\log t} = J_d(\kappa),
\quad \mbox{in $\P_{x_0}$-probability}.
\end{equation}
\end{proposition}

\begin{proposition}
\lbprop{HitManyLatticeAnimals}
Fix an integer-valued function $t \mapsto n(t)$ and a non-negative function $t \mapsto h(t)$ satisfying 
\begin{equation}
\label{WeakBoundOnCubeNumber}
\lim_{t\to\infty} \frac{\log [n(t)\phi_d(t)]}{\log t} = 0, \qquad 
\lim_{t\to\infty} \frac{\log[h(t)/\phi_d(t)]}{\log t} \leq 0,
\end{equation} 
and collections of points $(S(t))_{t>1}$ in $\T^d$ such that $\max_{x\in\T^d} d(x,S(t)) 
\leq h(t)$ for all $t>1$. Given $A\in\cA^\boxempty$, write $E(A)=n(t)^{-1}\phi_d(t)^{-1}A$. Then, for each $\kappa
\in(0,\kappa_d)$,
\begin{equation}
\P_{x_0} \! \left( \exists A\in\cA^\boxempty\colon\,\Capa E(A)\leq \kappa \text{ and }
(x+\phi_d(t)E(A))\intersect W[0,t] \neq \emptyset \; \forall x\in S(t) \right)
\leq \exp\left[ -t^{J_d(\kappa)-o(1)} \right] \! .
\end{equation}
\end{proposition}

Compared to \refthm{NoTranslates}, \refprop{HitManyLatticeAnimals} requires $(x+\phi_d(t)E(A))\intersect W[0,t] \neq \emptyset$ only for $x$ in some subset $S(t)$ of the torus, subject to the requirement that $S(t)$ should be within distance $h(t)$ of every point in $\T^d$.  The reader may assume that $S(t)=\T^d$, $h(t)=0$ for simplicity.

In \refprop{HitManyLatticeAnimals}, the scale $n(t)$ of the lattice need only satisfy 
\eqref{WeakBoundOnCubeNumber} instead of the stronger condition \eqref{BoundOnCubeNumber}.
This reflects the difference in scaling between the probabilities in 
\refprop{HitManyLatticeAnimals} compared to \refprop{HitLatticeAnimal}.


\subsubsection{Proof of \refprop{AnimalCounts}}
\lbsubsubsect{AnimalCountsProof}

\begin{proof}
Let $\delta\in(0,\tfrac{1}{2})$ be given. It suffices to show that $t^{J_d(\kappa)-O(\delta)}
\leq \chi_-^\boxempty(t,n,\kappa)$ and $\chi_+^\boxempty(t,n,\kappa)\leq t^{J_d(\kappa)
+O(\delta)}$ with high probability. (Given $\kappa<\kappa'$, the assumption $n\phi\to\infty$ 
implies the existence of some $A$ with $\kappa\leq\Capa E(A)\leq\kappa'$, and therefore 
$\chi_-^\boxempty(t,n,\kappa')\leq \chi_+^\boxempty(t,n,\kappa)$.)

For the upper bound, recall $N$ and $\tilde{\chi}^\boxempty$ from the proof of 
\refprop{HitLatticeAnimal}. On the event $\{N(x,t,r,R)<N \;\forall\,x\in G_n\}$ 
(whose probability tends to $1$) we have $\chi_+^\boxempty(t,\kappa,n) \leq 
\tilde{\chi}^\boxempty$. From \eqref{ExpectedSuccessfulAnimals} it follows that 
$\tilde{\chi}^\boxempty \leq t^{J_d(\kappa)+O(\delta)}$ with high probability.

For the lower bound, let $\set{x_1,\dotsc,x_K}$ denote a maximal collection of points in 
$G_n$ satisfying $d(x_j,x_k)>2R$ for $j\neq k$, so that $K=R^{-d+o(1)}=t^{d/(d-2)-O(\delta)}$.
Write $N_-=(1+\delta)N_d(t,r,R)$. By \refprop{ExcursionNumbers}, in the same way as in the 
proof of \refprop{HitLatticeAnimal}, $N(x_j,t,r,R)\leq N_-$ for each $j=1,\dotsc,K$, with 
high probability. Moreover we may take $t$ large enough so that $\phi E(A)\subset B(0,R)$, 
so that the translates $x_j+\phi E(A)$ are disjoint.  Let $\tilde{\chi}_-^\boxempty(A)$ 
denote the number of points $x_j$, $j\in \set{1,\dotsc,K}$, such that $(x_j,E(A))$ is 
$(\ceiling{N_-}\!,\phi,r,R)$-successful. We have $\chi_-^\boxempty(t,n(t),E(A)) \geq 
\tilde{\chi}_-^\boxempty(A) - 1$ on the event $\set{N(x_j,t,r,R)\leq N_- \: \forall j}$, 
since at most one translate $x_j+\phi E(A)$ may have been hit before the start of the 
first excursion, in the case $x_0\in B(x_j,R)$. On the other hand, since the balls 
$B(x_j,R)$ are disjoint, the excursions are conditionally independent given the starting 
and ending points $(\xi'_i(x_j),\xi_i(x_j))_{i,j}$. It follows that, for each $A$ with 
$\Capa E(A)\leq\kappa$, $\tilde{\chi}_-^\boxempty(A)$ is stochastically larger than a 
Binomial$(K,p)$ random variable, where $p\geq t^{-d\kappa/(d-2)-O(\delta)}$ by 
\refprop{NSuccessfulProb}. A straightforward calculation shows that $\P(\text{Binomial}
(K,p)<\frac{1}{2}Kp)\leq e^{-cKp}$ for some $c>0$, so that 
\begin{equation}
\P_{x_0}(\tilde{\chi}_-^\boxempty(A) \leq t^{J_d(\kappa)-O(\delta)})
\leq \exp\left[ -ct^{J_d(\kappa)-O(\delta)} \right].
\end{equation}  
As in the proof of \refprop{HitLatticeAnimal}, there are at most $t^{o(1)}$ animals $A$ to 
consider, so a union bound completes the proof.
\qed\end{proof}

As with \reflemma{HitFiniteUnionOfLAs}, we may modify \refprop{AnimalCounts} to deal with 
a finite union of lattice animals.

\begin{lemma}
\lblemma{FiniteUnionOfLACounts}
Assume the hypotheses of \refprop{AnimalCounts}, let $k\in\N$, and let $t\mapsto h(t)>0$ 
be a positive function satisfying \eqref{BoundOnAnimalSeparation}. Define 
\begin{equation}
\begin{aligned}
\chi^\boxempty_+(t,n(t),\kappa,k,h(t))
&=
\sum \chi(t,n(t),E),
\\
\chi^\boxempty_-(t,n(t),\kappa,k,h(t))
&=
\min \chi^{\rm disjoint}(t,n(t),E),
\end{aligned}
\end{equation}
where the sum and minimum are over sets $E=\union_{j=1}^k E(A^{(j)})$ such that $\phi_d(t) 
E\subset B(0,h(t))$; $(x+\phi_d(t)E)\intersect W[0,t]=\emptyset$; and $\Capa E\geq\kappa$ 
(for $\chi^\boxempty_+$) or $\Capa E\leq\kappa$ (for $\chi^\boxempty_-$), respectively. 
Then $(\log \chi^\boxempty_+(t,n(t),\kappa,k,h(t)))/\log t$ and $(\log \chi^\boxempty_-(t,n(t),
\kappa,k,h(t)))/\log t$ converge in $\P_{x_0}$-probability to $J_d(\kappa)$ as $t\to\infty$.
\end{lemma}


\subsubsection{Proof of \refprop{HitManyLatticeAnimals}}

The proof of \refprop{AnimalCounts} compares $\chi_-^\boxempty(t,n(t),\kappa)$ to a random 
variable that is approximately Binomial $(t^{d/(d-2)},t^{-d\kappa/(d-2)})$.  If this 
identification were exact, then the asymptotics in \refprop{HitManyLatticeAnimals} would 
follow in a similar way. However, the bound for each individual probability $\P_{x_0}
(N(x_j,t,r,R)\geq (1+\delta)N_d(t,r,R))$, $j=1,\dotsc,K$, although relatively small, is 
still much larger than the probability in \refprop{HitManyLatticeAnimals}. Therefore
an additional argument is needed.

\begin{proof}
Abbreviate $h=h(t),S=S(t)$.

Recall that the condition $\Capa E(A)\leq\kappa$ implies that $A$ consists of at 
most $Q$ cubes, where because of \eqref{QBound} and \eqref{WeakBoundOnCubeNumber} 
we have $Q=t^{o(1)}$. Fix such an $A$, and write $A=p+A'$, where $p\in\Z^d$ and 
$A'\in\cA^\boxempty_Q$. In particular, $E(A')\subset B(0,Q\sqrt{d})$. Since $x+\phi E(A)
=x+\tfrac{1}{n}p+\phi E(A')$, we can assume by periodicity that $p\in\set{0,\dotsc,n-1}^d$.

Let $\delta\in(0,\tfrac{1}{3})$, take $r,R$ as in \eqref{rRDefinition}, and choose $\tilde{n}
=\tilde{n}(t)\in\N$ such that $1/\tilde{n}= \phi^{1-3\delta+o(1)}$ and $1/\tilde{n}\geq 2R$. 
Let $\set{\tilde{x}_1,\dotsc,\tilde{x}_{\tilde{n}^d}}$ denote a grid of points in $\T^d$ with 
spacing $1/\tilde{n}$ (i.e., a translate of $G_{\tilde{n}}$), chosen in such a way that 
$d(x_0,\tilde{x}_j)>R$. To each grid point $\tilde{x}_j$, $j=1,\dotsc,\tilde{n}^d$, associate 
in some deterministic way a point $x_j\in S$ with $d(x_j+\tfrac{1}{n}p,\tilde{x}_j)
=d(x_j,\tilde{x}_j-\tfrac{1}{n}p)\leq h$ (this is always possible by the hypothesis on 
$S$). The choice of $\tilde{x}_j,x_j$ depends on $t$, but we suppress this dependence 
in our notation.  

Since $h/\phi\leq t^{o(1)}$, we have $r/h\geq \phi^{-\delta+o(1)}\to\infty$.  Since also 
$r/\phi Q\to\infty$, we may take $t$ large enough so that $h+\phi Q\sqrt{d}<r<R<1/
\tilde{n}$, implying that $x_j+\phi E(A)=x_j+\tfrac{1}{n}p+E(A')\subset B(\tilde{x}_j,r)$ 
for $j=1,\dotsc,\tilde{n}^d$, and so we can apply \reflemma{CapacityAndHittingDistant} to 
the sets $x_j+\phi E(A)$, uniformly in the choice of $A$ and $j$.

Let $\sigma(s)$ be the total amount of time, up to time $s$, during which the Brownian 
motion is \emph{not} making an excursion from $\boundary B(\tilde{x}_j,r)$ to $\boundary 
B(\tilde{x}_j,R)$ for any $j=1,\dotsc,\tilde{n}^d$.  In other words, $\sigma(s)$ is the 
Lebesgue measure of $[0,s] \setminus ( \union_{j=1}^{\tilde{n}^d} \union_{i=1}^\infty 
[T'_i(\tilde{x}_j),T_i(\tilde{x_j})])$. Define the stopping time $T''=\inf\set{s\colon\,
\sigma(s)\geq t}$. Clearly, $T''\geq t$. Define $N''_j$ to be the number of excursions 
from $\boundary B(\tilde{x}_j,r)$ to $\boundary B(\tilde{x}_j,R)$ by time $T''$, and write 
$(\xi'_i(\tilde{x}_j),\xi_i(\tilde{x}_j))_{i=1,\dotsc,N''_j}$ for the starting and ending 
points of these excursions.

If $(x+\phi E(A)) \intersect W[0,t]\neq\emptyset$ for each $x\in S$, then necessarily, 
for each $j=1,\dotsc,\tilde{n}^d$, at least one of the $N''_j$ excursions from $\boundary
B(\tilde{x}_j,r)$ to $\boundary B(\tilde{x}_j,R)$ must hit $x_j+\phi E(A)$. (Here we use 
that $d(x_0,\tilde{x}_j)>R$, which implies that the Brownian motion cannot hit $x_j+\phi 
E(A)$ before the start of the first excursion.) These excursions are conditionally 
independent given $(\xi'_i(\tilde{x}_j),\xi_i(\tilde{x}_j))$ for $i=1,\dotsc,N''_j, 
j=1,\dotsc,\tilde{n}^d$. Applying \reflemma{CapacityAndHittingDistant} and 
\eqref{CapacityScaling}, we get
\begin{align}
&\P_{x_0} \! \condparenthesesreversed{ (x+\phi E(A))\intersect W[0,t]\neq\emptyset \; 
\forall x\in S}{(N''_j)_j, 
(\xi'_i(\tilde{x}_j),\xi_i(\tilde{x}_j))_{i,j}}\notag\\
&\quad\leq
\P_{x_0} \! \condparenthesesreversed{ (x_j+\phi E(A))\intersect W[0,T'']\neq\emptyset \; 
\forall j }{ (N''_j)_j, (\xi'_i(\tilde{x}_j),\xi_i(\tilde{x}_j))_{i,j} }\notag\\
&\quad=
\prod_{j=1}^{\tilde{n}^d} \left( 1-\prod_{i=1}^{N''_j} 
\left( 1-\frac{\phi^{d-2}\Capa E(A)}{\kappa_d \, r^{d-2}}(1+o(1)) \right) \right)\notag\\
&\quad\leq
\exp \left[ \sum_{j=1}^{\tilde{n}^d} \log\left( 1-(1-(\phi/r)^{d-2} 
(\kappa/\kappa_d+o(1)))^{N''_j} \right) \right].
\end{align}
In this upper bound, which no longer depends on $(\xi'_i(\tilde{x}_j),\xi_i(\tilde{x}_j))_{i,j}$, 
the function $y \mapsto \log(1-e^{cy})$ is concave, and hence we can replace each $N''_j$ by 
the empirical mean $\bar{N}''=\tilde{n}^{-d} \sum_{j=1}^{\tilde{n}^d} N''_j$:
\begin{align}
&\P_{x_0} \! \condparenthesesreversed{ (x+\phi E(A))\intersect W[0,t]\neq\emptyset \; 
\forall x\in S }{ (N''_j)_j }\notag\\
&\quad\leq
\exp \left( \tilde{n}^d \log\left( 1-(1-(\phi/r)^{d-2} 
(\kappa/\kappa_d+o(1)))^{\bar{N}''} \right) \right)\notag\\
&\quad\leq
\exp \left[ -\tilde{n}^d (1-(\phi/r)^{d-2} (\kappa/\kappa_d+o(1)))^{\bar{N}''} \right].
\end{align}
Write $M=(1+\delta)N_d(t,r,R)$. On the event $\set{\bar{N}''\leq M}$, the relations 
$(\phi/r)^{d-2} M\sim (1+\delta)d(d-2)^{-1}\log t$ and $\tilde{n}^d=t^{d/(d-2)-O(\delta)}$ 
imply that
\begin{align}
&\indicator{\bar{N}''\leq M}\P_{x_0} \! \condparenthesesreversed{ (x+\phi E(A))\intersect 
W[0,t]\neq\emptyset \; \forall x\in S }{ (N''_j)_j }\notag\\
&\quad\leq
\exp\left[ -t^{d/(d-2)-O(\delta)} 
\exp\left[ -(\phi/r)^{d-2} M(\kappa/\kappa_d +o(1)) \right] \right]\notag\\
&\quad=
\exp\left[ -t^{J_d(\kappa)-O(\delta)} \right].
\label{HitAllSetsBarNSmall}
\end{align}

Next, we will show that $\P_{x_0}(\bar{N}''\geq M)\leq \exp[-ct^{d/(d-2)-O(\delta)}]$. To 
that end, let $\pi^{(\tilde{n})}$ denote the projection map from the unit torus $\T^d$ to 
a torus of side length $1/\tilde{n}$. Under $\pi^{(\tilde{n})}$, every grid point 
$\tilde{x}_j$ maps to the same point $\pi^{(\tilde{n})}(\tilde{x}_j)$, and $\sigma(s)$ is 
the total amount of time the projected Brownian motion $\pi^{(\tilde{n})}(W)$ in 
$\pi^{(\tilde{n})}(\T^d)$ spends \emph{not} making an excursion from $\boundary 
B(\pi^{(\tilde{n})}(\tilde{x}_j),r)$ to $\boundary B(\pi^{(\tilde{n})}(\tilde{x}_j),R)$, 
by time $s$. Moreover, $\tilde{n}^d \bar{N}'' = \sum_{j=1}^{\tilde{n}^d} N''_j$ can be 
interpreted as the number of such excursions in $\pi^{(\tilde{n})}(\T^d)$ completed by 
time $T''$.  

Write $x\mapsto \tilde{n}x$ for the dilation that maps the torus $\pi^{(\tilde{n})}(\T^d)$ 
of side length $1/\tilde{n}$ to the unit torus $\T^d$. By Brownian scaling, $(\tilde{W}(u)
)_{u\geq 0} = (\tilde{n}\pi^{(\tilde{n})}(W(\tilde{n}^{-2}u)))_{u\geq 0}$ has the law of a 
Brownian motion in $\T^d$. Moreover, $\tilde{n}^d\bar{N}''$ can be interpreted as the number 
of excursions of $\tilde{W}(u)$ from $\boundary B(\tilde{n}\pi^{(\tilde{n})}(\tilde{x_j}),
\tilde{n} r)$ to $\boundary B(\tilde{n}\pi^{(\tilde{n})}(\tilde{x_j}),\tilde{n} R)$ until 
the time spent not making such excursions first exceeds $\tilde{n}^2 t$, i.e., precisely 
the quantity $N'(\tilde{n}\pi^{(\tilde{n})}(\tilde{x}_j),\tilde{n}^2 t,\tilde{n}r,\tilde{n}R)$ 
from \refsubsect{Excursions}. We have $N_d(\tilde{n}^2 t,\tilde{n}r,\tilde{n}R)=\tilde{n}^d 
N_d(t,r,R)$, so \refprop{ExcursionNumbers} gives
\begin{align}
\P_{x_0}(\bar{N}''\geq M)
&= \P_{x_0}(\tilde{n}^d \bar{N}'' \geq \tilde{n}^d M)
= \P_{\tilde{n}\pi^{(\tilde{n})}(x_0)} \! \left( N'(\tilde{n}
\pi^{(\tilde{n})}(\tilde{x}_j),\tilde{n}^2 t,\tilde{n}r,\tilde{n}R) 
\geq \tilde{n}^d M \right)\notag\\
&= \P_{\tilde{n}\pi^{(\tilde{n})}(x_0)} \! \left( N'(\tilde{n}
\pi^{(\tilde{n})}(\tilde{x}_j),\tilde{n}^2 t,\tilde{n}r,\tilde{n}R) 
\geq (1+\delta)N_d(\tilde{n}^2 t,\tilde{n}r,\tilde{n}R) \right)\notag\\
&\leq \exp\left[ -c N_d(\tilde{n}^2 t,\tilde{n}r,\tilde{n}R) \right]
= \exp\left[ -c t^{d/(d-2)-O(\delta)} \right] \! .
\label{BarNNotLarge}
\end{align}

Equations \eqref{HitAllSetsBarNSmall}--\eqref{BarNNotLarge} imply that, for each fixed 
$A=p+A'$ with $\Capa E(A)\leq \kappa$, we have
\begin{equation}
\P_{x_0} \! \left( (x+\phi E(A))\intersect W[0,t]\neq \emptyset \; \forall x\in S \right) 
\leq \exp\left[ -t^{d(1-\kappa/\kappa_d)/(d-2)-O(\delta)} \right].
\end{equation}  
But the number of pairs $(p,A')$ is at most $n^d |\cA^\boxempty_Q|= t^{d/(d-2)+o(1)} 
e^{O(Q)}$, by \eqref{LatticeAnimalGrowth} and \eqref{WeakBoundOnCubeNumber}. Since 
$Q=t^{o(1)}$, a union bound completes the proof.
\qed\end{proof}


\section{Proofs of Theorems \ref{t:CapacitiesInWc}--\ref{t:ShapeOfComponents}
and Propositions \ref{p:NoWrapping}--\ref{p:Connected}}
\lbsect{ProofTheorems}

In proving Theorems~\ref{t:CapacitiesInWc}--\ref{t:ShapeOfComponents}, we bound 
non-intersection probabilities for Wiener sausages, e.g.
\begin{equation}
\P\left( \exists x\in\T^d\colon\, (x+\phi_d(t) E)\intersect W_{\rho(t)}[0,t]
=\emptyset \right), \qquad E\subset\R^d,
\end{equation}
in terms of the Brownian non-intersection probabilities estimated in Propositions~\ref{p:HitLatticeAnimal}
and \ref{p:AnimalCounts}--\ref{p:HitManyLatticeAnimals}, in which $E$ is a rescaled lattice 
animal. In \refsubsect{ApproxByLA} we prove an approximation lemma for lattice animals, which 
leads directly to the proofs of Theorems~\ref{t:CapacitiesInWc}--\ref{t:NoTranslates} and 
\refprop{NoWrapping}.  Proving \refthm{ShapeOfComponents} requires an additional argument 
to show that a component containing a given set is likely to be not much larger, and we 
prove this in \refsubsect{ShapeProof}.  Finally, in \refsubsect{ConnectedProof} we give the 
proof of \refprop{Connected}.


\subsection{Approximation by lattice animals}
\lbsubsect{ApproxByLA}

\begin{lemma}
\lblemma{SetsAndAnimals}
Let $\rho>0$ and $n\in\N$ satisfy $\rho n \geq 2\sqrt{d}$, and let $\phi>0$. Then, 
given a bounded connected set $E\subset\R^d$, there is an $A\in\cA^\boxempty$ such 
that $E(A)=n^{-1}\phi^{-1} A$ satisfies $E\subset E(A)\subset E_{\rho/\phi}$ and, 
for any $x\in\T^d$, $0\leq\tilde{\rho}\leq\tfrac{1}{4}\rho$,
\begin{align}
\label{SetInclusion}
x+\phi E\subset x'+ \phi & E(A) \subset x+(\phi E)_\rho \qquad\text{for some }x'\in G_n,\\
\label{MissSetAndAnimal}
\big\{(x+\phi E)\intersect W_\rho[0,t]=\emptyset\big\}
\subset&
\set{\exists x'\in G_n\colon\,(x'+\phi E(A))\intersect W[0,t]=\emptyset},\\
\label{HitSetAndAnimal}
\big\{(x+\phi E)\intersect W_{\tilde\rho}[0,t]\neq\emptyset\big\}
\subset&
\set{(x+\phi E(A))\intersect W[0,t]\neq\emptyset}.
\end{align}
\end{lemma}

\begin{proof}
Let $A$ be the union of all the closed unit cubes with centres in $\Z^d$ that intersect 
$n\phi E_{\rho/4\phi}$. This set is connected because $E$ is connected, and therefore 
$A\in\cA^\boxempty$. Every cube in $A$ is within distance $\sqrt{d}$ of some point of 
$n\phi E_{\rho/4\phi}$, so that $E\subset E_{\rho/4\phi}\subset E(A)\subset 
E_{\rho/4\phi+\sqrt{d}/n\phi}$. By assumption, $\sqrt{d}/n\leq\rho/2$, so that 
$E(A)\subset E_{3\rho/4\phi}\subset E_{\rho/\phi}$ (see \reffig{SetAndAnimal}(a)).

\begin{figure}[htbp]
{
\hfill
(a)
\includegraphics{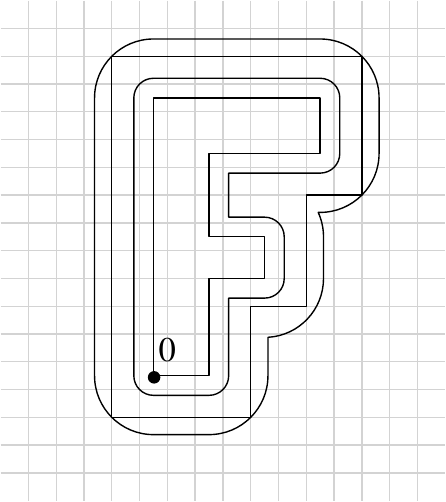}
\hfill
(b)
\includegraphics{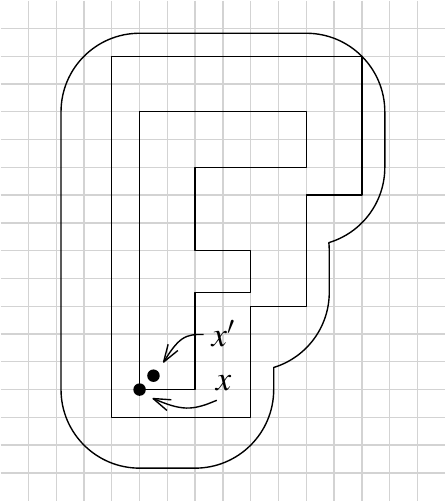}
\hfill
}
\caption{(a) From inside to outside: an F-shaped set $E$; the enlargement $E_{\rho/4\phi}$; 
$E(A)$, the union of the rescaled cubes intersecting $E_{\rho/4\phi}$; the bounding set 
$E_{3\rho/4\phi}$. The grid shows the cubes in the definition of $E(A)$, rescaled to have 
side length $1/n\phi$. The parameters $\rho,n$ satisfy $\rho n =2\sqrt{d}$. (b) From inside 
to outside (scaled by $\phi$ compared to part (a)): the prospective subset $x+\phi E$ of 
$\T^d\setminus W_\rho[0,t]$; the approximating grid-aligned set $x'+\phi E(A)$; the taboo 
set $x+(\phi E)_\rho$ that the Brownian motion must not visit.}
\lbfig{SetAndAnimal}
\end{figure}

Given $x\in\T^d$, let $x'\in G_n$ satisfy $d(x,x')\leq\sqrt{d}/2n$. Then $x+\phi E\subset 
x'+(\phi E)_{\sqrt{d}/2n}\subset x'+\phi E(A)\subset x+(\phi E(A))_{\sqrt{d}/2n}\subset x
+(\phi E)_\rho$ since $\sqrt{d}/2n\leq \rho/4$ and $\phi E(A)\subset (\phi E)_{3\rho/4}$. 
See \reffig{SetAndAnimal}(b). This proves \eqref{SetInclusion}; \eqref{MissSetAndAnimal} 
follows immediately because $(x+\phi E)\intersect W_\rho[0,t]=\emptyset$ is equivalent to 
$(x+(\phi E)_\rho) \intersect W[0,t]=\emptyset$.  

Similarly, since $(\phi E)_{\rho/4}\subset \phi E(A)$ and since $(x+\phi E)\intersect 
W_{\tilde\rho}[0,t]\neq\emptyset$ is equivalent to $(x+(\phi E)_{\tilde\rho})\intersect 
W[0,t]\neq\emptyset$, the inclusion in \eqref{HitSetAndAnimal} follows.
\qed\end{proof}


\subsubsection{Proof of \refthm{NoTranslates}}

In this section we prove the following theorem, of which \refthm{NoTranslates} is the special 
case with $S(t)=\T^d$.

\begin{theorem}
\lbthm{HitManySets}
Fix non-negative functions $t \mapsto \rho(t)$ and $t \mapsto h(t)$ satisfying 
\begin{equation}
\label{phihcond}
\lim_{t\to\infty} \frac{\rho(t)}{\phi_d(t)} = 0, \qquad 
\lim_{t\to\infty} \frac{\log[h(t)/\phi_d(t)]}{\log t} \leq 0,
\end{equation} 
and collections of points $(S(t))_{t>1}$ in $\T^d$ such that $\max_{x\in\T^d} d(x,S(t)) 
\leq h(t)$ for all $t>1$. Then, for any $E\subset\R^d$ compact with $\Capa E <\kappa_d$,
\begin{equation}
\label{AvoidAll}
\log \P\left( (x+\phi_d(t) E)\intersect W_{\rho(t)}[0,t]\neq\emptyset \; 
\forall x\in S(t) \right) \leq -t^{J_d(\Capa E)+o(1)}, \qquad t\to\infty.
\end{equation}
\end{theorem}

\begin{proof}
Fix $E\subset\R^d$ compact with $\Capa E<\kappa_d$, and let $\delta>0$ be arbitrary with 
$\Capa E +\delta <\kappa_d$. By \refprop{CapacityContinuity}\refitem{CapacityCompact}, we 
can choose $r>0$ so that $\Capa(E_r)\leq \Capa E+\tfrac{1}{2}\delta$. If $E_r$ is not 
already connected, then enlarge it to a connected set $E'\supset E_r$ by adjoining a 
finite number of line segments (this is possible because $E_r$ is the $r$-enlargement 
of a compact set). Doing so does not change the capacity, so we may apply 
\refprop{CapacityContinuity}\refitem{CapacityCompact} again to find $r'>0$ so that 
$\Capa((E')_{r'})\leq\Capa E+\delta$.

Define $\rho_0(t)=r'\phi_d(t)$ and $n(t)=\ceiling{2\big.\smash{\sqrt{d}}
/\rho_0(t)}$, so that $\rho_0(t)n(t)\geq 2\sqrt{d}$ and the condition 
\eqref{WeakBoundOnCubeNumber} from \refprop{HitManyLatticeAnimals} holds. Since 
$\rho(t)/\phi_d(t)\to 0$, we may choose $t$ sufficiently large so that $\rho(t)\leq 
\tfrac{1}{4}\rho_0(t)$.

Apply \reflemma{SetsAndAnimals} to $E'$ with $\rho=\rho_0(t)$, $\tilde\rho=\rho(t)$, and 
$\phi=\phi_d(t)$. Note that if $(x+\phi_d(t)E)\intersect W_{\rho(t)}[0,t]\neq\emptyset$ 
for all $x\in S(t)$, then $(x+\phi_d(t)E(A))\intersect W[0,t]\neq\emptyset$ for all 
$x\in S(t)$, where $\Capa E(A)\leq\Capa((E')_{\rho/\phi})=\Capa((E')_{r'})\leq\Capa E+\delta$. 
By \refprop{HitManyLatticeAnimals} with $\kappa=\Capa E +\delta$, this event has 
a probability that is at most $\exp[-t^{J_d(\Capa E)-O(\delta)}]$, and taking $\delta\decreasesto 0$ 
we get the desired result.
\qed\end{proof}


\subsubsection{Proof of \refthm{CapacitiesInWc}}

\begin{proof}
First consider $\kappa<\kappa_d$. Since $I_d(\kappa)$ is infinite for such $\kappa$, it 
suffices to show that $\lim_{t\to\infty} \log\P(\kappa^*(t,\rho(t))\leq \kappa\phi^{d-2})
/\log t=-\infty$. Let $\kappa<\kappa'<\kappa_d$, and take $E$ to be a ball of capacity 
$\kappa'$.  If $\kappa^*(t,\rho(t)) \leq \kappa\phi^{d-2}$, then no translate $x+\phi_d(t)E$, 
$x\in\T^d$, can be a subset of $\T^d\setminus W_{\rho(t)}[0,t]$. Applying \refthm{NoTranslates}, we conclude that $\P(\kappa^*(t,\rho(t))\leq \kappa\phi^{d-2})\leq 
\exp[-t^{J_d(\kappa)+o(1)}]$, which implies the desired result.

Next consider the LDP upper bound for $\kappa\geq\kappa_d$. Since $\kappa \mapsto I(\kappa)$ 
is increasing and continuous on $[\kappa_d,\infty]$, it suffices to show that $\P(\kappa^*
(t,\rho(t))\geq\kappa\phi^{d-2})\leq t^{-I_d(\kappa)+o(1)}$ for $\kappa>\kappa_d$. Therefore, 
suppose that $x+\phi_d(t) E\subset\T^d\setminus W_{\rho(t)}[0,t]$ for some $x\in\T^d$ and 
$E\subset\R^d$ compact with $\Capa E\geq\kappa$. As in the proof of \refthm{HitManySets}, 
define $n(t)=\ceiling{2\big.\smash{\sqrt{d}}/\rho(t)}$. \reflemma{SetsAndAnimals} gives 
$(x'+\phi_d(t) E(A))\intersect W[0,t]=\emptyset$ for some $x'\in G_{n(t)}$ and $\Capa E(A)
\geq\Capa E\geq\kappa$. The condition in \eqref{RadiusBound} on $\rho(t)$ implies the 
condition in \eqref{BoundOnCubeNumber} on $n(t)$, and therefore we may apply 
\refprop{HitLatticeAnimal} to conclude that $\P(\kappa^*(t,\rho(t))\geq\kappa\phi^{d-2})
\leq t^{-I_d(\kappa)+o(1)}$.

Finally, the LDP lower bound for $\kappa\geq\kappa_d$ will follow (with $E$ the ball of 
capacity $\kappa$, say) from the lower bound proved for \refthm{ShapeOfComponents} (see 
\refsubsect{ShapeProof}).
\qed\end{proof}


\subsubsection{Proof of \refthm{ComponentCounts}}

\begin{proof}
As in the proof of \refthm{CapacitiesInWc}, the lower bound will follow from the more 
specific lower bound proved for \refthm{ShapeOfComponents} (see \refsubsect{ShapeProof}).

Choose $n(t)$ such that $n(t)\geq 2\sqrt{d}/\rho(t)$ and the hypotheses of \refprop{AnimalCounts} 
hold. (The conditions on $n(t)$ are mutually consistent because $2\sqrt{d}/\rho(t)=O(1/\phi_d(t))$.)  
Given any component $C$ containing a ball of radius $\rho(t)$ and having the form $C=x+\phi_d(t)E$ 
for $\Capa E\geq\kappa$, apply \reflemma{SetsAndAnimals} to find $x'_C\in G_{n(t)}$ and 
$A_C\in\cA^\boxempty$ such that $C\subset x'_C+\phi_d(t)E(A_C)\subset C_{\rho(t)}\subset 
\T^d\setminus W[0,t]$. The pairs $(x'_C,E(A_C))$ so constructed must be distinct: for 
$C'\neq C$, we have $x'_{C'}+\phi_d(t)E(A_{C'})\subset C'_{\rho(t)} \subset (\T^d\setminus 
C)_{\rho(t)}=\T^d\setminus C_{-\rho(t)}$, and since $C_{-\rho(t)}$ is non-empty by assumption, 
it follows that $C\nsubseteq x'_{C'}+\phi_d(t)E(A_{C'})$. We therefore conclude that $\chi_{\rho(t)}
(t,\kappa)\leq \chi_+^\boxempty(t,n(t),\kappa)$, so the required upper bound follows from 
\refprop{AnimalCounts}.
\qed\end{proof}


\subsubsection{Proof of \refprop{NoWrapping}}

\begin{proof}
Abbreviate $\phi=\phi_d(t),\rho=\rho(t)$. It suffices to bound the probability that 
$\T^d\setminus W_\rho[0,t]$ has a component of diameter at least $\tfrac{1}{2}$, since 
the mapping $x+y\mapsto y$ from $B(x,r)\subset\T^d$ to $B(0,r)\subset\R^d$ is a well-defined 
local isometry if $r<\tfrac{1}{2}$. 

Suppose that $x\in\T^d\setminus W_\rho[0,t]$ belongs to a connected component intersecting 
$\boundary B(x,\tfrac{1}{2})$. Then there is a bounded connected set $E\subset\R^d$ such 
that $(x+\phi E)\intersect W_\rho[0,t]$ and $E\intersect\boundary B(0,\tfrac{1}{2}\phi^{-1})
\neq\emptyset$ (see \reffig{Diameter1}). Define $n=n(t)=\ceiling{2\big.\smash{\sqrt{d}/\rho}}$ 
and apply \reflemma{SetsAndAnimals} to conclude that $(x'+\phi E(A))\intersect W[0,t]=\emptyset$ 
with $E\subset E(A)$, $A\in\cA^\boxempty$, $x'\in G_n$. Since $E(A)$ contains $E$, it has 
diameter at least $\tfrac{1}{2}\phi^{-1}$, so $A$ has diameter at least $\tfrac{1}{2}n$ and 
must consist of at least $n/(2\sqrt{d})$ unit cubes. Since $\rho=o(\phi)$ and $\phi=t^{-d/(d-2)
+o(1)}$, we have $n\geq t^{d/(d-2)+o(1)}$. The hypothesis in \eqref{RadiusBound} implies that 
$n\phi=o((\log t)^{1/d})$, as in condition \eqref{BoundOnCubeNumber} from \refprop{HitLatticeAnimal}.
Therefore $\Vol E(A) \geq (n\phi)^{-d}n/(2\sqrt{d})\geq t^{d/(d-2)+o(1)}$, and in particular 
$\Vol E(A)\to\infty$. By \eqref{CapaVolEval}, $\Capa E(A)\to\infty$ also. Thus, if $\T^d\setminus 
W_\rho[0,t]$ has a component of diameter at least $\tfrac{1}{2}$, then the event in 
\refprop{HitLatticeAnimal} occurs, with $\kappa$ arbitrarily large for $t\to\infty$. By 
\refprop{HitLatticeAnimal}, the probability of this occuring is negligible, as claimed.
\qed\end{proof}

\begin{figure}[htbp]
\begin{center}
\includegraphics{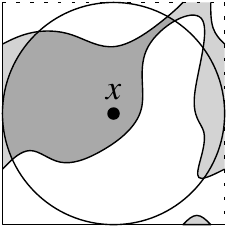}
\caption{A large connected component of $\T^d\setminus W_\rho[0,t]$ that is not isometric 
to a subset of $\R^d$ (shading) and a possible choice of the set $x+\phi E$ (dark shading).}
\lbfig{Diameter1}
\end{center}
\end{figure}

This proof is unchanged if the radius $\tfrac{1}{2}$ is replaced by any $\delta\in(0,\tfrac{1}{2})$, which shows that the maximal diameter $D(t,\rho(t))$ satisfies $D(t,\rho(t))\to 0$ in 
$\P$-probability when \eqref{RadiusBound} holds (see \refsubsubsect{DiameterDiscussion}).


\subsection{Proof of \refthm{ShapeOfComponents}}
\lbsubsect{ShapeProof}

In Theorems~\ref{t:CapacitiesInWc}--\ref{t:NoTranslates} we deal with components 
that contain a subset $x+\phi_d(t)E$ of a given form.  \refthm{ShapeOfComponents} adds the requirement that 
the component containing such a subset should not extend further than distance $\delta\phi_d(t)$ 
from $x+\phi_d(t)E$. In the proof, we will bound the probability that the component extends no 
further than distance $\rho(t)$ from $x+\phi_d(t)E$, but only for sets $E\in\cE^\boxempty_c$ of 
the following kind: define
\begin{equation}
\label{cEboxcDefn}
\cE^\boxempty_c = \set{\text{$E\in\cE_c$: $E=\tfrac{1}{n}A$ for some $A\in\cA^\boxempty$}}
\end{equation} 
to be the collection of sets in $\cE_c$ that are rescalings of lattice animals.  

Note that, unlike in \refsect{LatticeAnimals}, the scaling factor $\tfrac{1}{n}$ in 
\eqref{cEboxcDefn} is fixed and does not depend on $t$. We begin by showing that the collection 
$\cE^\boxempty_c$ is dense in $\cE_c$.

\begin{lemma}
\lblemma{cEboxcDense}
Given $E\in\cE_c$ and $\delta>0$, there exists $E^\boxempty\in\cE^\boxempty_c$ with $E\subset 
E^\boxempty\subset E_\delta$.
\end{lemma}

\reflemma{cEboxcDense} will allow us to prove \refthm{ShapeOfComponents} only for 
$E\in\cE^\boxempty_c$.

\begin{proof}
For $y\notin E$, define 
\begin{equation}
b(y)=\sup\set{r>0\colon\, y\text{ belongs to the unbounded component of }\R^d\setminus E_r}.
\end{equation}
Since $\R^d\setminus E$ is open and connected, $b(y)$ is continuous and positive on 
$\R^d\setminus E$. By compactness, we may choose $\eta\in(0,\delta)$ such that $b(y)>\eta$ 
for $y\notin E_\delta$. Apply \reflemma{SetsAndAnimals} (with $\rho$ and $\phi$ replaced 
by $\eta$ and 1, and $n$ sufficiently large) to find $E'=\tfrac{1}{n}A$ with $E\subset 
E'\subset E_\eta$. The set $E'$ is a rescaled lattice animal, but $\R^d\setminus E'$ might 
not be connected. However, if $y$ belongs to a bounded component of $\R^d\setminus E'$, then 
$b(y)\leq\eta$ by construction: since $E'\subset E_\eta$, $y$ cannot belong to the unbounded 
component of $\R^d\setminus E_\eta$. By choice of $\eta$, it follows that every bounded 
component of $\R^d\setminus E'$ is contained in $E_\delta$. Thus, if we define $E^\boxempty$ 
to be $E'$ together with these bounded components (see \reffig{LAwConnCompl}), then 
$E^\boxempty\in\cE_c^\boxempty$ and $E^\boxempty\subset E_\delta$, as claimed.
\qed\end{proof}

\begin{figure}[htbp]
\begin{center}
\includegraphics{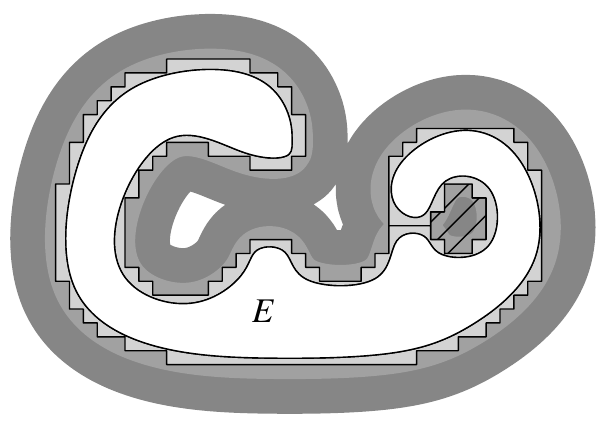}
\caption{A set $E$ (white) and its enlargement $E_\delta$ (dark shading). Every 
bounded component of $\R^d\setminus E_\delta$ can reach infinity without touching 
$E_\eta$ (medium shading). A set $E'$ (light shading) with $E\subset E'\subset E_\eta$ 
may disconnect a region from infinity (diagonal lines), but this region must belong to 
$E_\delta$.}
\lbfig{LAwConnCompl}
\end{center}
\end{figure}

In the proof of \refthm{ShapeOfComponents}, we adapt the concept of $(N,\phi,r,R)$-successful 
from \refdefn{NSuccessful} to formulate the desired event in terms of excursions. To this end 
we next introduce the sets and events that we will use. In the remainder of this section, we 
abbreviate $\phi=\phi_d(t)$, $\rho=\rho(t)$, $I_d(\kappa)=I(\kappa)$ and $J_d(\kappa)=J(\kappa)$.

Fix $E\in\cE_c^\boxempty$ and $\delta>0$. We may assume that $E\subset B(0,a)$ with $a>\delta$.
Let $\eta\in(0,\tfrac{1}{2})$ be small enough that $\kappa_d \eta^{d-2}<\Capa E$. Set 
$r=\phi^{1-\eta}$, $R=\phi^{1-2\eta}$, and let $\set{x_0,\dotsc,x_K}\subset\T^d$ denote a 
maximal collection of points in $\T^d$ satisfying $d(x_0,x_j)>R$ and $d(x_j,x_k)>2R$ for 
$j\neq k$, so that 
\begin{equation}
\label{KAsymptotics}
K= R^{-d-o(1)}
=t^{d/(d-2)+O(\eta)}
\end{equation}
Take $t$ large enough that $\rho<\tfrac{1}{2}\delta\phi$ and $R<\tfrac{1}{2}$. Set $N=(1+\eta)
N_d(t,r,R)$ (see \eqref{NdDefinition}).

Choose $q=q(t)$ with $q>2a+\delta$, $q\geq\log t$, and $q=(\log t)^{O(1)}$. Let $\set{y_1,
\dotsc,y_L}\subset B(0,2q)\setminus E_\delta$ denote a maximal collection of points in 
$B(0,2q)\setminus E_\delta$ satisfying $d(y_\ell,E)\geq \delta$, $d(y_\ell,y_m)\geq
\tfrac{1}{2}\rho/\phi$ for $\ell\neq m$, so that $L=O((q\phi/\rho)^d)=(\log t)^{O(1)}$ 
by \eqref{RadiusBound}.

(The collection $\set{y_1,\dotsc,y_L}$ will be used to ensure that a component containing 
$x_j+\phi E$ is contained in $x_j+\phi E_\delta$; see the event $F_3(j)$ below. The requirements 
on $q$ are chosen so that $L$ is suitably bounded, while also allowing us to apply 
\reflemma{LargeDistantComponents} to deal with components that are relatively far from $x_j$.)

Let $Z=\boundary(E_{\rho/\phi}) \union ( \union_{z\in B(0,2a)\intersect\eta\Z^d}
\boundary B(z,\eta) \setminus E_{\rho/\phi} )$ (see \reffig{SetZ}: $Z$ consists 
of a $(d-1)$-dimensional shell around $E$ together with a finite number of 
$(d-1)$-dimensional spheres). Let $\set{z_1,\dotsc,z_M}\subset Z$ denote a maximal 
collection of points in $Z$ with $d(z_m,z_p)\geq\tfrac{1}{2}\rho/\phi$ for $m\neq p$. 
Since $Z$ is $(d-1)$-dimensional, we have $M=O((\rho/\phi)^{d-1})$.

\begin{figure}[htbp]
\begin{center}
\includegraphics{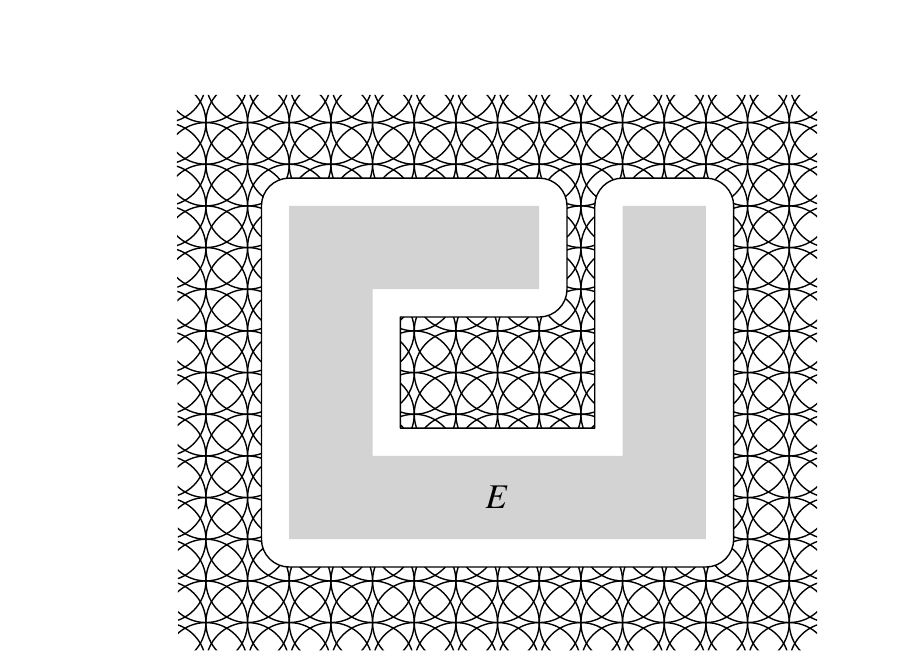}
\caption{The set $E$ (shaded) and part of the $(d-1)$-dimensional set $Z$.}
\lbfig{SetZ}
\end{center}
\end{figure}

For $j=1,\dotsc,K$, define the following events.

\begin{itemize}
\item[$\bullet$]
$F_1(j)=\set{\frac{1}{2}N\leq N(x_j,t,r,R)\leq N}$ is the event that $W$ makes between 
$\tfrac{1}{2}N$ and $N$ excursions from $\boundary B(x_j,r)$ to $\boundary B(x_j,R)$ 
by time $t$.
\item[$\bullet$]
$F_2(j)$ is the event that $(x_j,E_{\rho/\phi})$ is $(\floor{N},\phi,r,R)$-successful.
\item[$\bullet$]
$F_3(j)$ is the event that, for each $\ell=1,\dotsc,L$, the $i^\th$ excursion from 
$\boundary B(x_j,r)$ to $\boundary B(x_j,R)$ hits $x_j+B(\phi y_\ell, \tfrac{1}{2}\rho)$ 
for some $i=i(\ell)\in\set{1,\dotsc,\floor{N/4}}$.
\item[$\bullet$]
$F_4(j)$ is the event that, for each $m=1,\dotsc,M$, the $i^\th$ excursion from $\boundary 
B(x_j,r)$ to $\boundary B(x_j,R)$ hits $x_j+B(\phi z_m, \tfrac{1}{2}\rho)$ for some $i=i(m)
\in\set{\floor{N/4}+1,\dotsc,\floor{N/2}}$.
\item[$\bullet$]
$F_5(j)$ is the event that $\T^d\setminus W_\rho[0,t]$ contains no component of capacity 
at least $\phi^{d-2}\Capa E $ disjoint from $B(x_j,2q\phi)$.
\item[$\bullet$]
$F(j)=F_1(j)\intersect F_2(j)\intersect F_3(j)$.
\item[$\bullet$]
$F_{\rm max}(j)=F_1(j)\intersect F_2(j)\intersect F_3(j)\intersect F_4(j)\intersect F_5(j)$.
\end{itemize}

\begin{lemma}
\lblemma{FImplies}
On $F(j)$, the component of $\T^d\setminus W_\rho[0,t]$ containing $x_j+\phi E$ satisfies 
condition \eqref{CtrhoEEprime} with $E'=E_\delta$. Furthermore, $F_{\rm max}(j)\subset 
F_\rho(t,E,E_\delta)$ for $t$ sufficiently large.
\end{lemma}

\begin{proof}
Note that if $F_1(j)\intersect F_2(j)$ occurs, then $x_j+\phi E\subset \T^d\setminus W_\rho[0,t]$. 
If $F_1(j)\intersect F_3(j)$ occurs, then the set $x_j+\union_{\ell=1}^L B(\phi y_\ell,
\tfrac{1}{2}\rho)$ is entirely covered by the Wiener sausage. By choice of $\set{y_1,\dotsc,
y_L}$, this set contains $x_j+\left( B(0,2q\phi)\setminus \phi E_\delta \right)$, and 
consequently $\left( \T^d\setminus W_\rho[0,t] \right) \intersect B(x_j,2q\phi) \subset x_j
+\phi E_\delta$.

We have therefore shown that, on $F(j)$, $\T^d\setminus W_\rho[0,t]$ has a component containing 
$x_j+\phi E$ and satisfying condition \hyperref[CtrhoEEprime]{$\cC(t,\rho,E,E_\delta)$}. To show 
further that $F_{\rm max}(j)\subset F_\rho(t,E,E_\delta)$, we will show any other component must 
have capacity smaller than $\phi^{d-2}\Capa E$.  

If $F_1(j)\intersect F_4(j)$ occurs, then $x_j+\phi Z$ is entirely covered by the Wiener sausage, 
by choice of $\set{z_1,\dotsc,z_M}$. By choice of $Z$, all components of $B(x_j,a\phi)\setminus
(x_j+\phi Z)$, other than any components that are subsets of $x_j+\phi E_{\rho/\phi}=x_j+(\phi 
E)_\rho$, must be contained in a ball of radius $\eta\phi$, and in particular have capacity at 
most $\kappa_d(\eta\phi)^{d-2}<\phi^{d-2}\Capa E$.  

Finally, if $F_5(j)$ occurs, then the component of largest capacity cannot occur outside 
$B(x_j,2q\phi)$, and therefore must be the component of largest capacity contained in 
$x_j+(\phi E)_\rho$.

It therefore remains to show that the component of largest capacity in $x_j+(\phi E)_\rho$ 
is in fact the component containing $x_j+\phi E$. Suppose that $a\in x_j+\phi E$ is the 
centre of a $(d-1)$-dimensional ball of radius $\rho$ that is completely contained in some 
face of $x_j+\phi E$, and let $b$ be a point at distance at most $\rho$ from $a$ along the 
line perpendicular to the face (see \reffig{NearComponent}). If both $x_j+\phi E$ and $b$ 
are contained in $\T^d\setminus W_\rho[0,t]$, then so is the line segment from $a$ to $b$, 
so that $b$ belongs to the same component as $x_j+\phi E$.

\begin{figure}[htbp]
\begin{center}
\includegraphics{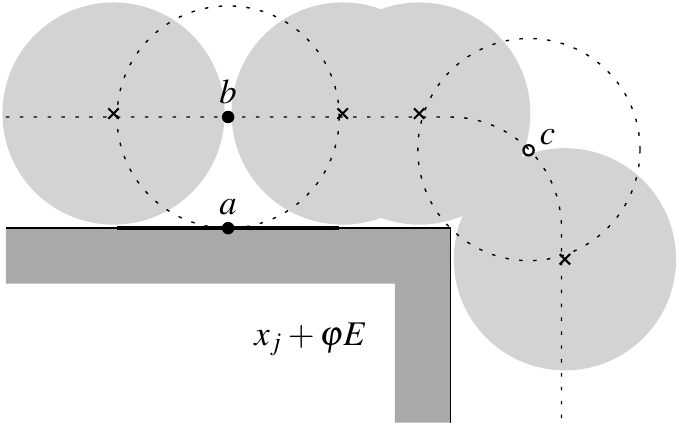}
\caption{A point $b$ near the centre $a$ of a ball (thicker line) on a face of $x_j+\phi E$, 
and a point $c$ near the boundary of a face. The Brownian path must not touch the dotted 
lines, but the Wiener sausage can fill the shaded circles by visiting the crossed points.
The point $c$ can belong to a different component than $x_j+\phi E$, but $b$ cannot.}
\lbfig{NearComponent}
\end{center}
\end{figure}

We therefore conclude that, on $F_{\rm max}(j)$, any point of $x_j+(\phi E)_\rho$ that is not 
in the same component as $x_j+\phi E$ must lie within distance $2\rho$ of the boundary of some 
face of $x_j+\phi E$. Write $H$ for the set of boundaries of faces of $E$. Since $H$ is 
\mbox{$(d-2)$}-dimensional, its capacity is $0$, and therefore $\Capa((\phi H)_{2\rho})=\phi^{d-2} 
\Capa(H_{2\rho/\phi})=o(\phi^{d-2})$ by \refprop{CapacityContinuity}\refitem{CapacityCompact}, 
since $\rho/\phi\to 0$. In particular, for $t$ sufficiently large the component of largest 
capacity in $x_j+(\phi E)_\rho$ must be the component containing $x_j+\phi E$, which completes 
the proof of \reflemma{FImplies}.
\qed\end{proof}

\begin{proof}[Proof of \refthm{ShapeOfComponents}]
Because of the upper bound proved for Theorems~\ref{t:CapacitiesInWc}--\ref{t:ComponentCounts}, 
we need only prove the lower bounds
\begin{equation}
\label{BasicLowerBoundF}
\P\left( F_\rho(t,E,E_\delta) \right)
\geq t^{-I(\Capa E)-o(1)},
\qquad \Capa E\geq\kappa_d, \delta>0.
\end{equation}
and
\begin{equation}
\label{BasicLowerBoundchi}
\chi_\rho(t,E,E_\delta)\geq t^{J(\Capa E)-o(1)}
\text{ with high probability,} \qquad\Capa E<\kappa_d,\delta>0.
\end{equation}
Moreover, it suffices to prove \eqref{BasicLowerBoundF}--\eqref{BasicLowerBoundchi} under the 
assumption that $E\in\cE_c^\boxempty$ and, in \eqref{BasicLowerBoundF}, that $\Capa E>\kappa_d$. 
Indeed, given any $\delta'\in(0,\tfrac{1}{2}\delta)$, apply \reflemma{cEboxcDense} to find 
$E^\boxempty\in\cE_c^\boxempty$ with $E\subset E^\boxempty\subset E_{\delta'}$. By adjoining, 
if necessary, a sufficiently small cube to $E^\boxempty$, we may assume that $\Capa E^\boxempty>
\Capa E$. Apply \eqref{BasicLowerBoundF}--\eqref{BasicLowerBoundchi} with $E$ and $\delta$ 
replaced by $E^\boxempty$ and $\delta'$, respectively. \refprop{CapacityContinuity}\refitem{CapacityCompact} 
implies that $\Capa E^\boxempty\decreasesto\Capa E$ as $\delta'\decreasesto 0$. Since $\kappa\mapsto
J(\kappa)$ is continuous, we conclude that the bounds for $E\in\cE_c$ follows from those for 
$E\in\cE^\boxempty_c$. 

We next relate the left-hand side of \eqref{BasicLowerBoundF} to the events $F_1(j),\dotsc,F_5(j)$. 
Noting that $F_1(j)\intersect F_2(j)\intersect F_1(k)\intersect F_2(k) \subset F_5(j)^c$ 
for $j\neq k$, \reflemma{FImplies} implies that
\begin{align}
\P\left( F_\rho(t,E,E_\delta) \right)
&\geq \sum_{j=1}^K \P_{x_0}(F(j))
\notag\\&
\geq \sum_{j=1}^K \P_{x_0}(F_2(j)\intersect F_3(j)\intersect F_4(j))
-\sum_{j=1}^K \P_{x_0}(F_1(j)^c) 
-\sum_{j=1}^K \P_{x_0}(F_1(j)\intersect F_2(j)\intersect F_5(j)^c).
\label{BoundByF}
\end{align}
We will bound each of the sums in the right-hand side of \eqref{BoundByF}.  

Applying \refprop{ExcursionNumbers} and \eqref{KAsymptotics} (and noting that $N_d(t,r,R)
=t^{\eta+o(1)}$ and that $\tfrac{1}{2}N/N_d(t,r,R)=\tfrac{1}{2}(1+\eta)<\tfrac{3}{4}$), we see 
that the second sum in the right-hand side of \eqref{BoundByF} is at most $t^{d/(d-2)+O(\eta)}
\exp[-c t^{\eta+o(1)}]$. This term will be negligible compared to the scale of 
\eqref{BasicLowerBoundF}.

For the last sum in \eqref{BoundByF}, we assume that $\Capa E>\kappa_d$ and use 
\reflemma{LargeDistantComponents}. Set $h(t)=2q\phi$, and note that $h(t)/(\phi \log t)\geq 1$ 
by assumption on $q$. If $F_1(j)\intersect F_2(j)\intersect F_5(j)^c$ occurs, then, by 
\reflemma{SetsAndAnimals}, there are lattice animals $A,A'\in\cA^\boxempty$ with $\Capa E(A),
\Capa E(A')\geq \Capa E$ and a point $x'\in\T^d\setminus B(x_j,2q\phi)$ with $(x_j+\phi E(A))
\intersect W[0,t]=(x'+\phi E(A'))\intersect W[0,t]=\emptyset$. By \reflemma{LargeDistantComponents} 
with $\kappa^{(1)}=\kappa^{(2)}=\Capa E$, we have
\begin{equation}
\P_{x_0}(F_1(j)\intersect F_2(j)\intersect F_5(j)^c)
\leq t^{-d\Capa(E)/[(d-2)\kappa_d]-I(\Capa E)+o(1)}.
\end{equation}  
Hence the last sum in \eqref{BoundByF} is at most $t^{-2I(\Capa E)+O(\eta)}$. Since $I(\Capa E)>0$, 
this term is also negligible, for $\eta$ sufficiently small, compared to the scale of 
\eqref{BasicLowerBoundF}. (This is the only part of the proof where $\Capa E>\kappa_d$ is used.)

We have therefore proved that \eqref{BasicLowerBoundF} will follow if we can give a suitable lower 
bound for the first sum on the right-hand side of \eqref{BoundByF}.  Using again the asymptotics 
\eqref{KAsymptotics} for $K$, \eqref{BasicLowerBoundF} will follow from
\begin{equation}
\label{F234Bound}
\P_{x_0}\condparenthesesreversed{F_2(j)\intersect F_3(j)
\intersect F_4(j)}{((\xi'_i(x_j),\xi_i(x_j))_{i=1}^N)_{j=1}^K} 
\geq t^{-d\Capa E/[(d-2)\kappa_d]-O(\eta)}.
\end{equation}
In fact, \eqref{F234Bound} also implies \eqref{BasicLowerBoundchi}.  On the event $\intersect_{j=1}^K 
F_1(j)$ (which occurs with high probability, by \refprop{ExcursionNumbers}), \reflemma{FImplies} 
implies that $\chi_\rho(t,E,E_\delta)$ is at least as large as the number of $j\in\set{1,\dotsc,K}$ 
for which $F_2(j)\intersect F_3(j)$ occurs. Since the events $F_2(j)\intersect F_3(j)$ are 
conditionally independent for different $j$ given the starting and ending points $((\xi'_i(x_j),
\xi_i(x_j))_{i=1}^N)_{j=1}^K$, \eqref{F234Bound} and \eqref{KAsymptotics} immediately imply that 
$\chi_\rho(t,E,E_\delta)\geq t^{J_d(\kappa)-O(\eta)}$ with high probability (cf.\ the proof of 
\refprop{AnimalCounts} in \refsubsubsect{AnimalCountsProof}).

It therefore remains to prove \eqref{F234Bound}. To do so, we will condition on not hitting 
$x_j+(\phi E)_\rho$ and use the following lemma to estimate the conditional probability of hitting 
small nearby balls. Note that, conditional on the occurence of $F_2(j)$ and the starting and 
ending points $(\xi'_i(x_j),\xi_i(x_j))_{i=1}^N$, the events $F_3(j)$ and $F_4(j)$ are independent.

\begin{lemma}
\lblemma{HitWhileAvoiding}
Fix $E\in\cE_c^\boxempty$ and $\delta>0$, and let $0<\rho<\phi<r<R<\tfrac{1}{2}$. Then there 
is an $\epsilon>0$ such that if $\rho/\phi<\epsilon$, $\phi/r<\epsilon$ and $r/R\leq\tfrac{1}{2}$, 
then, uniformly in $x\in\T^d$, $\xi'\in \boundary B(x,r)$, and $\xi\in\boundary B(x,R)$,
\begin{multline}
\P_{\xi',\xi} \! \condparentheses{(x+B(\phi y,\tfrac{1}{2}\rho))\intersect W[0,\zeta_R] 
\neq\emptyset}{(x+(\phi E)_\rho)\intersect W[0,\zeta_R] =\emptyset }
\\
\geq
\begin{cases}
\epsilon (\phi/r)^{d-2} (\rho/\phi)^{d-2}, & \text{if }y\in B(0,r/\phi)\setminus E_\delta,\\
\epsilon (\phi/r)^{d-2} (\rho/\phi)^\alpha, & \text{if }y\in E_\delta\setminus E_{\rho/\phi},
\end{cases}
\end{multline}
where $\alpha>d-2$ is some constant depending only on $d$.
\end{lemma}

We give the proof of \reflemma{HitWhileAvoiding} in \refsubsect{HitWhileAvoidingProof}.

The event $F_3(j)$ says that all $(x_j,B(y_\ell,\tfrac{1}{2}\rho/\phi))$, $\ell=1,\dotsc,L$, 
are \emph{not} $(\floor{N/4},\phi,r,R)$-successful. \reflemma{HitWhileAvoiding} implies (as 
in the proof of \refprop{NSuccessfulProb}) that, uniformly in $\ell$,
\begin{align}
&\P_{x_0} \! \condparentheses{(x_j,B(y_\ell,\tfrac{1}{2}\rho/\phi))
\text{ is $(\floor{N/4},\phi,r,R)$-successful}}{F_2(j)}\notag\\
&\quad\leq
\left(1-\epsilon (\rho/r)^{d-2}(\rho/\phi) \right)^{\floor{N/4}} 
= \exp\left[ -\epsilon \floor{N/4} (\phi/r)^{d-2} (\rho/\phi)^{d-1}(1+o(1))\right] \! .
\end{align}
Recalling \eqref{phidtDefinition} and \eqref{NdDefinition}, we have $N(\phi/r)^{d-2}\geq
(d/(d-2)+O(\eta))\log t$, so that
\begin{align}
&\P_{x_0} \! \condparentheses{\text{some $(x_j,B(y_\ell,\tfrac{1}{2}\rho/\phi))$ 
is $(\floor{N/4},\phi,r,R)$-successful}}{F_2(j)}\notag\\
&\quad\leq
L\exp\left[ -\epsilon\left( \tfrac{1}{4}d/(d-2)+O(\eta) \right) 
\left( \frac{(\log t)^{1/d} \rho}{\phi} \right)^{d-2} (\log t)^{2/d} \right] \! .
\label{CondProbyl}
\end{align}
By \eqref{RadiusBound}, $(\log t)^{1/d}\rho/\phi \to\infty$, whereas $L=(\log t)^{O(1)}$. 
Hence, the conditional probability in \eqref{CondProbyl} is $o(1)$ and $\condP{F_3(j)}
{F_2(j)}=1-o(1)$.

For $F_4(j)$, write $k=\floor{N/2}-\floor{N/4}$ and $p=\epsilon (\phi/r)^{d-2} (\rho/
\phi)^\alpha$.  \reflemma{HitWhileAvoiding} states that, conditional on $F_2(j)$, each 
ball $x_j+B(\phi z_m,\tfrac{1}{2}\rho)$ has a probability at least $p$ of being hit 
during each of the $k$ excursions from $\boundary B(x_j,r)$ to $\boundary B(x_j,R)$ 
in the definition of $F_4(j)$. It follows that $\condP{F_4(j)}{F_2(j)}$ is at least the 
probability that a Binomial$(k,p)$ random variable has value $M$ or larger. We have 
$p\to 0$ and $k-M\to\infty$ as $t\to\infty$, so using Stirling's approximation, we get
\begin{multline}
\P_{x_0} \! \condparentheses{F_4(j)}{F_2(j)}
\geq \binom{k}{M} p^M(1-p)^{k-M}
= \frac{k^k p^M(1-p)^{k-M}}{M^M (k-M)^{k-M}(\sqrt{2\pi}+o(1))}\\
\geq \left( \frac{kp}{M} \right)^M \frac{(1-p)^k}{\sqrt{2\pi}+o(1)}
= \exp\left[ -M\log(M/kp) -O(kp)-O(1)\right] \! .
\end{multline}

Observe that $kp=e^{O(1)}N_d(t,r,R)(\phi/r)^{d-2}(\rho/\phi)^\alpha=e^{O(1)}(\rho/\phi)^\alpha
\log t$. The assumption $\rho/\phi\to 0$ implies that $kp=o(\log t)$.  On the other hand, 
recall that $M=O((\phi/\rho)^{d-1})$, so that $M/kp=e^{O(1)}(\phi/\rho)^{\alpha+d-1}/\log t$.  
The hypothesis \eqref{RadiusBound} means that $\phi/\rho=o((\log t))^{1/d}$. Consequently, 
$M=o((\log t)^{(d-1)/d})$ and $\log(M/kp)\leq O(\log\log t)$. In particular, $M\log(M/kp)
\leq o(\log t)$, and we conclude that 
\begin{equation}
\label{CondProbzm}
\P_{x_0} \! \condparentheses{F_4(j)}{F_2(j)}=\exp\left( -o(\log t) \right)=t^{o(1)}.
\end{equation}

Combining \eqref{CondProbyl}, \eqref{CondProbzm}, and \refprop{NSuccessfulProb}, we obtain
\begin{align}
\P_{x_0}(F_2(j)\intersect F_3(j)\intersect F_4(j))
&= 
\P_{x_0}(F_2(j)) \P_{x_0}\condparentheses{F_3(j)}{F_2(j)} 
\P_{x_0}\condparentheses{F_4(j)}{F_2(j)}
\notag\\&=
t^{-d\Capa(E_{\rho/\phi})/[(d-2)\kappa_d]+O(\eta)}\,[1-o(1)]\,t^{o(1)}
=t^{-d\Capa(E)/[(d-2)\kappa_d]+O(\eta)}.
\end{align}
We have therefore verified \eqref{F234Bound}, and this completes the proof.
\qed\end{proof}


\subsection{Proof of \refprop{Connected}}
\lbsubsect{ConnectedProof}

\begin{proof}
$\T^d\setminus W[0,t]$ is open since $W[0,t]$ is the (almost surely) continuous image of a 
compact set.

Consider first a Brownian motion $\tilde{W}$ in $\R^d$.  Define
\begin{equation}
\tilde{Z}=\bigunion_{q,q'\in\Q}\bigunion_{1\leq i<j\leq d}
\set{(x_1,\dotsc,x_d)\in\R^d: x_i=q,x_j=q'}
\end{equation}
and note that $\tilde{Z}$ is the inverse image $\pi_0^{-1}(Z)$ of a path-connected, locally 
path-connected, dense subset $Z=\pi_0(\tilde{Z})\subset\T^d$ (where $\pi_0\colon,\R^d\to\T^d$ 
is the canonical projection map).  Since $\tilde{Z}$ is the countable union of 
$(d-2)$-dimensional subspaces, $\tilde{W}\cointerval{0,\infty}$ does not intersect $\tilde{Z}$, 
except perhaps at the starting point, with probability $1$. Projecting onto $\T^d$, it follows 
that $W\cointerval{0,\infty}$ intersects $Z$ in at most one point, and in particular $\T^d
\setminus W\cointerval{0,\infty}$ contains a path-connected, locally path-connected, dense 
subset. This implies the remaining statements in \refprop{Connected}.
\qed\end{proof}


\section{Proofs of Corollaries \ref{c:UnhitSet}--\ref{c:CoverTimeLDP}}
\lbsect{CoroProofs}


\subsection{Proof of \refcoro{UnhitSet}}

\begin{proof}
\eqref{HittingSetProbSimple} follows immediately from the more precise statements in 
\eqref{HittingLargeSetProb}--\eqref{DisjointTranslates}. By monotonicity and continuity, 
it suffices to show \eqref{HittingLargeSetProb} for $\Capa E>\kappa_d$.

Consider first the lower bounds in \eqref{HittingLargeSetProb}--\eqref{DisjointTranslates}.  
Replace $E$ by the compact set $\closure{E}$ (by hypothesis, this does not change the value 
of $\Capa E$). Let $\kappa>\Capa E$ be arbitrary and use 
\refprop{CapacityContinuity}\refitem{CapacityCompact} to find $r>0$ such that $\Capa(E_r)\leq\kappa$. 
Adjoin finitely many lines to $E_r$ to make it into a connected set $E'$ (as in the proof of \refthm{HitManySets}) 
and then adjoin any bounded components of $\R^d\setminus E'$ to form a set $E''\in\cE_c$ that 
satisfies the conditions of \refthm{ShapeOfComponents}. For $\Capa E\geq\kappa_d$, 
\refthm{ShapeOfComponents} implies that $x+\phi_d(t)E\subset\T^d\setminus W[0,t]$ for some 
$x\in\T^d$, with probability at least $t^{J_d(\kappa)-o(1)}$. If instead $\Capa E<\kappa_d$, 
then it is no loss of generality to assume that $\kappa<\kappa_d$ also. Then 
\refthm{ShapeOfComponents} shows that there are at least $t^{J_d(\kappa)-o(1)}$ components 
containing translates $x+\phi_d(t) E$; these translates are necessarily disjoint. In both cases 
we conclude by taking $\kappa\decreasesto\Capa E$.

For the upper bounds, we will shrink the set $E$.  The results nearly follow from 
Theorems~\ref{t:CapacitiesInWc}--\ref{t:ComponentCounts}, since the existence of $x+\phi_d(t) 
E\subset\T^d\setminus W[0,t]$ implies the existence of $x+(\phi_d(t) E)_{-\rho(t)}\subset 
\T^d\setminus W_{\rho(t)}[0,t]$. However, the set $E$ might not be connected. To handle this 
possibility, we will appeal directly to Lemmas~\ref{l:HitFiniteUnionOfLAs} and 
\ref{l:FiniteUnionOfLACounts}.

Let $\kappa\in(\kappa_d,\Capa E)$ (for \eqref{HittingLargeSetProb}) or $\kappa\in(0,\Capa E)$ 
(for \eqref{DisjointTranslates}) be arbitrary.  Apply \refprop{CapacityContinuity}\refitem{CapacityCtyPoint} 
to find an $r>0$ such that $\Capa(E_{-2r})>\kappa$. The enlargement $(E_{-2r})_r$ has a finite number 
$k$ of components, by boundedness. Set $\rho=\rho(t)=r\phi_d(t)$ and choose $n=n(t)$ such that $n(t)
\geq 2\sqrt{d}/\rho(t)$ and the hypotheses of \refprop{AnimalCounts} hold. (As in the proof of 
\refthm{ComponentCounts}, these conditions on $n(t)$ are mutually consistent.)  Apply 
\reflemma{SetsAndAnimals} to each of the $k$ components of $(E_{-2r})_r$ to obtain a set 
$E^\boxempty=\union_{j=1}^k E(A^{(j)})$ satisfying $(E_{-2r})_r\subset E^\boxempty\subset 
(E_{-2r})_{2r}\subset E$. Thus, $\Capa E^\boxempty\geq\kappa$. Furthermore, given $x\in\T^d$ 
there is $x'\in G_{n(t)}$ such that $x'+\phi_d(t)E^\boxempty\subset x+\phi_d(t)((E_{-2r})_{2r})
\subset x+\phi_d(t) E$.  Define $h(t)=C\phi_d(t)$, where $C$ is a constant large enough so 
that $E\subset B(0,C)$.  For $\Capa E>\kappa_d$, we can then apply \reflemma{HitFiniteUnionOfLAs} 
to conclude that $\P(\exists x\in\T^d\colon\, x+\phi_d(t)E\subset\T^d\setminus W[0,t])\leq 
t^{J_d(\kappa)+o(1)}$.  For $\Capa E<\kappa_d$, \reflemma{FiniteUnionOfLACounts} implies that 
$\chi(t,E)\leq\chi_+^\boxempty(t,n(t),\kappa,h(t))\leq t^{J_d(\kappa)+o(1)}$ with high 
probability. In both cases take $\kappa\increasesto\Capa E$.
\qed\end{proof}


\subsection{Proof of Corollaries \ref{c:VolumeLDP}--\ref{c:EvalLDP}}

\begin{proof}
Note the scaling relation
\begin{equation}
\label{EvalScaling}
\lambda(\phi D)=\phi^{-2}\lambda(D).
\end{equation}
Corollaries~\ref{c:VolumeLDP}--\ref{c:EvalLDP} follow from Theorems~\ref{t:CapacitiesInWc}, 
\ref{t:NoTranslates} and \ref{t:ShapeOfComponents} because of the inequalities \eqref{CapaVolEval}. Indeed, apart from the fact that the principal Dirichlet eigenvalue 
$\lambda(E)$ is decreasing in $E$ rather than increasing, the proofs are identical and 
we will prove only \refcoro{EvalLDP}.

Since $\lambda \mapsto I^{\rm Dirichlet}_d(\lambda)$ is continuous and decreasing on 
$\ocinterval{0,\lambda_d}$, it suffices to prove \eqref{LargeEval} and to show that 
$\P(\phi_d(t)^2\lambda(t,\rho(t)) \leq \lambda)=t^{-I_d^{\rm Dirichlet}(\lambda)+o(1)}$ 
for $\lambda<\lambda_d$.

For \eqref{LargeEval}, note that $\T^d\setminus W_{\rho(t)}[0,t]$ cannot contain a ball 
of capacity $>\kappa_d(\lambda_d/\lambda(t,\rho(t)))^{(d-2)/2}$: by \eqref{CapacityScaling} 
and \eqref{EvalScaling}, the component of $\T^d\setminus W_{\rho(t)}[0,t]$ containing 
such a ball would have an eigenvalue strictly smaller than $\lambda(t,\rho(t))$. In 
particular, if $\lambda>\lambda_d$ and $\lambda(t,\rho(t)) \geq \lambda \phi_d(t)^{-2}$, 
then $\T^d\setminus W_{\rho(t)}[0,t]$ cannot contain a ball of capacity $\kappa_d\,
\phi_d(t)^{d-2}((\lambda_d/\lambda)^{(d-2)/2}+\delta)$ for any $\delta>0$. Taking 
$\delta$ small enough so that $(\lambda_d/\lambda)^{(d-2)/2}$ $+\delta<1$, applying 
\refthm{NoTranslates} with $E$ the ball of capacity $\kappa_d((\lambda_d/\lambda)^{(d-2)/2}
+\delta)$, and letting $\delta\decreasesto 0$, we obtain \eqref{LargeEval}.

Now take $\lambda<\lambda_d$. By \refprop{NoWrapping}, apart from an event of negligible 
probability, every component $C$ of $\T^d\setminus W_{\rho(t)}[0,t]$ can be isometrically 
identified (under its intrinsic metric) with a bounded open subset $E$ of $\R^d$, via 
$C=x+E$ for some $x\in\T^d$. In particular, $\lambda(C)=\lambda(E)$, and we can apply 
\eqref{CapaVolEval} to conclude that $\kappa^*(t,\rho(t))\geq\Capa E\geq \kappa_d(\lambda_d/
\lambda(C))^{(d-2)/2}$. Applying \refthm{CapacitiesInWc}, 
\begin{equation}
\P(\phi_d(t)^2\lambda(t,\rho(t))\leq\lambda)
\leq 
\P(\kappa^*(t,\rho(t))\geq \kappa_d (\lambda_d/\lambda)^{(d-2)/2}\phi_d(t)^{d-2})
\leq t^{-I_d^{\rm Dirichlet}(\lambda)+o(1)}.
\end{equation}

For the reverse inequality, note that \refthm{ShapeOfComponents} implies that $\T^d\setminus 
W_{\rho(t)}[0,t]$ contains a ball of capacity $\kappa_d\,\phi_d(t)^{d-2}(\lambda_d/
\lambda)^{(d-2)/2}$ with probability at least $t^{-I_d^{\rm Dirichlet}(\lambda)-o(1)}$.
\qed\end{proof}


\subsection{Proof of \refcoro{InradiusLDP}}

\begin{proof}
Since $r\mapsto I_d^{\rm inradius}(r)$ is continuous and strictly increasing on 
$\cointerval{1,\infty}$ and is infinite elsewhere, it suffices to verify 
\eqref{SmallInradius} and show $\P(\rho_{\rm in}(t)> r\phi_d(t))= t^{-I_d^{\rm inradius}(r)
+o(1)}$ for $r\geq 1$. But the events $\set{\rho_{\rm in}(t)\leq r\phi_d(t)}$ and 
$\set{\rho_{\rm in}(t)> r\phi_d(t)}$ are precisely the event 
\begin{equation}
\set{(x+\phi_d(t) B(0,r))\intersect W[0,t]\neq \emptyset \; \forall x\in\T^d}
\end{equation} 
and its complement
\begin{equation}
\set{\exists x\in\T^d\colon\,(x+\phi_d(t) B(0,r))\intersect W[0,t]=\emptyset}
\end{equation} 
from \refthm{NoTranslates}, with $\rho(t)=0$, and equation \eqref{HittingLargeSetProb} from \refcoro{UnhitSet}, 
with $E=B(0,r)$.
\qed\end{proof}


\subsection{Proof of \refcoro{CoverTimeLDP}}

\begin{proof}
Recall that $\set{\rho_{\rm in}(t) > \epsilon} = \set{\cC_\epsilon > t}$, so that setting 
$t=u\psi_d(\epsilon)$, $r=\epsilon/\phi_d(u\psi_d(\epsilon))$ rewrites the event 
$\set{\cC_\epsilon>u\psi_d(\epsilon)}$ as $\set{\rho_{\rm in}(t)>r\phi_d(t)}$. By 
\eqref{phipsiAsymptotics}, $r\to (u/d)^{1/(d-2)}$ as $\epsilon\decreasesto 0$. It follows that 
$\P(\cC_\epsilon>u\psi_d(\epsilon))=t^{-I_d^{\rm inradius}((u/d)^{1/(d-2)})+o(1)}$ for $u>d$, 
since $r\mapsto I_d^{\rm inradius}(r)$ is continuous on $(1,\infty)$.  Noting that 
$t=\epsilon^{-(d-2)+o(1)}$, this last expression is $\epsilon^{I_d^{\rm cover}(u)+o(1)}$. 
A similar argument proves \eqref{SmallCoverTime}. Because $u\mapsto I_d^{\rm cover}(u)$ is 
continuous and strictly increasing on $\cointerval{d,\infty}$ and $I_d^{\rm cover}(v)=\infty$ 
otherwise, these two facts complete the proof.
\qed\end{proof}

\appendix

\section{Hitting probabilities for excursions}
\lbappendix{ExcursionProofs}


\subsection{Unconditioned excursions: proof of \reflemma{CapacityAndHittingDistant}}
\lbsubsect{CapaHittingDistantProof}

\begin{proof}
Since $R<\tfrac{1}{2}$, we may consider $x,\xi',\xi,W(t)$ to have values in $\R^d$ instead 
of $\T^d$. Furthermore, w.l.o.g.\ we may assume that $x=0$.

We first remove the effect of conditioning on the exit point $\xi\in\boundary B(0,R)$. Define 
$T=\sup\set{t<\zeta\colon\,d(0,W(t))\leq r}$ to be the last exit time from $B(0,r)$ before 
time $\zeta$; note that $E\intersect W[0,\zeta_R]=E\intersect W[0,T]$. Let $\tilde{r}\in(r,R)$ 
and define $\tilde{\tau}=\inf\set{t>T\colon\,d(0,W(t))=\tilde{r}}$ to be the first hitting 
time of $\boundary B(0,\tilde{r})$ after time $T$.

Under $\P_{\xi'}$ (i.e., without conditioning on the exit point $W(\zeta_R)$) we can express 
$(W(t))_{0\leq t\leq\zeta_R}$ as the initial segment $(W(t))_{0\leq t\leq\tilde{\tau}}$ 
followed by a Brownian motion, conditionally independent given $W(\tilde{\tau})$, started 
at $\tilde{\xi}=W(\tilde{\tau})$ and conditioned to exit $B(0,R)$ before hitting $B(0,r)$. 
Let $f_{\tilde{r},R}(\tilde{\xi},\cdot)$ denote the density, with respect to the uniform 
measure $\sigma_R$ on $\boundary B(0,R)$, of the first hitting point $W(\zeta_R)$ on 
$\boundary B(0,R)$ for this conditioned Brownian motion. Then for $S\subset\boundary B(0,R)$ 
measurable, we have
\begin{equation}
\label{PxiprimetautildeFormula}
\P_{\xi'} \! \left( E\intersect W[0,\zeta_R]\neq\emptyset, W(\zeta_R)\in S \right)
=\E_{\xi'} \! \left( \indicator{E\intersect W[0,T]\neq\emptyset} 
\int_S f_{\tilde{r},R}(W(\tilde{\tau}),\xi) d\sigma_R(\xi) \right).
\end{equation}
From \eqref{PxiprimetautildeFormula} it follows that the conditioned measure $\P_{\xi',\xi}$ 
satisfies 
\begin{equation}
\label{PxiprimexiFormula}
\P_{\xi',\xi}(E\intersect W[0,\zeta_R]\neq\emptyset)
= \frac{\E_{\xi'} \! \left( \indicator{E\intersect W[0,T]\neq\emptyset} 
f_{\tilde{r},R}(W({\tilde{\tau}}),\xi) \right)}{\E_{\xi'} \!
\left( f_{\tilde{r},R}(W({\tilde{\tau}}),\xi) \right)}.
\end{equation}
(More precisely, we would conclude \eqref{PxiprimexiFormula} for $\sigma_R$-a.e.\ $\xi$, 
but by a continuity argument we can take \eqref{PxiprimexiFormula} to hold for all $\xi$.)

Now choose $\tilde{r}$ in such a way that $R/\tilde{r}\to\infty, \tilde{r}/r\to\infty$, 
for instance, $\tilde{r}=\sqrt{rR}$. Since $R/\tilde{r}\to\infty$, we have $f_{\tilde{r},R}
(\tilde{\xi},\xi)=1+o(1)$, uniformly in $\tilde{\xi},\xi$. Therefore
\begin{align}
\label{DifferenceOfHitting}
\P_{\xi',\xi}(E\intersect W[0,\zeta_R]\neq\emptyset)
&= [1+o(1)]\,\P_{\xi'}(E\intersect W[0,\zeta_R]\neq\emptyset)
\notag\\
&= 
[1+o(1)]\left( \big. \P_{\xi'}(E\intersect W\cointerval{0,\infty}\neq\emptyset)
-\P_{\xi'}(E\intersect W\cointerval{\zeta_R,\infty}\neq\emptyset) \right) \! .
\end{align}
By the Markov property, the last term in \eqref{DifferenceOfHitting} is the probability of 
hitting $E$ when starting from some point $W(\zeta_R)\in\boundary B(0,R)$ (averaged over the 
value of $W(\zeta_R)$). Since $R/r\to\infty$, this will be shown to be an error term, and 
the proof will have been completed, once we show that
\begin{equation}
\label{HittingFromPoint}
\P_{\xi'} \! \left( W\cointerval{0,\infty}\intersect E\neq\emptyset \right)
=\frac{\Capa E }{\kappa_d \, r^{d-2}}\,[1+o(1)] \qquad\text{as }r/\epsilon\to\infty.
\end{equation}
Note that \eqref{HittingFromPoint} is essentially the limit in \eqref{CapacityAndHittingSimple}, 
taken uniformly over the choice of $E\subset B(0,\epsilon)$.

To show \eqref{HittingFromPoint}, let $g_\epsilon(\xi',\cdot)$ denote the density, with 
respect to the uniform measure $\sigma_\epsilon$ on $\boundary B(0,\epsilon)$, of the 
first hitting point of $\boundary B(0,\epsilon)$ for a Brownian motion started at $\xi'$ 
and conditioned to hit $B(0,\epsilon)$. Then
\begin{equation}
\label{HittingFromPointIntegral}
\P_{\xi'} \! \left( W\cointerval{0,\infty}\intersect E\neq\emptyset \right)
= \frac{\epsilon^{d-2}}{r^{d-2}}\int_{\boundary B(0,\epsilon)} 
\P_y \! \left( W\cointerval{0,\infty}\intersect E\neq\emptyset \right) 
g_\epsilon(\xi',y) d\sigma_\epsilon(y).
\end{equation}
Since $r/\epsilon\to\infty$, we have $g_\epsilon(\xi',y)\to 1$ uniformly in $\xi',y$. 
Hence \eqref{HittingFromPoint} follows from \eqref{HittingFromPointIntegral} and 
\eqref{CapacityAndHittingUniform}.
\qed\end{proof}


\subsection{Excursions avoiding an obstacle: proof of \reflemma{HitWhileAvoiding}}
\lbsubsect{HitWhileAvoidingProof}

\begin{proof}
It suffices to bound from below
\begin{equation}
\label{UnconditionalHitWhileAvoiding}
\P_{\xi',\xi} \! \left( W[0,\zeta_R]\text{ intersects $x+B(\phi y,\tfrac{1}{2}\rho)$ 
but not }x+(\phi E)_\rho \right),
\end{equation}
since conditioning on $(x+(\phi E)_\rho)\intersect W[0,\zeta_R]=\emptyset$ can only 
increase the probability in \eqref{UnconditionalHitWhileAvoiding}. Moreover, as in the 
proof of \reflemma{CapacityAndHittingDistant}, we may replace $\P_{\xi',\xi}$ by $\P_{\xi'}$, 
using now that the densities $f_{\tilde{r},R}$ and $g_\epsilon$ are bounded away from $0$ 
and $\infty$ when $r\leq\tfrac{1}{2}R$.

Fix $E\in\cE_c^\boxempty$, so that $E=\tfrac{1}{n}A$ for some $A\in\cA^\boxempty\intersect
\cE_c$ and $n\in\N$, and fix $\delta>0$ (we may assume that $\delta<1/(2n)$). By 
assumption, $E$ is bounded, say $E\subset B(0,a)$. By adjusting $\epsilon$, we may assume 
that $\rho/\phi<a$ (so that $(\phi E)_\rho\subset B(0,2a\phi)$) and $r>4a\phi$. We 
distinguish between three cases:

\medskip\noindent
$\bullet$
$y\in B(0,3a) \setminus E_\delta$. Consider $w\in\Z^d\setminus A$. Because 
$A\in\cE_c$, there is a finite path of open cubes with centres $w_0,w_1,\dots,w_k\in\Z^d$ 
such that $w_0\in\Z^d\setminus B(0,3an)$, $w_k=w$, $d(w_{j-1},w_j)=1$ and $\interior{
\union_{j=0}^k (w_j+[-\tfrac{1}{2},\tfrac{1}{2}]^d)}\intersect A=\emptyset$. By compactness, 
the length $k$ of such paths may be taken to be uniformly bounded. Hence, if $\rho/\phi
<\delta/2$, then, given $\xi''\in\boundary B(x,3a\phi)$, there is a path 
$\Gamma\subset B(x,3a\phi)$ from $\xi''$ to $x+\phi y$ consisting of a finite number of 
line segments, each of length at most $\phi$, such that $\Gamma_{\delta\phi/2}\intersect
(x+(\phi E)_\rho)=\Gamma_{\delta\phi/2} \intersect (x+\phi (E_{\rho/\phi}))=\emptyset$.  
Moreover, the number of line segments can be taken to be bounded uniformly in $y$ and 
$\xi''$. In fact, $\Gamma$ can be chosen as the union of line segments between points 
$x+\phi w_0/n,\dotsc, x+\phi w_k/n$ as above, together with a bounded number of line 
segments to join $\xi''$ to $x+\phi w_0/n$ in $B(x,3a\phi) \setminus B(x,2a\phi)$ and 
to join $x+\phi w_k/n$ to $x+\phi y$ in the cube $x+(\phi/n)(w+[-\tfrac{1}{2},
\tfrac{1}{2}]^d)$ containing $y$ (see \reffig{PathGamma})

\begin{figure}[htbp]
\vspace{0.3cm}
\begin{center}
\includegraphics{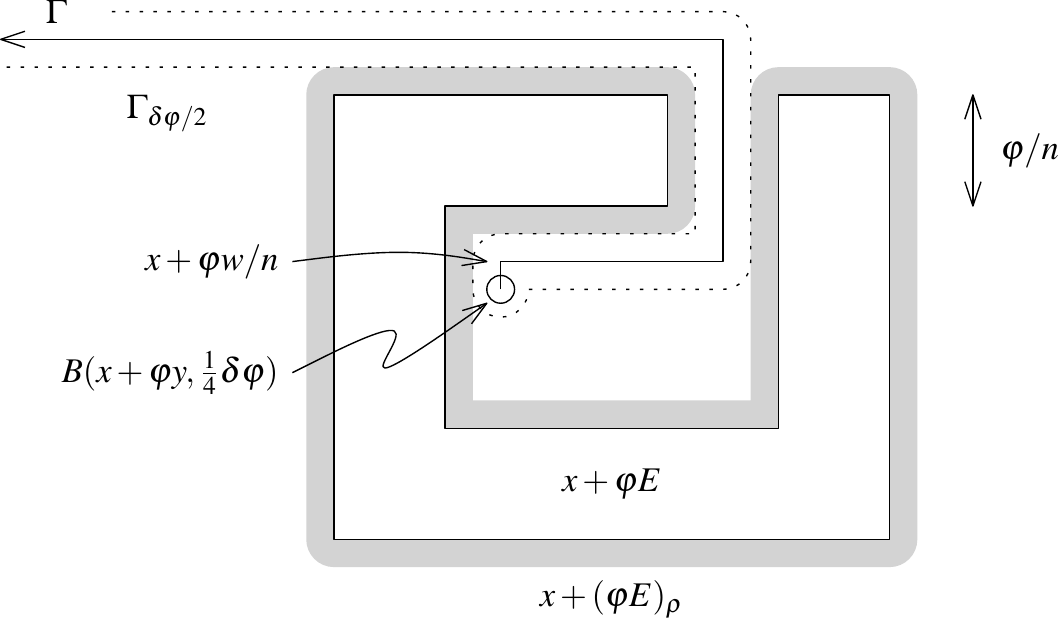}
\caption{The path $\Gamma$ reaching $x+\phi y$. The sets $\Gamma_{\delta\phi/2}$ and 
$x+(\phi E)_\rho=x+\phi(E_{\rho/\phi})$ are depicted for the worst-case scenario where 
the parameters $\rho/\phi<\delta/2<1/4n$ are equal.}
\lbfig{PathGamma}
\end{center}
\end{figure}

From $\xi'\in\boundary B(x,r)$, the Brownian path reaches $\boundary B(x,3a\phi)$ before 
$\boundary B(x,R)$ with probability $(r^{-(d-2)}-R^{-(d-2)})/((3a\phi)^{-(d-2)}-R^{-(d-2)})$. 
By our assumptions, this is at least $c_1 (\phi/r)^{d-2}$ for some $c_1>0$. Uniformly in 
the first hitting point $\xi''$ of $\boundary B(x,3a\phi)$, there is a positive probability 
of hitting $\boundary B(x+\phi y,\tfrac{1}{4}\delta\phi)$ via $\Gamma_{\delta\phi/2}$ before 
hitting $\boundary B(x,4a\phi)$. The probability of next hitting $\boundary B(x+\phi y,
\tfrac{1}{2}\rho)$ before $\boundary B(x+\phi y,\tfrac{1}{2}\delta\phi)$ is 
\begin{equation}
[(\tfrac{1}{4}\delta\phi)^{-(d-2)}-(\tfrac{1}{2}\delta\phi)^{-(d-2)}]
/[(\tfrac{1}{2}\rho)^{-(d-2)}-(\tfrac{1}{2}\delta\phi)^{-(d-2)}],
\end{equation} 
which is at least $c_2(\rho/\phi)^{d-2}$ for some $c_2>0$.  Thereafter there is a positive 
probability of returning to $\boundary B(x,r)$ without hitting $x+(\phi E)_\rho$, via 
$\Gamma_{\delta\phi/2}$. Combining these probabilities gives the required bound.

\medskip\noindent
$\bullet$
$y\in E_\delta\setminus E_{\rho/\phi}$. We have $y\in\tfrac{1}{n}(w+[-\tfrac{1}{2},
\tfrac{1}{2}]^d)$ for some $w\in\Z^d$. Write $C_\theta(y,\tfrac{1}{n}w)$ for the cone with 
vertex $y$, central angle $\theta$, and axis the ray from $y$ to $\tfrac{1}{n}w$. We can 
choose the angle $\theta$ and a constant $c_3>0$ small enough (in a manner depending only 
on $d$) so that $C_\theta(y,\tfrac{1}{n}w) \intersect E_{\rho/\phi} \intersect B(y,(1+c_3)
d(y,w))=\emptyset$. With $\theta$ and $c_3$ fixed, we can choose $c_4>0$ so that every point 
of $B(\tfrac{1}{n}w,c_4)$ is a distance at least $c_5>0$ from $\boundary C_\theta(y,
\tfrac{1}{n}w)$ and $\boundary B(y,(1+c_3)d(y,\tfrac{1}{n}w))$ (see \reffig{ConeInCube}).

\begin{figure}[htbp]
{\hfill\includegraphics{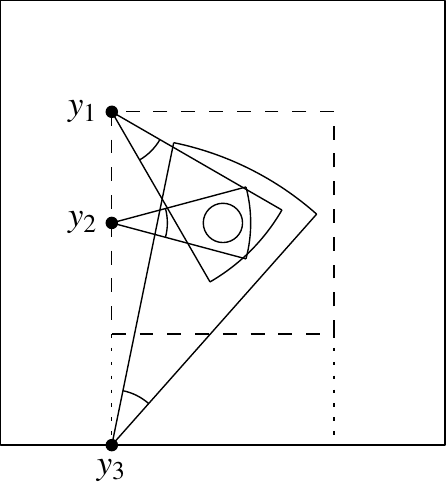}\hfill\includegraphics{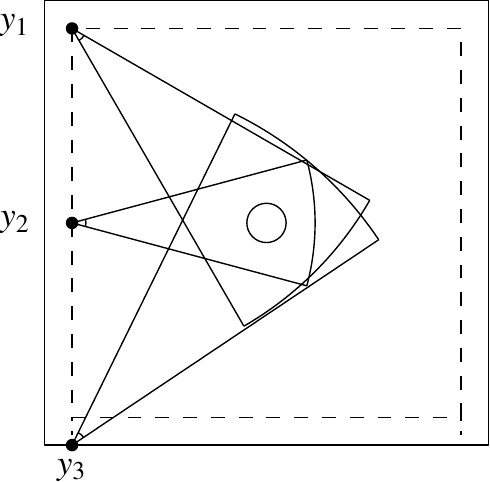}\hfill}
\caption{Cones $C_\theta(y,\tfrac{1}{n}w)$ and parts of balls $B(y,\rho/(2\phi))$ and 
$B(y,(1+c_3)d(y,\tfrac{1}{n}w))$ for three choices of $y$.  The outer square is 
the cube $\tfrac{1}{n}(w+[-\tfrac{1}{2},\tfrac{1}{2}]^d)$ containing $y$. The dashed 
line shows the greatest possible extent of $E_{\rho/\phi}$. At least one face of the 
cube is not contained in $E$, resulting in a conduit to the outside of the cube (dotted 
lines). The ball $B(\tfrac{1}{n}w,c_4)$ in the centre is uniformly bounded away from 
the sides of the cones and from the other balls. On the left the parameters $\rho/\phi
<1/4n$ are depicted as equal. On the right is the more relevant situation $\rho/\phi 
\ll 1/(4n)$.}
\lbfig{ConeInCube}
\end{figure}

Under these conditions, there is a probability at least $c_6(\rho/\phi)^\alpha$ for a 
Brownian path started from a point of $\boundary B(x+\phi w/n,c_4 \phi)$ to reach 
$\boundary B(x+\phi y,\tfrac{1}{2}\rho)$ before hitting $\boundary B(x+\phi y,\phi (1+c_3)
d(y,w))\union \boundary(x+\phi C_\theta(y,w))$, and then to reach $\boundary B(x+\phi y,
\phi d(y,w))$ before hitting $\boundary(x+\phi C_\theta(y,w))$.\footnote{This follows from 
hitting estimates for Brownian motion in a cone. For instance, via the notation of 
Burkholder~\cite[pp.\ 192--193]{Burk1977}, the harmonic functions on $C(0,z_0)$ given 
by $u_1(z)=r_0^{p+d-2}(\abs{z}^{-(p+d-2)}-\abs{z}^p) h(\vartheta)$ and $u_2(z)=\abs{z}^p 
h(\vartheta)$ (with $\vartheta$ the angle between $z$ and $z_0$ and the value $p>0$ chosen 
so that $u_1(z)=u_2(z)=0$ on $\boundary C(0,z_0)$) are lower bounds for the probabilities, 
starting from $z\in C(0,z_0)$, of hitting $\boundary B(0,r_0)$ before $\boundary B(0,1)
\union\boundary C(0,z_0)$ and of hitting $\boundary B(0,1)$ before $\boundary C(0,z_0)$, 
respectively.} The rest of the proof proceeds as in the previous case.

\medskip\noindent
$\bullet$
$y\in B(0,r/\phi)\setminus B(0,3a)$. Let $b=d(0,y)\in[3a,r/\phi]$. The 
probability that a Brownian path started from $\xi'$ first hits $\boundary B(x,b\phi)$ 
without hitting $\boundary B(x,R)$, then hits $\boundary B(x+\phi y,\tfrac{1}{12}b\phi)$ 
without hitting $\boundary B(x,\tfrac{2}{3}b\phi)$, then hits $\boundary B(x+\phi y,
\tfrac{1}{2}\rho)$ before hitting $\boundary B(x+\phi y,\tfrac{1}{6}b\phi)$, and
finally exits $B(x,R)$ without hitting $\boundary B(x,\tfrac{2}{3}b\phi)$, is at least 
$[c_7 (b\phi/r)^{d-2}] [c_8][c_9(\rho/(b\phi))^{d-2}][c_{10}]$. Since $x+(\phi E)_\rho
\subset B(x,2a\phi)\subset B(x,\tfrac{2}{3}b\phi)$, this is the required bound.
\qed\end{proof}

\begin{acknowledgements}
The research of the authors was supported by the European
Research Council through ERC Advanced Grant 267356 VARIS. The authors are grateful 
to M.\ van den Berg for helpful input.
\end{acknowledgements}


\bibliographystyle{spmpsci}
\bibliography{PMreferences}

\end{document}